\documentclass[12pt,reqno]{amsart}

\usepackage{cite}
\usepackage{xcolor}
\usepackage{tikz-cd}
\usepackage{float}
\usepackage{bm}
\usepackage{mathtools}
\usepackage{graphicx}
\usepackage{enumerate}

\usepackage{amsfonts,amssymb,verbatim,amsmath,amsthm,latexsym,textcomp,amscd,hyperref}
\usepackage{latexsym,amsfonts,amssymb,epsfig,verbatim}
\usepackage{graphicx}
\usepackage{color}
\usepackage{url}
\usepackage{pgfplots}
\usepackage{tikz}
\usepackage{pdfpages}
\usepackage{subcaption}

\graphicspath{ {./images/} }
\newtheorem{theorem}{Theorem}[section]
\newtheorem*{theorem*}{Theorem}
\newtheorem*{lemma*}{Lemma}
\newtheorem*{proposition*}{Proposition}
\newtheorem*{corollary*}{Corollary}
\newtheorem{lemma}[theorem]{Lemma}
\newtheorem{defn}[theorem]{Definition}
\newtheorem{prop}[theorem]{Proposition}
\newtheorem{cor}[theorem]{Corollary}
\newtheorem{example}[theorem]{Example}
\theoremstyle{definition}
\newtheorem{remark}[theorem]{Remark}

\newtheorem{question}[theorem]{Question}

\def\dl{\delta}

\def\ep{\epsilon}

\def\lm{\lambda}

\def\ri{\rightarrow}

\def\sse{\subseteq}
\def\gm{\gamma}
\def\bt{\beta}
\def\al{\alpha}

\def\pa{\partial}
\def\map{\rightarrow}

\newcommand\LMY{\lim^Y_{n \map \infty}}
\newcommand\LMX{\lim^X_{n \map \infty}}

\newcommand\RED{\textcolor{red}}

\def\L{\mathcal{L}}

\def\F{\mathcal{F}}

\def\Y{\mathcal{Y}}
\def\G{\mathcal{G}}

\def\R{\mathbb {R}}

\def\N{\mathbb {N}}

\def\Z{\mathbb {Z}}

\title[Surjectivity of the CT map in metric (graph) bundles]{Surjectivity of the Cannon--Thurston map in metric (graph) bundles}
\author{Rakesh Halder}
\address{Tata Institute of Fundamental Research (TIFR), Mumbai, India}
\email{rhalder.math@gmail.com}
\subjclass[2020]{20F65, 20F67}
\keywords{Hyperbolic metric spaces (groups), Cannon--Thurston map, Metric (graph) bundles}
\begin{document}
	
	\maketitle
	
	\begin{abstract}
		
{\em Metric (graph) bundles} generalize the notion of {\em fiber bundles} to the context of geometric group theory and were introduced by Mj and Sardar in \cite{pranab-mahan}. Suppose $X$ is a metric (graph) bundle over $B$ such that the fibers are ({\em uniformly}) hyperbolic, and the total space $X$ is also hyperbolic. In this generality, Mj--Sardar proved that the inclusion of a fiber into $X$ admits a continuous extension to the (Gromov) boundary.

\noindent In this article, we prove that such a continuous extension map between boundaries is {\em surjective} in the following two key settings.

\begin{enumerate}
\item The fibers are uniformly quasiisometric to a nonelementary hyperbolic group.

\item The fibers are one-ended hyperbolic metric spaces.
\end{enumerate}

\noindent Our result generalizes a theorem of Bowditch in which the fibers were assumed to be the hyperbolic plane, and it answers a question posed by Lazarovich, Margolis and Mj in \cite{NirMaMj-commen}.
	\end{abstract}
	
	\tableofcontents
	
	\section{Introduction}\label{intro}
{\em Metric (graph) bundles} generalize the notion of {\em fiber bundles} to the context of geometric group theory and were introduced by Mj and Sardar (\cite{pranab-mahan}). At the same time, metric (graph) bundles generalize the existing coarse-geometric notion of {\em trees of metric spaces} a la Bestvina--Feighn (\cite{BF,mitra-trees}) in the special case where the inclusions of the edge spaces into the adjacent vertex spaces are uniformly coarsely surjective (quasiisometries). Suppose $X$ is a metric (graph) bundle over $B$ such that the fibers are ({\em uniformly}) hyperbolic, and the total space $X$ is also hyperbolic. In this generality, Mj--Sardar proved that the inclusion of a fiber into $X$ admits a continuous extension to the (Gromov) boundary.\smallskip

\noindent{\em In this article, we prove that such a continuous extension map between boundaries is {\em surjective} in the following two key settings (Corollary \ref{cor-main gp over ray new intro}, Theorem \ref{thm-combo1}).

\begin{enumerate}
	\item The fibers are uniformly quasiisometric to a fixed nonelementary hyperbolic group.
	
	\item The fibers are one-ended hyperbolic metric spaces.
\end{enumerate}}
	
Suppose $H<G$ are (Gromov) hyperbolic groups. A question posed in \cite[p. $136$]{mitra-ct}, \cite[p. $527$]{mitra-trees} is the following: {\em does the inclusion $i:H\map G$ extend to a continuous map $\pa i:\pa H\map\pa G$}? See also \cite[Question $1.19$]{bestvinaprob}. Such a continuous extension, if it exists, is known as a {\em Cannon--Thurston map} (or {\em CT map} for short) \cite{mitra-ct, mitra-trees} after the pioneering work of Cannon and Thurston (\cite{CT,CTpub}). 
Their result proves the following. Suppose $M$ is a closed hyperbolic $3$-manifold fibering over a circle with fiber a closed orientable hyperbolic surface $S$. Then the inclusion $\pi_1(S)\map\pi_1(M)$ extends continuously to the boundary (and the extension map between boundaries is surjective). More generally, one may ask the same question in the context of (Gromov) hyperbolic metric spaces $Y\sse X$ (see Definition \ref{CT-map}). 

\noindent Generalizing the foundational work of Cannon and Thurston, some positive answers to the above question were given in \cite{mitra-ct, mitra-trees}. Since then, numerous significant results concerning the existence and structure of CT maps have been obtained, particularly in the setting of hyperbolic subgroups of hyperbolic groups -- see, for example, \cite{ps-krishna, ps-kap, mitra-endlam, kl15, dkt, JKLO, baker-riley-hydra, ps-rakesh-ct}. However, the general question posed by Mahan Mitra regarding the existence of CT maps for arbitrary hyperbolic groups was answered negatively in \cite{baker-riley}. Further examples illustrating the nonexistence of CT maps in the hyperbolic group context can be found in \cite{mats-oguni} and \cite{HMS-landing}. For a comprehensive survey of CT maps, we refer the reader to \cite{icm-mj}.
	
In \cite{bowditch-stacks}, Bowditch studied metric (graph) bundles in a special case, namely when the base is either $\mathbb{R}$ or $\mathbb{R}_{\ge 0}$, but using a different terminology. He referred to such structures as {\em bi-infinite stacks} and {\em semi-infinite stacks}, respectively. Through this framework, Bowditch established several results concerning Teichm\"uller geodesics and ending laminations in the setting of hyperbolic surfaces. The notion of {\em metric fibrations} were introduced in \cite{farb-mosher} and that can be thought of as metric bundles \cite{pranab-mahan}. In \cite{hamenst-word}, the author used metric fibrations and proved a combination theorem regarding the hyperbolicity of extensions of surface groups. 

Metric graph bundles (respectively, metric bundles) are formally defined in Definition \ref{defn-metric graph bundle} (respectively, Definition \ref{defn-metric bundle}). Let $\dl\ge0$, $L\ge0$. Suppose $\pi:X\map B$ is a metric (graph) bundle such that:

\begin{itemize}
\item the fibers $F_b:=\pi^{-1}(b),~b\in B$ are $\dl$-hyperbolic,

\item the barycenter maps $\pa^3F_b\map F_b,~b\in B$ are $L$-coarsely surjective, i.e., $L$-neighborhood of the image is $F_b$, and

\item the total space $X$ is $\dl$-hyperbolic.
\end{itemize}
Then for any $b\in B$, the inclusion $F_b\map X$ admits a CT map $\pa F_b\map\pa X$ (\cite[Theorem $5.3$]{pranab-mahan}). In this paper, we address the following question.
	
\begin{question}\label{qsn-surj in metric bundle}
Is the above CT map $\pa F_b\map\pa X$ surjective?
\end{question}	

{\em This paper aims to give a positive answer to Question \ref{qsn-surj in metric bundle}.} One of the consequences of our more general results give an affirmative answer to Question \ref{qsn-surj in metric bundle}  as follows. 


\begin{cor}[see Theorem \ref{thm-application}]\label{cor-main gp over ray new intro}
Let $k\ge1$, $D\ge1$. Suppose $\pi:X\map B$ is a metric graph bundle such that:

\begin{enumerate}
\item the fibers $F_b:=\pi^{-1}(b),~b\in V(B)$ are $k$-quasiisometric to a fixed nonelementary hyperbolic group, and

\item for any vertex $b\in V(B)$, the valence at $u$ in $F_b$ is bounded above by $D$, where $u$ is a vertex of $F_b$.
\end{enumerate}

Further, assume that $X$ is hyperbolic. Then for any $b\in V(B)$, the CT map $\pa F_b\to\pa_s X$ is surjective. 
\end{cor}

Question \ref{qsn-surj in metric bundle} has its origins in two key settings, outlined below (see Subsection \ref{exp sec} for additional sources of examples). Suppose 
$1 \xrightarrow{} H \xrightarrow{} G \xrightarrow{p} Q \xrightarrow{}1$ 
is a short exact sequence of hyperbolic groups with $H$ nonelementary. Let $S_H\sse S_G$ be a finite generating sets for $H$ and $G$ respectively, and let $S_Q=\{p(x):x\in S_G\}\setminus\{1\}$. This gives rise to a metric graph bundle $\pi: X \to B$, where $X$, $B$, and the fibers are the Cayley graphs of $G$, $Q$, and $H$, respectively, with respect to $S_G,~S_Q$ and $S_H$ (see \cite[Example $1.8$]{pranab-mahan}). In this setting, note that the CT map, as in Question \ref{qsn-surj in metric bundle} (see \cite{mitra-ct}), is surjective. Indeed, it is enough to show that the limit set $\Lambda_G(H)$ is equal to $\pa G$ (see Lemma \ref{CT image is limit set}). By \cite[Theorem 5.1]{coornaert-mea}, $\pa G$ is minimal, closed, $G$-invariant set (in $\pa G$). On the other hand, for any $g\in G$, $g.\Lambda_G(H)=\Lambda_G(gH)=\Lambda_G(gHg^{-1})=\Lambda_G(H)$ (for the first two equalities, see \cite[Lemma 2.9]{ps-conical}, for instance). Hence $\Lambda_G(H)=\pa G$ as $\Lambda_G(H)$ is closed in $\pa G$. Examples of this sort include the celebrated work of Cannon and Thurston \cite{CT}.

\noindent As mentioned above, Bowditch, in \cite{bowditch-stacks}, studied {\em bi-infinite stacks} and {\em semi-infinite stacks} over $\R$ and $\R_{\ge0}$ respectively. 
However, following \cite{pranab-mahan}, we will simply refer to such spaces as metric bundles (Definition \ref{defn-metric bundle}). Bowditch gave a positive answer to Question~\ref{qsn-surj in metric bundle} in the case where the fibers are hyperbolic planes as follows.
\begin{theorem}\textup{(\cite[Theorem 2.6.1]{bowditch-stacks})}\label{bowditch thm}
Suppose $\pi:X\map[0,\infty)$ is a metric bundle such that the fibers are isometric to the hyperbolic plane $\mathbb H^2$. Further, suppose that $X$ is hyperbolic. Then the inclusion $\pi^{-1}(0)\map X$ admits a surjective Cannon--Thurston map.
\end{theorem}


	
In the proof of Theorem \ref{bowditch thm}, apart from the hyperbolicity of $\mathbb H^2$, one crucial condition used is that $\mathbb H^2$ is one-ended. 
It is not hard to see that the barycenter map for the hyperbolic plane $\mathbb H^2$ is coarsely surjective. 
However, one can easily construct examples of hyperbolic metric bundles whose fibers do not satisfy the coarse surjectivity of the barycenter map, and for which the CT map is not surjective (for instance, see \cite[Example $3.14$]{NirMaMj-commen}).
	Motivated by such examples and Theorem \ref{bowditch thm}, Lazarovich, Margolis and Mj posed the following question. 
	
	\begin{question}\textup{(\cite[Question 3.15]{NirMaMj-commen})}\label{main qsn}
Let $k\ge1$, $L\ge0$. Suppose $\pi:X\map [0,\infty)$ is a metric graph bundle  such that: $(1)$ the fibers $\pi^{-1}(i),~i\in\N\cup\{0\}$ are $k$-quasiisometric to a fixed one-ended hyperbolic space, and $(2)$ the barycenter maps $\pa^3(\pi^{-1}(i))\map \pi^{-1}(i)$ are $L$-coarsely surjective. Further, assume that $X$ is hyperbolic.

Is the CT map $\pa\pi^{-1}(0)\map\pa X$ surjective?
	\end{question}
	
Note that the existence of a CT map in Question \ref{main qsn} follows from a more general result of Mitra (\cite{mitra-trees}). 
In this article, we give an affirmative answer to Question \ref{main qsn}. In fact, we prove this result when the fibers are not necessarily one-ended as follows. 
	

\begin{theorem}[see Theorem \ref{thm-application}]\label{thm-combo1}
Let $\dl\ge0$, $L\ge0$ and $D\ge1$. Suppose $\pi:X\map B$ is a metric graph bundle such that 

\begin{itemize}
	\item the fibers $F_b:=\pi^{-1}(b),~b\in V(B)$ are $\dl$-hyperbolic,
	
	\item the barycenter maps $\pa^3F_b\map F_b,~b\in V(B)$ are $L$-coarsely surjective, and 
	
	\item the total space $X$ is $\dl$-hyperbolic.
\end{itemize}
Further, we assume one of the following.

$(A)$ Suppose that for each fiber $F_b,~b\in V(B)$ and for each vertex $u\in F_b$, the valence at $u$ in $F_b$ is bounded above by $D$.

$(B)$ Suppose that the fibers $F_b,~b\in V(B)$ are one-ended and proper metric spaces.\smallskip

Finally, let $A$ be a qi embedded subgraph of $B$, and let $Y=\pi^{-1}(A)$. Then the CT map $\pa_s Y\map\pa_s X$ is surjective.
\end{theorem}

For convenience, we will refer to Theorem~\ref{thm-combo1} under the assumption $(A)$ as Theorem~\ref{thm-combo1} $(A)$, and under the assumption $(B)$ as Theorem~\ref{thm-combo1} $(B)$.

\begin{remark}
$(1)$ In the above theorem, $Y$ is hyperbolic (\cite[Remark 4.4]{pranab-mahan}; see Remark \ref{crucial remark} (2)), and the inclusion $Y\to X$ admits a CT map due to Krishna and Sardar \cite{ps-krishna} (see Theorem \ref{all direction surj imply surj} $(2)$).

For a Gromov hyperbolic geodesic metric space $W$ that need not be proper, $\pa_sW$ denotes the sequential boundary (or the Gromov boundary) of $W$ (see Section \ref{prelims}).\smallskip

	

$(2)$ Note that the condition in Theorem \ref{thm-combo1} $(A)$ that `the valence at each vertex of any fiber $F_b,~b\in V(B)$ in $F_b$ is bounded above by some constant $D\ge1$' is not there in Theorem~\ref{thm-combo1} $(B)$. Also observe that the assumptions in Question~\ref{main qsn} imply the hypotheses of Theorem~\ref{thm-combo1} $(B)$ when the base $B$ is $[0,\infty)$. Hence, the conditions in Theorem~\ref{thm-combo1} $(B)$ are more general. For instance, cf. Example \ref{exp-combinatorial horoball} and Remark \ref{rmk-fibers are not uniformly QI}. Example \ref{exp-combinatorial horoball} discusses a metric graph bundle arising from combinatorial horoball, in which the fibers are not uniformly quasiisometric to a fixed (one-ended) hyperbolic space, and that the CT map is known to be surjective.
\end{remark}

\noindent{\bf CT map is not injective}: Recall that {\em dendrite} is a compact metric space in which any two distinct points are connected by a unique arc. Suppose $X$ is as in Theorem \ref{thm-combo1} or Corollary \ref{cor-main gp over ray new intro} with the base is a ray, i.e., $B=[0,\infty)$. Bowditch proved that the boundary $\pa X$ is a dendrite (see \cite[Proposition 2.5.2]{bowditch-stacks}). Under a mild assumption, using the topological property of dendrite, one can obtain that the surjective CT map in this setting is not injective (see Theorem \ref{thm-lamination}). (This is studied in terms of Cannon--Thurston lamination in Subsection \ref{subsec-CT lamination}.) As a particular case, if the fibers are uniformly quasiisometric to a fixed nonelementary hyperbolic group, then the surjective CT map fails to be injective. For instance, the CT map $\pa\pi^{-1}(b)\map\pa X,~b\in V(B)$ as in Corollary \ref{cor-main gp over ray new intro} is not injective (see Theorem \ref{thm-lamination}).\smallskip

In the course of proving our main results, we establish the following two theorems, which may be of independent interest. For the notion of exponential growth used here -- generalizing the classical definition for groups -- we refer the reader to Definition \ref{defn-exponential growth}. We say that a metric graph has bounded valence if the valence at each vertex is bounded above by a fixed constant (Definition \ref{defn-valence}). 
	
	\begin{theorem}[Theorem \ref{thm-bary imp expo growth}]\label{thm-bary imp expo growth1}
		Suppose $\Gamma$ is a hyperbolic metric graph and the barycenter map $\pa^3\Gamma\map\Gamma$ is coarsely surjective. Moreover, suppose that the graph $\Gamma$ has bounded valence. Then $\Gamma$ has exponential growth.
	\end{theorem}

	We prove Theorem \ref{thm-bary imp expo growth1} using the following result. 
	
	\begin{theorem}[see Theorem \ref{thm-emb of T3}]\label{thm-emb of T31}
		Suppose $X$ is a hyperbolic geodesic metric space and the barycenter map $\pa^3X\map X$ is coarsely surjective. Then we can quasiisometrically embed the trivalent tree $T_3$ in $X$. 
		
		Moreover, the quasiisometric embedding constant depends only on the hyperbolicity constant of $X$ and the coarse surjectivity constant of the barycenter map.
	\end{theorem}

Now we give some final remarks.

\begin{remark}\label{some remarks}
$(1)$ Suppose $\pi:X\to B$ is a metric graph bundle. Since the fibers are proper embedding in $X$, it follows from the definition of a metric graph bundle that $X$ is proper if and only if the base $B$ and the fibers are proper (see Definition \ref{defn-metric graph bundle}). Moreover, there is a constant $D\ge1$ such that the valence at each vertex of the fiber $F_b,~b\in V(B)$ in $F_b$ and the valence of each vertex of $B$ in $B$ is bounded above by $D$ if and only if there is a constant $D'\ge1$ such that the valence at each vertex of $X$ is bounded above by $D'$. Thus, in general, the total space $X$ need not be proper.

Note that neither Theorem~\ref{thm-combo1} nor Corollary~\ref{cor-main gp over ray new intro} requires the properness of $X$. However, the properness of the fibers plays a crucial role in their proofs (Sections \ref{metric graph bundle sec} and \ref{sec-main theorems over rays}). \smallskip

$(2)$ Given a metric bundle $\pi':X'\map B'$ (Definition \ref{defn-metric bundle}) there exists an {\em approximating metric graph bundle} $\pi:X\map B$ such that the natural maps between the fibers, and $X\map X'$ and $B\map B'$ are uniformly quasiisometry (see Appendix). So surjectivity of the CT map in the approximating metric graph bundle ensures the same in the metric bundle via this dictionary. In the Appendix, we have studied that the (uniform) bounded valence of the fibers of the approximating metric graph bundle $\pi:X\map B$ can be achieved if the fibers of $\pi':X'\map B'$ are {\em uniformly strongly proper} (see Definition \ref{strongly proper}). Consequently, Theorem \ref{thm-combo1} $(A)$ holds in metric bundles when the fibers are uniformly strongly proper, whereas Theorem \ref{thm-combo1} $(B)$ holds without this assumption (see Theorem \ref{thm-all in one appendix}).
\end{remark}

\noindent{\bf Some words on the proof of Theorem \ref{thm-combo1}}:
	
First, we recall a theorem of Krishna and Sardar (Theorem~\ref{all direction surj imply surj}), which states that, in order to prove Theorem~\ref{thm-combo1}, it suffices to consider the special case where the base $B$ of the metric graph bundle is $[0,\infty)$. Accordingly, we first prove the surjectivity of the CT map in this setting. The corresponding results are Theorem~\ref{thm-main one ended} for Theorem~\ref{thm-combo1} $(B)$ and Theorem~\ref{thm-main bdd valence} for Theorem~\ref{thm-combo1} $(A)$. We now sketch the proofs of Theorems~\ref{thm-main one ended} and \ref{thm-main bdd valence}, which may be viewed as the versions of Theorems~\ref{thm-combo1} $(B)$ and \ref{thm-combo1} $(A)$, respectively, with the base of the metric graph bundle $[0,\infty)$.\smallskip

\noindent\underline{Theorem \ref{thm-main one ended}}: Following \cite[Proposition 6.6]{ps-krishna}, we give a more detailed description of the boundary of $X$ where $\pi:X\map[0,\infty)$ is a metric graph bundle (see Proposition \ref{description of pa X}). This result states that a boundary point of $X$ is either a limit point of a fiber or a {\em good qi section} -- that is, it appears as barycenters of the flow of an ideal triangle (Definition \ref{good qi section}). Therefore, to prove the surjectivity of CT map, it is enough to check that these good qi sections can be realized as limit points of a fiber (see Lemma \ref{CT image is limit set}). Since the fibers are one-ended, we have paths that connect points on a geodesic line forming the ideal triangle, remaining entirely outside any large radius ball centered at the barycenter of the ideal triangle. This property allows us to adapt Bowditch's original approach, which was initially developed for the case when the fibers are isometric to the hyperbolic plane $\mathbb H^2$.\qed\smallskip
	

	\noindent\underline{Theorem \ref{thm-main bdd valence}}: One of the main ingredients in the proof of Theorem \ref{thm-main bdd valence} is Theorem \ref{thm-bary imp expo growth1}. We then apply Theorem~\ref{main thm over ray graph gen} to complete the proof of Theorem \ref{thm-main bdd valence}. Theorem \ref{main thm over ray graph gen} isolates the relevant condition, namely, {\em uniform exponential growth} (Definition \ref{defn-exponential growth}) of the fibers under which the CT map is surjective. The assumptions on fibers in Theorem \ref{thm-main bdd valence} ensure that Theorem \ref{thm-bary imp expo growth1} applies, yielding uniform exponential growth of the fibers. 
	\qed\smallskip

	\noindent{\bf Organization of the paper}: In Section \ref{prelims}, we recall basic results on hyperbolic geodesic metric spaces and their (Gromov) boundaries, and define the Cannon--Thurston map. In Subsection \ref{subsec-qi emb of T3}, we prove Theorem \ref{thm-emb of T31}, which serves as a key ingredient in the proof of Theorem \ref{thm-bary imp expo growth1}. The latter is proved in Subsection \ref{subsec-expo growth}.
	
	\noindent Section \ref{metric graph bundle sec} reviews relevant results on metric graph bundles from \cite{pranab-mahan}. Following \cite[Proposition $6.6$]{ps-krishna}, we provide a detailed description of the boundary of total space of a metric graph bundle over $[0,\infty)$ in Proposition \ref{description of pa X}.
	
\noindent In Section \ref{sec-main theorems over rays}, we prove Theorems \ref{thm-main one ended}, \ref{main thm over ray graph gen} and \ref{thm-main bdd valence} which are special versions of Theorems \ref{thm-combo1} $(B)$ and \ref{thm-combo1} $(A)$. Notably, Theorem \ref{main thm over ray graph gen} is the main step towards proving Theorem \ref{thm-main bdd valence}. We end this section with a proof of Corollary \ref{cor-main gp over ray new intro} in the special case, and with some interesting questions.

	\noindent In Section \ref{sec-main thm}, we prove our main theorem (Theorem \ref{thm-combo1} and Corollary \ref{cor-main gp over ray new intro}) using results proved in Section \ref{sec-main theorems over rays}.
	
	\noindent In Section \ref{application and exp sec}, we show some examples of our main theorem and an application.

	
	
	
\noindent Finally, Appendix \ref{appendix sec} extends our study to metric bundles, following the dictionary provided in \cite{pranab-mahan}.\smallskip
	
{\bf Acknowledgement.} The author would like to thank Pranab Sardar and Mahan Mj for many helpful discussions, which also helped in the exposition of the article. The author is also grateful to them for their valuable feedback on earlier drafts. The author also thanks the anonymous referee(s) for a careful reading of the paper and for numerous insightful comments and suggestions that greatly improved the exposition and clarity of the article. This work was supported by the TIFR Visiting (Postdoctoral) Fellowship.

	\section{Preliminaries}\label{prelims}
	We start by recalling some standard notions used later in this article (see \cite[Chapters $I.1$ and $III.H$]{bridson-haefliger}). In a geodesic metric space $X$, a {\em geodesic segment} joining $x$ and $y$ is denoted by $[x,y]_X$ (or simply $[x,y]$ when $X$ is understood). For a metric space $(X,d)$, $x\in X$ and a subset $A\sse X$, we define $d(x,A)=inf~\{d(x,a):a\in A\}$. Then for $r\ge0$, we define $N_r(A)=\{x'\in X:d(x',A)\le r\}$. For two subsets $A,B$ of a metric space $X$, the {\em Hausdorff distance} between $A$ and $B$ is $Hd(A,B):=inf~\{r\ge0:A\sse N_r(B),B\sse N_r(A)\}$. Two subsets $A$ and $B$ of $X$ are said to be {\em $r$-separated} if $d(a,b)\ge r$ for all $a\in A$ and $b\in B$ where $r\ge0$. A metric space is said to be {\em proper} if closed and bounded sets are compact. In this paper, {\em graph} is an $1$-dimensional connected CW complex in which no edges form a loop. {\em Metric graph} is a graph in which every edge has length $1$. Hence, it is a geodesic metric space (\cite[1.9, Part-I]{bridson-haefliger}). A subgraph of a metric graph is always assumed to be connected. For a graph $X$, we denote the vertex set of $X$ by $V(X)$. 
	
	Suppose $c:[a,b]\map X$ is a continuous path in a metric space $(X,d_X)$. Then {\em length of $c$} is defined to be $$length~(c):= \sup_{\substack{a=t_0 \le t_1 \le \dots \le t_n=b}} 
	\sum_{i=0}^{n-1} d_X\big(c(t_i), c(t_{i+1})\big).$$We say $c$ is {\em arclength parametrized} if $c(t)$ is the point so that $t=length~(c|_{[a,t]})$.
	
	Suppose $f:(X,d_X)\map(Y,d_Y)$ is a map between two metric spaces. We say that $f$ is $C$-{\em coarsely surjective} for some $C\ge0$ if $N_C(f(X))=Y$; and $f$ is {\em coarsely surjective} if $f$ is $C$-coarsely surjective for some $C\ge0$. Let $k\ge1$ and $\ep\ge0$. We say $f$ is a $(k,\ep)$-{\em quasiisometric embedding $($or $(k,\ep)$-qi embedding, in short$)$} if for all $x,x'\in X$, we have $$\frac{1}{k}d_X(x,x')-\ep\le d_Y(f(x),f(x'))\le kd_X(x,x')+\ep.$$By a $k$-{\em qi embedding}, we mean $(k,k)$-qi embedding. We say that $f:X\map Y$ is $k$-{\em quasiisometry} if $f$ is $k$-qi embedding and $k$-coarsely surjective. We say that $f$ is a {\em qi embedding} (respectively, {\em quasiisometry}) if $f$ is $k$-qi embedding (respectively, $k$-quasiisometry) for some $k\ge1$.  Let $I\sse\R$ be an interval. A map $\al:I\map X$ is said to be $k$-{\em quasigeodesic} if it is $k$-qi embedding. We say that $\al$ is a quasigeodesic if it is a $k$-quasigeodesic for some $k\ge1$. When working with a quasigeodesic $\al:I\sse\R\map X$, we disregard the domain $I$ and think of $\al$ sits in $X$ as a (possibly discontinuous) path.\smallskip
	
	%

	We assume that the reader is familiar with hyperbolic spaces (and groups) (see \cite{abc}, \cite[Chapter III.H]{bridson-haefliger} and \cite{GhH}; \cite{gromov-hypgps}). In this paper, hyperbolic spaces are assumed to be geodesic metric spaces. Therefore, for a $\delta$-hyperbolic geodesic metric space (for some $\delta \ge 0$), we assume that geodesic triangles are both $\delta$-slim and $\delta$-thin (see \cite[Proposition 2.1]{abc} for the equivalence of these definitions). When we specifically use the thin triangle condition, we will mention it explicitly. Otherwise, we work with the slim triangle condition.\smallskip 

	{\bf \em Convention}: {\em When we say that a collection of spaces $\{X_{\lambda} : \lambda \in \Lambda\}$ satisfies a property $\mathcal{P}$ {\em uniformly}, we mean that the constants or parameters associated with $\mathcal{P}$ are independent of $X_{\lambda},~\lambda \in \Lambda$. For example, the collection $\{X_{\lambda} : \lambda \in \Lambda\}$ is said to be {\em uniformly hyperbolic} if there exists a constant $\delta \geq 0$ such that each $X_{\lambda}$ is $\delta$-hyperbolic.}\smallskip
	
	
	
	\begin{defn}\label{defn-gromov inner pd}
		Let $X$ be a metric space. Let $x,y,z\in X$. Then the Gromov inner product of $y$ and $z$ with respect to $x$ is defined and denoted as follows. $$(y,z)_x:=\frac{1}{2}\{d(x,y)+d(x,z)-d(y,z)\}$$
	\end{defn}
	
	\begin{lemma}\label{lem-small inner prod}
		Let $\dl\ge0$ and $K\ge1$. Then there is a constant $D_{\ref{lem-small inner prod}}(\dl,K)=K\dl+K^2/2$ such that the following holds. Suppose $X$ is a $\dl$-hyperbolic geodesic metric space, and let $y,z,w\in X$. Further, suppose that the arclength parametrization of $[z,y]\cup[y,w]$ is a $K$-quasigeodesic. Then for any $x\in X$, we have $min~\{(x,z)_y,(x,w)_y\}\le D_{\ref{lem-small inner prod}}(\dl,K)$. Moreover, $(z,w)_y\le D_{\ref{lem-small inner prod}}(\dl,K)$.
	\end{lemma}
	
	\begin{proof}
Without loss of generality, we assume that the geodesic triangles in $X$ are $\dl$-thin (see \cite[Proposition $2.1$]{abc}). 

Let us show the moreover part first. Let $t'\in[y,z],t''\in[y,w]$ be such that $d(y,t')=d(y,t'')=(z,w)_y$. Then by $\dl$-thinness of triangle, $d(t',t'')\le\dl$. Since the arclength parametrization of $[z,y]\cup[y,w]$ is a $K$-quasigeodesic, $\frac{1}{K}\{d(t',y)+d(y,t'')\}-K\le d(t',t'')\le\dl$, and so $(z,w)_y=d(t',y)=d(y,t'')\le\frac{1}{2}(K\dl+K^2)$.

For the first part, without loss of generality, we assume that $(x,z)_y=\min\{(x,z)_y,(x,w)_y\}$. Let $t\in[y,x]$ and $t'\in[y,z]$ be such that $d(y,t)=d(y,t')=(x,z)_y$. We will prove that $d(y,t')\le D_{\ref{lem-small inner prod}}(\dl,K)$. Since $(x,z)_y\le(x,w)_y$, there is $t''\in[y,w]$ such that $d(y,t)=d(y,t'')$ and $d(t,t'')\le\dl$. Again, $d(t,t')\le\dl$ by $\dl$-thinness and hence by triangle inequality, $d(t',t'')\le2\dl$. Now the arclength parametrization of $[z,y]\cup[y,w]$ is a $K$-quasigeodesic. Thus $\frac{1}{K}\{d(t',y)+d(y,t'')\}-K\le d(t',t'')\le2\dl$ and so, $d(y,t')=d(y,t'')=(x,z)_y\le\frac{1}{2}(2K\dl+K^2)$.

Therefore, we can take $D_{\ref{lem-small inner prod}}(\dl,K)=K\dl+K^2/2$. This completes the proof.
	\end{proof}

	For the following result, one is referred to \cite[Theorem $1.4$, Chapter $3$]{CDP}. Recall that given $k\ge1,~\ep\ge0$, $L>0$ and an interval $I\sse\R$, a map $\al:I\map X$ is said to be $(k,\ep,L)$-{\em local quasigeodesic} if $\al$ restricted to any subinterval of length $\le L$ is $(k,\ep)$-quasigeodesic.
	
	\begin{lemma}[{Local quasigeodesic vs global quasigeodesic}]\label{local vs global}
		For all $\dl\ge0,~k\ge1$ and $\ep\ge0$ there are constants $L_{\ref{local vs global}}=L_{\ref{local vs global}}(\dl,k,\ep)>0$ and $ \lm_{\ref{local vs global}}=\lm_{\ref{local vs global}}(\dl,k,\ep)\ge1$ such that the following holds.
		
		Suppose $X$ is a $\dl$-hyperbolic geodesic metric space. Then any $(k,\ep,L_{\ref{local vs global}})$-local quasigeodesic in $X$ is a $\lm_{\ref{local vs global}}$-quasigeodesic.
	\end{lemma}
	
	As an application of Lemma \ref{local vs global}, we have the following.
	
	\begin{lemma}\label{local to global}
		Let $\dl\ge0$ and $C\ge0$. Then there are constants $D_{\ref{local to global}}(\dl,C)=L_{\ref{local vs global}}(\dl,1,2C+2\dl)>0$ and $K_{\ref{local to global}}=K_{\ref{local to global}}(\dl,C)\ge1$ depending only on $\dl$ and $C$ such that the following holds.
		
		Suppose $X$ is a $\dl$-hyperbolic geodesic metric space. Let $x=x_0,x_1,\dots,x_n=y\in X$ and $\al_i=[x_{i-1},x_i]$ where $i\in\{1,2,\dots,n\}$.  If $d(x_{i-1},x_i)\ge D_{\ref{local to global}}(\dl,C)$ and the Gromov inner product $(x_{i-1},x_{i+1})_{x_i}\le C$ then the arclength parametrization of the concatenation $\bt=\al_1*\al_2*\dots*\al_n$ is $K_{\ref{local to global}}$-quasigeodesic.
	\end{lemma}
	
	\begin{proof}
		We will show that $\bt$ is a local quasigeodesic. Then we will be done by Lemma \ref{local vs global}.
		
		Without loss of generality, we assume that the geodesic triangles in $X$ are $\dl$-thin (see \cite[Proposition $2.1$]{abc}). Let $p\in\al_i$ and $q\in\al_{i+1}$ such that $(x_{i-1},x_{i+1})_{x_i}=d(x_i,p)=d(x_i,q)$. Then $d(p,q)\le\dl$, and by our assumption $d(p,x_i)=d(q,x_i)\le C$. 
		We will show that the arclength parametrization, $\gm:I=[0,d(x_{i-1},x_i)+d(x_i,x_{i+1})]\map X$ say, of the concatenation $\al_i*\al_{i+1}$ is a quasigeodesic where $\gm(0)=x_{i-1}$. Suppose $s,t\in I$. First of all, note that $d(\gm(s),\gm(t))\le|s-t|$ since $\gm$ is arclength parametrization. So we have to prove the other inequality. If $\gm(s),\gm(t)\in\al_i$ or $\gm(s),\gm(t)\in\al_{i+1}$ then we already have $d(\gm(s),\gm(t))=|s-t|$ since $\al_i$'s are geodesics. Now suppose, without loss of generality, $\gm(s)\in\al_i$ and $\gm(t)\in\al_{i+1}$, and so $s\le t$.
		
		\noindent{\em Case} 1: Suppose $\gm(s)\in[p,x_i]$. Then
		\begin{eqnarray*}
			d(\gm(s),\gm(t))&\ge&d(\gm(s'),\gm(t))-C\\
			&&\text{(by triangle inequality, where }\gm(s')=x_i\text{ and }s'-s=d(\gm(s),x_i)\le C)\\
			&=&(t-s')-C=(t-s)-(s'-s))-C\ge (t-s)-2C\\
			&&\text{(since }\al_i\text{ is geodesic})
		\end{eqnarray*}
		
		\noindent{\em Case} 2: Suppose $\gm(t)\in[x_i,q]$. By a similar argument as in Case 1, we will have $d(\gm(s),\gm(t))\ge(t-s)-2C$.
		
		
		\noindent{\em Case} 3: Suppose $\gm(s)\in[x_{i-1},p]$ and $\gm(t)\in[q,x_{i+1}]$. We also assume that $X$ satisfies Gromov inner product condition (of hyperbolicity) with constant $\dl$ (see \cite[Definition 1.20, Part-III.H]{bridson-haefliger}). Then by $4$-point condition (see \cite[p. 410]{bridson-haefliger}), we have $d(\gm(s),p)+d(\gm(t),q)\le max\{d(\gm(s),\gm(t))+d(p,q),d(\gm(s),q)+d(\gm(t),p)\}+2\dl$. Since $d(p,q)\le\dl$, by triangle inequality, $d(\gm(s),p)+d(\gm(t),q)\le d(\gm(s),q)+d(\gm(t),p)+2\dl$. Hence,
		\begin{eqnarray*}
			d(\gm(s),\gm(t))&\ge&d(\gm(s),p)+d(q,\gm(t))-2\dl\\
			&=&d(\gm(s),x_i)-d(p,x_i)+d(x_i,\gm(t))-d(x_i,q)-2\dl\\
			&\ge&d(\gm(s),x_i)+d(x_i,\gm(t))-2C-2\dl=(t-s)-2C-2\dl\\
			&&\text{(since }d(p,x_i)=d(q,x_i)\le C)
		\end{eqnarray*}
		
		\noindent This shows that $\gm$ is a $(1,2C+2\dl)$-quasigeodesic. Since $i$ was arbitrary, the restriction of $\bt$ on the concatenation $\al_i*\al_{i+1}$ is $(1,2C+2\dl)$-quasigeodesic for any $i\in\{1,2,\dots,n-1\}$.
		
		Again, we have $d(x_i,x_{i+1})\ge D_{\ref{local to global}}(\dl,C)=L_{\ref{local vs global}}(\dl,1,2C+2\dl)$ for all $i\in\{0,1,\dots,n-1\}$. So $\bt$ is a $(1,2C+2\dl,L_{\ref{local vs global}}(\dl,1,2C+2\dl))$-local quasigeodesic. Therefore, by Lemma \ref{local vs global}, $\bt$ is a $K_{\ref{local to global}}(\dl,C)$-quasigeodesic where $K_{\ref{local to global}}(\dl,C)=\lambda_{\ref{local vs global}}(\dl,1,2C+2\dl)$. This completes the proof.
	\end{proof}
	
	We end this subsection by recalling the exponential divergence property -- which is in a specialized form -- of (quasi)geodesics in hyperbolic metric spaces. Throughout the paper and in particular, for the following lemma, we will use the convention that $\mathbb N$ does not include $0$.

	\begin{lemma}\textup{(\cite[Proposition 1.6, III.H]{bridson-haefliger})}\label{geo line div exp}
		Suppose $X$ is a $\dl$-hyperbolic geodesic metric space for some $\dl\ge0$. Let $p,q\in X$, and let $\al$ be a geodesic segment joining $p$ and $q$. Fix $x\in\al$. Then any path lying outside $n$-radius ball centered at $x$ joining $p$ and $q$ has length at least $b^{n-1}$ where $b=2^{\frac{1}{\delta+1}}$ and $n\in\N$.
	\end{lemma}
	\subsection{Geodesic boundary (and Gromov boundary), and Cannon--Thurston map}\label{subsec-gromov bdry and CT}
In this subsection, we will see some properties of the {\em geodesic boundary} of hyperbolic spaces and the {\em Cannon--Thurston map} between hyperbolic spaces. We assume that the reader is familiar with the geodesic boundary of hyperbolic spaces (and groups) (see \cite{gromov-hypgps}, \cite[Chapter III.H]{bridson-haefliger} and \cite{GhH}). For a proper hyperbolic geodesic metric space $X$, the {\em geodesic boundary} of $X$ is denoted by $\pa X$ and is defined to be the equivalence classes of geodesic rays. Two geodesic rays $\al$ and $\bt$ are said to be related if $Hd(\al,\bt)<\infty$. For a geodesic ray $\alpha:I\to X$, where $I=[0,\infty)$ or $(-\infty,0]$, its equivalence class in $\partial X$ is denoted by $\alpha(\infty)$ or $\alpha(-\infty)$, respectively. 
	
	\begin{lemma}\label{existence of geo ray line}
		Suppose $X$ is a proper $\dl$-hyperbolic geodesic metric space for some $\dl\ge0$. Then:
		\begin{enumerate}
			\item \textup{(\cite[Lemma $3.1$, $III.H$]{bridson-haefliger})} Given $x\in X$ and $\xi\in\pa X$, there is a geodesic ray, say $\al$, such that $\al(0)=x$ and $\al(\infty)=\xi$.
			
			\item \textup{(\cite[Lemma $3.2$, $III.H$]{bridson-haefliger})} \textup{(Visibility of $\pa X$)} Given two distinct points $\xi_1,\xi_2\in \pa X$ there is a geodesic line, say $\bt$, in $X$ such that $\bt(-\infty)=\xi_1$ and $\bt(\infty)=\xi_2$.
		\end{enumerate}
	\end{lemma}

We note that the Gromov compactification $X\cup\pa X$ of $X$ admits a compact metric such that the inclusion $X\to X\cup\pa X$ is homeomorphism onto its image (see \cite[Lemma 11.76, Corollary 11.78]{kap-drutu-book}, \cite[Proposition 3.7, III.H]{bridson-haefliger}). With respect to this topology, we note the following notion of convergence. Let $\{x_n\}\sse X\cup\pa X$ be  a sequence and let $x\in X$. Let $\al_n$ be a geodesic segment or ray joining $x$ and $x_n$, according as $x_n\in X$ or $x_n\in\pa X$. We say that $\{x_n\}$ converges to a point $\xi\in\pa X$ if geodesic segments $\al_n$ converges $($uniformly on compact sets$)$ to a geodesic ray representing $\xi$ (see \cite[Definition 3.5]{bridson-haefliger}).\medskip

\noindent{\bf Gromov boundary.} We will not need the following discussion on the Gromov boundary until the proof of the main theorem in Section~\ref{sec-main thm}. For the reader's convenience, we briefly recall the notion of the {\em Gromov boundary} (also called the {\em sequential boundary}) of a hyperbolic geodesic metric space, which need not be proper.

Suppose $W$ is a $\dl$-hyperbolic metric space for some $\dl\ge0$ {\em in the sense of Gromov}, i.e., for a fixed (hence, any) $p\in W$ and for any $x,y,z\in W$, the Gromov inner products (see Definition \ref{defn-gromov inner pd}) satisfy the following: $$(x,y)_p\ge min\{(x,z)_p,(z,y)_p\}-\dl.$$ Let $\mathcal S$ be the set of all sequences $\{x_n\}\sse W$ such that $\lim_{i,j\map\infty}(x_i,x_j)_p=\infty$. We refer to such sequences as {\em converging to infinity}. On $\mathcal S$ we define an equivalence relation where $\{x_n\}\sim\{y_n\}$ if and only if $\lim_{i,j\map\infty}(x_i,y_j)_p=\infty$. The {\em Gromov boundary} or {\em the sequential boundary} of $W$ is defined to be $\mathcal S/\sim$ as a set and denoted by $\pa_sW$ (\cite[Definition 3.12, III.H]{bridson-haefliger}). The equivalence class of $\{x_n\}$ is denoted by $[\{x_n\}]$. The topology on $\pa_sW$ is defined by the following convergence. We say that a sequence $\{\xi_n\}_n\sse\pa_sW$ where $\xi_n=[\{x^{(n)}_i\}_i]\in\pa_sW$ converges to $\xi\in\pa_sW$ where $\xi=[\{x_i\}_i]$ if $$\lim_{n\map\infty}(\lim\text{inf}_{i,j\map\infty}(x_i,x^{(n)}_j)_p)=\infty$$ (for instance, see \cite[Definition 2.43]{ps-krishna}, and this definition coincides with that in \cite[III.H]{bridson-haefliger} and \cite{abc}). (A subset $A\sse\pa_sW$ is closed if for any sequence $\{\xi_n\}\sse A$ such that $\{\xi_n\}$ converges to $\xi\in\pa_sW$ then $\xi\in A$.)

Now assume, moreover, that $W$ is unbounded, geodesic metric space (that need not be proper). Then there is a (continuous) $k$-quasigeodesic ray where $k\ge1$ depends only on $\dl$ (see \cite[Lemma $2.4$]{pranab-mahan}). Equivalence classes of these $k$-quasigeodesic rays are denoted by $\pa_q W$. (Two quasigeodesic rays are related if their Hausdorff distance is finite.) For a quasigeodesic ray $\al:[0,\infty)\map W$, its equivalence class is denoted by $\al(\infty)$. Then Lemma \ref{existence of geo ray line} holds for $\pa_qW$ as well where geodesics are replaced by $k$-quasigeodesics  (for instance, see \cite[Lemma 2.4]{pranab-mahan}). By the same result, one concludes that there is a natural map $\pa_qW\map\pa_sW$ which is a bijection.

Finally, if $W$ is a proper, hyperbolic geodesic metric space, then the natural map $\pa W\map \pa_qW$ is a bijection (see \cite[Lemma $3.1$]{bridson-haefliger}). In fact, the composition of maps $\pa W\map\pa_qW$ and $\pa_qW\map\pa_sW$ is a homeomorphism (see \cite[Proposition $3.21$]{bridson-haefliger}).\qed\medskip

\noindent{\bf Notation.} Suppose $X$ is a hyperbolic geodesic metric space that need not be proper. Let $\{x_n\}\sse X\cup\pa_s X$ be a sequence such that $\{x_n\}$ converges to a point $\xi\in\pa_s X$. In this case, we write $\LMX x_n=\xi\in\partial_s X$ and sometimes, we say $\LMX x_n$ exists in $\pa_s X$. (Otherwise, $\LMX x_n$, if exists, means usual convergence in $X$.) We frequently encounter hyperbolic spaces $Y\sse X$ and sequence, say $\{y_n\}$, in $Y$, so to differentiate the convergence of $\{y_n\}$ in $\pa_s Y$ and in $\pa_s X$ we put a superscript in the limit.\medskip

		The results discussed in this section also apply to hyperbolic geodesic metric spaces that are not necessarily proper; the only modification required is to replace geodesic rays (respectively, lines) with quasigeodesic rays (respectively, lines). However, the properness (of {\em fibers of metric graph bundle}) is crucial in later sections, specifically, in Proposition \ref{mild imp special}.

	\begin{lemma}\textup{(Stability of quasigeodesic; \cite[Theorem 1.7, Lemma 3.3, III.H]{bridson-haefliger})}\label{stability of qc}
		Given $\dl\ge0$ and $k\ge1$ there is a constant $D_{\ref{stability of qc}}=D_{\ref{stability of qc}}(\dl,k)\ge0$ such that the following holds.
		
		Suppose $X$ is a $\dl$-hyperbolic geodesic metric space. Let $\al$ be a geodesic segment $($respectively, ray or line$)$ in $X$, and let $\bt$ be a $k$-quasigeodesic segment $($respectively, ray or line$)$ in $X$ joining endpoints of $\al$. Then $Hd(\al,\bt)\le D_{\ref{stability of qc}}$.
		
		Moreover, if both $\al$ and $\bt$ are geodesic segments, rays or lines, then $Hd(\al,\bt)\le2\dl$.
	\end{lemma}

As a simple application of the stability of quasigeodesic, we obtain the following result. We omit its proof.
	\begin{lemma}\label{lem-ideal tripod to finite}
		Let $\dl\ge0$ and $C\ge0$. Then there is a constant $K_{\ref{lem-ideal tripod to finite}}=K_{\ref{lem-ideal tripod to finite}}(\dl,C)\ge1$ such that we have the following.
		
		Suppose that $X$ is a $\delta$-hyperbolic geodesic metric space, and let $\xi, \eta \in \partial X$ be distinct points. Let $\alpha$ be a geodesic line joining $\xi$ and $\eta$, and let $x \in X$ be such that $d(x, \alpha) \le C$. Further, suppose that $\beta$ and $\gamma$ are geodesic rays starting at $x$ such that $\beta(\infty) = \xi$ and $\gamma(\infty) = \eta$.

		Then for any points $y\in\bt$ and $z\in\gm$ the arclength parametrization of $[y,x]\cup[x,z]$ is a $K_{\ref{lem-ideal tripod to finite}}$-quasigeodesic.
	\end{lemma}
	\begin{theorem}\textup{(\cite[Theorem $3.9$]{bridson-haefliger})}\label{qi emb gives top emb}
		Suppose $f:X\map Y$ is a qi embedding between two proper hyperbolic geodesic metric spaces. Then $f$ defines a topological embedding $\pa f:\pa X\map\pa Y$ sending $\al(\infty)\mapsto f\circ\al(\infty)$ where $\al$ is a geodesic ray in $X$. Moreover, $\pa f:\pa X\map\pa Y$ is a homeomorphism if $f$ is a quasiisometry.
	\end{theorem}
	
	While the following results are commonly stated for (quasi)geodesic segments, we will prove them for geodesic lines, as this format will be used in subsequent applications.
	
	\begin{lemma}\label{lem-geo lying out imply out}
		Suppose $k\ge1$, $\dl\ge0$ and $D_1=D_{\ref{stability of qc}}(\dl,k)$. Then there are constants $$D_{\ref{lem-geo lying out imply out}}=D_{\ref{lem-geo lying out imply out}}(\dl,k)=4k+D_1 \text{ and }A_{\ref{lem-geo lying out imply out}}=A_{\ref{lem-geo lying out imply out}}(\dl,k)=~max~\{2k,2D_1\}(\ge2)$$ depending only on $\dl$ and $k$ such that the following hold.
		
Suppose $X$ and $Y$ are $\delta$-hyperbolic, proper, geodesic metric spaces, and let $\phi : X \to Y$ be a $k$-qi embedding. Let $\alpha$ be a geodesic line in $X$, and let $\beta$ be a geodesic line in $Y$ such that $\beta(-\infty) = \partial\phi(\alpha(-\infty))$ and $\beta(\infty) = \partial\phi(\alpha(\infty))$. Let $x \in X$ and $n \in \mathbb{N}$. Suppose $\gamma$ is a continuous rectifiable path in $X$, lying outside the ball of radius $n$ centered at $x$, and joining two points $x_1, x_2 \in \alpha$.

Then there exists a continuous path $\gamma'$ in $Y$ joining two points on $\beta$ such that:


$(a)$ $Hd_Y(\phi(\gamma), \gamma') \le D_{\ref{lem-geo lying out imply out}}$.

$(b)$  $length~(\gm')\le A_{\ref{lem-geo lying out imply out}}~length(\gm)+A_{\ref{lem-geo lying out imply out}}$. Hence $\gm'$ is also rectifiable.
	\end{lemma}
	
	\begin{proof}
		Let $x_1=z_0,z_1,\dots,z_l=x_2$ be points on $\gm$ such that $length~(\gm|_{[z_{t-1},z_t]})=1$ for all $1\le t\le l-1$ and $length~(\gm|_{[z_{l-1},z_l]})\le 1$, where $length~(\gm|_{[z_{t-1},z_t]})$ denotes the length of $\gm$ when restricted from $z_{t-1}$ to $z_t$. In particular, $d_X(z_{t-1},z_t)\le 1$ for all $1\le t\le l$ and $length(\gm)\le l$. Let $\gm_1=\bigcup_{1\le t\le l}[z_{t-1},z_t]_X$ and $\gm_2=\bigcup_{1\le t\le l}[\phi(z_{t-1}),\phi(z_t)]_Y$. Note that $Hd_X(\gm,\gm_1)\le 1/2<1$. 
		Since $\phi$ is $k$-qi embedding, we have
		
		\begin{enumerate}
\item $length~(\gm_2)\le \sum_{t=1}^{t=l} (d_X(z_{t-1},z_t)k+k)\le2k~length~(\gm)$, and
			
			
			\item $Hd_Y(\phi(\gm),\phi(\gm_1))\le2k$, $Hd_Y(\phi(\gm_1),\gm_2)\le2k$, and so $Hd_Y(\phi(\gm),\gm_2\le4k$.
						
		\end{enumerate}
		
		\noindent	Suppose $y_1,y_2\in\bt$ such that $d_Y(\phi(x_i),y_i)\le D_1$ where $D_1=D_{\ref{stability of qc}}(\dl,k)$ and $i=1,2$ (by the stability of quasigeodesics, Lemma \ref{stability of qc}). 
		Set $\gm'=[y_1,\phi(x_1)]_Y\cup\gm_2\cup[\phi(x_2),y_2]_Y$.\smallskip
		
		
		\noindent	$(a)$ Then $Hd_Y(\gm_2,\gm')\le D_1$ implies $Hd_Y(\phi(\gm),\gm')\le4k+D_1=D_{\ref{lem-geo lying out imply out}}(\dl,k)$.\smallskip
		
		\noindent	$(b)$ Finally, $length~(\gm')\le~length~(\gm_2)+2D_1\le2k~length~(\gm)+2D_1\le A_{\ref{lem-geo lying out imply out}}~length(\gm)+A_{\ref{lem-geo lying out imply out}}$ where $A_{\ref{lem-geo lying out imply out}}(\dl,k)=~max~\{2k,2D_1\}$.
	\end{proof}


	

The following two results (Lemmas \ref{lem-geom criterion for conv} and \ref{require in bary}) provide geometric criteria for convergence. For their proofs we need the following discussion. 
	
Suppose $X$ is a proper hyperbolic metric space. One can extend the Gromov inner product (Definition \ref{defn-gromov inner pd}) in $X$ to $X\cup\pa X$ as follows. Let $x\in X$ and $y,z\in X\cup\pa X$. Then $$(y,z)_x=\text{sup lim inf }_{i,j\to\infty}(y_i,z_j)_x$$ where $\{y_i\},\{z_i\}\sse X$, $\lim^X_{i\to\infty}y_i=y$, $\lim^X_{i\to\infty}z_i=z$ and the supremum is taken over all such sequences $\{y_i\}$ and $\{z_i\}$ (see \cite[Definition 3.15, III.H]{bridson-haefliger}). Further, assume that $X$ is $\delta$-hyperbolic for some $\delta\ge 0$. Let $\alpha$ denote a geodesic segment, ray, or line joining $y$ and $z$, according as $y,z\in X$, $y\in X$ and $z\in\partial X$, or $y,z\in\partial X$, respectively. Then $|(y,z)_x-d(x,\al)|\le D(\dl)$ for some uniform constant $D(\dl)\ge0$ depending only on $\dl$ (see \cite[Exercises 3.18 (3), III.H]{bridson-haefliger}, \cite[Lemma $17$]{GhH}).

Let $\{\xi_n\}$ be a sequence in $\pa X$ and $\xi\in\pa X$. 
Then $\lim^X_{n\to\infty}\xi_n=\xi\in \pa X$ if and only if $\lim_{n\to\infty}(\xi_n,\xi)_x=\infty$ (see \cite[Remarks 3.17 (6), III.H]{bridson-haefliger}). Hence combining this with the above fact we have the following. For a complete proof of Lemma \ref{lem-geom criterion for conv}, one may look at \cite[Lemma 2.45 (2)]{ps-krishna}.

\begin{lemma}\label{lem-geom criterion for conv}
Suppose $X$ is a proper hyperbolic geodesic metric space. Let $\{\xi_n\}$ be a sequence in $\pa X$ and $\xi\in\pa X$. Let $\bt_n$ be a geodesic line joining $\xi_n$ and $\xi$. Then for any fixed $x\in X$, $\lim_{n\to\infty}d(x,\bt_n)=\infty$ if and only if $\lim^X_{n\to\infty}\xi_n=\xi$.
\end{lemma}
	
With the help Lemma \ref{lem-geom criterion for conv}, we prove the following.
	
	\begin{lemma}\label{require in bary}
	Suppose $X$ is a proper hyperbolic geodesic metric space. Let $\{\xi_n\}$ and $\{\eta_n\}$ be sequences in $X$ (respectively, in $\partial X$), and let $\alpha_n$ be a geodesic segment (respectively, line) in $X$ joining $\xi_n$ and $\eta_n$. Assume that $\lim^X_{n \to \infty} \xi_n=\xi\in\pa X$. Let $x \in X$ be arbitrary. Then
$\lim^X_{n\map\infty}\eta_n=\xi$ if and only if $\lim_{n\map\infty}d(x,\al_n)=\infty$.

\begin{proof}
If both $\{\xi_n\}$ and $\{\eta_n\}$ are sequences of $X$, then the result easily follows from the definition of points of $\pa X$ coming from Gromov inner products (see \cite[Definition 3.12]{bridson-haefliger}) and the fact $|(p,q)_r-d(r,[p,q])|\le D(\dl)$ as mentioned above where $p,q,r\in X$.


Now we assume that both $\{\xi_n\}$ and $\{\eta_n\}$ are sequences of $\pa X$. In this case we will conclude the result using Lemma \ref{lem-geom criterion for conv}. Let $\bt_n$ (resp. $\gm_n$) be a geodesic line joining $\eta_n$ and $\xi$ (resp. $\xi_n$ and $\xi$). Since $\lim^X_{n\to\infty}\xi_n=\xi$, by Lemma \ref{lem-geom criterion for conv}, $\lim_{n\to\infty}d(x,\gm_n)=\infty$. Now by slimness of ideal geodesic triangle $\al_n$, $\bt_n$ and $\gm_n$, we have $\lim_{n\to\infty}d(x,\bt_n)=\infty$ if and only if $\lim_{n\to\infty}d(x,\al_n)=\infty$. Finally, applying Lemma \ref{lem-geom criterion for conv} once again, $\lim^X_{n\to\infty}\eta_n=\xi$ if and only if $\lim_{n\to\infty}d(x,\al_n)=\infty$. This completes the proof.
\end{proof}	
	
%
	\end{lemma}
	
	As a consequence of Lemma \ref{require in bary} and the stability of quasigeodesic (Lemma \ref{stability of qc}), we have the following results (Lemmas \ref{close give same limit} and \ref{on qc imp same}).
	
	\begin{lemma}\label{close give same limit}
		Suppose $X$ is a proper hyperbolic geodesic metric space. Let $\{x_n\},\{y_n\}\sse X$ be sequences such that $\LMX x_n\in\pa X$ and $d(x_n,y_n)\le D$ for some $D\ge0$. Then $\lim^X_{n\map\infty}y_n$ exists in $\pa X$ and $\lim^X_{n\map\infty}x_n=\lim^X_{n\map\infty}y_n$.
	\end{lemma}	
	
	\begin{lemma}\label{on qc imp same}
		Suppose $X$ is a proper hyperbolic geodesic metric space. Let $\{x_n\}\sse X$ be such that $\LMX x_n$ exists in $\pa X$. Fix $x\in X$, and let $k\ge1$. Let $\al_n$ be a $k$-quasigeodesic joining $x$ and $x_n$, and let $y_n\in\al_n$ such that $d(x,y_n)\map\infty$ as $n\map\infty$. Then $\LMX y_n$ exists in $\pa X$ and $\LMX y_n=\LMX x_n$.
	\end{lemma}

We end this subsection with the definition of a Cannon--Thurston map, followed by a basic property of this map.

	\begin{defn}\label{CT-map}\textup{({\bf Cannon--Thurston (CT) map})}\textup{(\cite[p. $136$]{mitra-trees})}
	Suppose $f:Y\ri X$ is a map between hyperbolic metric spaces. We say that $f$ admits the Cannon--Thurston map $($in short, CT map$)$ if there is a continuous map $\partial f :\partial Y\ri \partial X$ induced by $f$ in the following sense:
	
	For all $\xi\in \partial Y$ and for any sequence $\{y_n\}$ in $Y$ with $\LMY y_n=\xi$ one has $\LMX f(y_n)\newline=\partial f(\xi)$.
\end{defn}

In this case, $\partial f$ is called the {\em CT map} induced by $f$. When $Y$ is a subspace of $X$ equipped with the induced path metric, a CT map for the inclusion $i_{Y,X}:Y\map X$ is denoted by $\pa i_{Y,X}$. 
We also note that the CT map is unique if it exists. 

\begin{defn}
	Suppose $A$ is any subset of a hyperbolic space $X$. The limit set of $A$ in $\pa X$ is $\Lambda_X(A):=\{\LMX a_n\in\pa X:\text{ for some sequence }\{a_n\}\sse A\}$.
\end{defn}
\begin{lemma}\textup{(\cite[Lemma 2.55]{ps-krishna})}\label{CT image is limit set}
	Suppose $Y\sse X$ are proper hyperbolic metric spaces such that the inclusion $i_{Y,X}:Y\map X$ admits a CT map $\pa i_{Y,X}:\pa Y\map\pa X$. Then $\pa i_{Y,X}(\pa Y)=\Lambda_X(Y)$. Note here that $Y$ is equipped with the induced path metric from $X$.
\end{lemma}

	\subsection{Barycenter of ideal triangle}\label{bary sec}

	\begin{lemma}\textup{(\cite[Lemmas 1.17, 3.3, III.H]{bridson-haefliger})} \textup{(Barycenter for ideal triangles)}\label{barycenter}
		Given $\dl\ge0$, there are constants $D_{\ref{barycenter}}(\dl)\ge0$ and $R_{\ref{barycenter}}(\dl)\ge0$ depending only on $\dl$ such that the following holds.
		
		Suppose $X$ is a proper $\dl$-hyperbolic geodesic metric space. Let $\pa^3X:=\{(\xi_1,\xi_2,\xi_3)\in\pa X\times\pa X\times \pa X:\xi_1\ne\xi_2\ne\xi_3\ne\xi_1\}$. Given $(\xi_1,\xi_2,\xi_3)\in\pa^3X$ there is a point $x\in X$ which is $D_{\ref{barycenter}}(\dl)$-close to any line joining $\xi_i$ and $\xi_j$ for $i\ne j$ and $i,j\in\{1,2,3\}$. Moreover, for any other such point $x'\in X$, we have $d(x,x')\le R_{\ref{barycenter}}(\dl)$. 
	\end{lemma}

	Let $X$ be a proper hyperbolic geodesic metric space. Let $(\xi_1,\xi_2,\xi_3)\in\pa^3 X$. Then an ideal triangle with vertices $\xi_1,\xi_2,\xi_3$ is a union of choice of three geodesic lines joining $\xi_i$ and $\xi_j$ for all distinct $i,j\in\{1,2,3\}$; and one such triangle is denoted by $\triangle(\xi_1,\xi_2,\xi_3)$. Note that ideal triangles are coarsely well-defined. 
	
	The following definition is motivated by Lemma \ref{barycenter} (see \cite[p. $1668$]{pranab-mahan}).
	\begin{defn}[Barycenter map]\label{barycenter map}
		Suppose $X$ is a proper $\dl$-hyperbolic geodesic metric space where $\dl\ge0$. For any $\xi=(\xi_1,\xi_2,\xi_3)\in\pa^3X$, a point $x$, as in Lemma \ref{barycenter} corresponding to $\xi$, is called barycenter of the ideal triangle $\triangle(\xi_1,\xi_2,\xi_3)$. Then we have a coarsely well-defined map $\pa^3X\map X$ (by Lemma \ref{barycenter}) sending $\xi$ to a barycenter. One such map is called the barycenter map and is denoted by $Bary_X$.
		
		For $D \ge 0$, a point $z \in X$ is called a $D$-barycenter of the ideal triangle $\triangle(\xi_1, \xi_2, \xi_3)$ if it $D$-close to each geodesic line joining $\xi_i$ and $\xi_j$ for distinct $i, j \in \{1, 2, 3\}$.
	\end{defn}
	\begin{remark}\label{rmk-bary map coarse surj in gp}
		Let $G$ be a nonelementary hyperbolic group. Fix a finite generating set for $G$. Let $\Gamma_G$ denote the Cayley graph of $G$ with respect to this generating set. Assume that $\Gamma_G$ is $\dl$-hyperbolic for some $\dl\ge0$. Note that $G$ acts transitively on the vertex set $V(\Gamma_G)$, and this action naturally induces an action on the Gromov boundary $\partial G$ by homeomorphism. Consequently, a chosen barycenter map for $\Gamma_G$ is $D_{\ref{barycenter}}(\dl)$-coarsely surjective by Lemma \ref{barycenter}.
	\end{remark}

	The following result says that the second part of Lemma \ref{barycenter} is also true for $D$-barycenter where $D\ge0$.
	\begin{lemma}\textup{( \cite[Lemma 2.7]{pranab-mahan})}\label{D-barycenter}
		Given $D\ge0$ and $\dl\ge0$, there is $D_{\ref{D-barycenter}}(\dl,D)\ge0$ depending only on $\dl$ and $D$ such that we have the following. Suppose $X$ is a proper $\dl$-hyperbolic geodesic metric space. Then for any $\xi=(\xi_1,\xi_2,\xi_3)\in\pa^3X$ if $x,x'$ are two $D$-barycenters of the ideal triangle $\triangle(\xi_1,\xi_2,\xi_3)$ then $d(x,x')\le D_{\ref{D-barycenter}}(\dl,D)$.
	\end{lemma}

	As an immediate consequence of Lemma \ref{D-barycenter} and the stability of quasigeodesic (Lemma \ref{stability of qc}) we have the following. This says that barycenter commutes with quasiisometry.
	
\begin{lemma}\textup{(\cite[Lemma 2.8]{pranab-mahan})}\label{lem-barycenter commuting}
Let $k\ge1$, $\dl\ge0$ and $C\ge0$. Suppose $\phi:X\map Y$ is a $k$-quasiisometry between two (proper) $\dl$-hyperbolic geodesic metric spaces. Further, assume that the barycenter map for $X$ is $C$-coarsely surjective. Then the barycenter map for $Y$ is $C'$-coarsely surjective for some constant $C'\ge0$ depending only on $C$, $k$ and $\dl$.
\end{lemma}

	\begin{lemma}\label{bdd barycen conv}
		Suppose $X$ is a proper $\dl$-hyperbolic geodesic metric space where $\dl\ge0$. Let $\xi_n=(\xi_{1,n},\xi_{2,n},\xi_{3,n})\in\pa^3X$ be such that $diam~\{Bary_X(\xi_n):n\in\N\}$ is bounded. Moreover, suppose that $\lim^{X}_{n\map\infty}\xi_{i,n}=\xi_i\in\pa X$ for $i\in\{1,2,3\}$. Then
		
		$(1)$ $\xi=(\xi_1,\xi_2,\xi_3)\in\pa^3X$, and
		
		$(2)$ for all large $n$, we have $d_X(Bary_X(\xi_{n}),Bary_X(\xi))\le D_{\ref{bdd barycen conv}}(\dl)$ for some constant $D_{\ref{bdd barycen conv}}(\dl)\ge0$ depending only on $\dl$.
	\end{lemma}
	
	\begin{proof}
		$(1)$ If not, without loss of generality, we assume that $\xi_1=\xi_2$. Suppose $\al_n$ is a geodesic line in $X$ joining $\xi_{1,n}$ and $\xi_{2,n}$. Note that for any $x\in X$, $d(x,Bary_X(\xi_{n}))\ge d(x,\al_n)-D_{\ref{barycenter}}(\dl)$ (by Lemma \ref{barycenter}). Again by Lemma \ref{require in bary}, $d(x,\al_n)\map\infty$ as $n\map\infty$, and so $d(x,Bary_X(\xi_{n}))\map\infty$ as $n\map\infty$. This contradicts to the assumption that $diam~\{Bary_X(\xi_n):n\in\N\}$ is bounded. Hence (1) is proved.\smallskip
		
		$(2)$ Let $i,j\in\{1,2,3\}$ and $i\ne j$. Suppose $\al^{ij}$ is a geodesic line joining $\xi_i$ and $\xi_j$ such that (after reparametrization if necessary) $\al^{ij}(0)$ is $D_{\ref{barycenter}}(\dl)$-close to $Bary_X(\xi)$. Let $\al^{ij}_n$ be a geodesic line joining $\xi_{i,n}$ and $\xi_{j,n}$ such that $\al^{ij}_n$ converges to a geodesic line $\bt^{ij}$, say, joining $\xi_i$ and $\xi_j$. Then $Hd_X(\al^{ij},\bt^{ij})\le2\dl$ (Lemma \ref{stability of qc}). Hence, for all large $n$, $d_X(\al^{ij}(0),\al^{ij}_n)\le 1+2\dl$ (say), and so $d_X(Bary_X(\xi),\al^{ij}_n)\le D_{\ref{barycenter}}(\dl)+1+2\dl$ where $i\ne j$ and $i,j\in\{1,2,3\}$. This says that $Bary_X(\xi)$ is a $(D_{\ref{barycenter}}(\dl)+1+2\dl)$-barycenter of triangle $\triangle(\xi_{1,n},\xi_{2,n},\xi_{3,n})$ for all large $n\in\N$. Then by Lemma \ref{D-barycenter}, for all large $n$, $d_X(Bary_X(\xi_{n}),Bary_X(\xi))\le D_{\ref{bdd barycen conv}}(\dl)$ where $D_{\ref{bdd barycen conv}}(\dl):=D_{\ref{D-barycenter}}(\dl,D_{\ref{barycenter}}(\dl)+1+2\dl)$.
	\end{proof}

	\section{Exponential growth of graphs}\label{sec-exponential growth}We define the exponential growth of bounded valence metric graphs generalizing that of groups in Subsection \ref{subsec-expo growth} (see Definition \ref{defn-exponential growth}). This section aims to prove Theorem \ref{thm-bary imp expo growth}, which provides a sufficient condition under which a metric graph possesses exponential growth. Theorem \ref{thm-bary imp expo growth} is obtained by applying Theorem \ref{thm-emb of T3} and Lemma \ref{lem-qi implies exponential growth}.
	
\subsection{Embedding of trivalent trees in hyperbolic metric spaces}\label{subsec-qi emb of T3}In this subsection, we prove Theorem \ref{thm-emb of T3} and the proof requires following result.
	
	\begin{lemma}\label{lem-finding ind bary}
		Let $\dl\ge0$, $L\ge0$ and $D\ge0$. Then there exists a constant $C_{\ref{lem-finding ind bary}}=C_{\ref{lem-finding ind bary}}(\dl,L)\ge0$, depending only on $\dl$ and $L$, such that the following hold.
		
		Let $X$ be a (proper) $\dl$-hyperbolic geodesic metric space. Suppose the barycenter map $Bary_X:\pa^3X\map X$ is $L$-coarsely surjective.  Then:
		\begin{enumerate}
			\item	Let $x,y\in X$. Then we can find distinct points $y_1,y_2\in X$ such that $d(y,y_i)=D$ for $i=1,2$ and $$max\{(y_1,y_2)_y,~(x,y_1)_y,~(x,y_2)_y\}\le C_{\ref{lem-finding ind bary}}.$$
			
			\item Let $y\in X$. Then we can find distinct points $y_1,y_2,y_3\in X$ such that $d(y,y_i)=D$ for $i=1,2,3$ and $$max\{(y_i,y_j)_y:i\ne j\text{ and }i,j\in\{1,2,3\}\}\le C_{\ref{lem-finding ind bary}}.$$
		\end{enumerate}	
	\end{lemma}
	
	\begin{proof}
		$(1)$ Let $\xi=(\xi_1,\xi_2,\xi_3)\in\pa^3X$ be such that $d(Bary_X(\xi),y)\le L$. Let $\al_i$ be a geodesic ray joining $y$ and $\xi_i$ for $i\in\{1,2,3\}$. Fix $y_i\in\al_i$ such that $d(y,y_i)=D$ for $i\in\{1,2,3\}$.
		
		Fix distinct $i,j$ in $\{1,2,3\}$. Let $\bt$ be a geodesic line joining $\xi_i$ and $\xi_j$. Then $d(y,\bt)\le d(y,Bary_X(\xi))+d(Bary_X(\xi),\bt)\le L+D_{\ref{barycenter}}(\dl)$ (see Lemma \ref{barycenter}). Thus by Lemma \ref{lem-ideal tripod to finite}, we have $K=K_{\ref{lem-ideal tripod to finite}}(\dl,L+D_{\ref{barycenter}}(\dl))$ such that the arclength parametrization of $[y_i,y]\cup[y,y_j]$ is a $K$-quasigeodesic. Hence by Lemma \ref{lem-small inner prod}, we have $C=D_{\ref{lem-small inner prod}}(\dl,K)$ such that $$min\{(x,y_i)_y,~(x,y_j)_y\}\le C\hspace{4mm} (*)$$ Note that $i$ and $j$ are distinct in $\{1,2,3\}$. Therefore, at least two of $(x,y_1)_y,~(x,y_2)_y$ and $(x,y_3)_y$ are bounded above by $C$. Without loss of generality, we assume that $(x,y_1)_y\le C$ and $(x,y_2)_y\le C$.
		
		Since the arclength parametrization of $[y_1,y]\cup[y,y_2]$ is a $K$-quasigeodesic, by the moreover part of Lemma \ref{lem-small inner prod}, we have $(y_1,y_2)_y\le C$.
		
		Therefore, we fix $C_{\ref{lem-finding ind bary}}(\dl,L)=C=D_{\ref{lem-small inner prod}}(\dl,K)$. This completes the proof of $(1)$.
		
		$(2)$ Replace $x$ in the proof of $(1)$ by $y_l$, where $l \in \{1,2,3\} \setminus \{i,j\}$. The same inequality $(*)$ then holds for $l$. Varying $i, j, l$ over $\{1,2,3\}$ with $i \neq j \neq l\ne i$, we conclude $(2)$.
	\end{proof}
	%
	
	In the following Theorem, $T_3$ denotes the trivalent tree.
	\begin{theorem}\label{thm-emb of T3}
		Let $\dl\ge0$ and $L\ge0$. Then there exists a constant $K_{\ref{thm-emb of T3}}=K_{\ref{thm-emb of T3}}(\dl,L)\ge1$, depending only on $\dl$ and $L$, such that the following hold. 
		
		Suppose $X$ is a proper, $\dl$-hyperbolic geodesic metric space such that the barycenter map $Bary_X:\pa^3 X\map X$ is $L$-coarsely surjective. Then we have the following.
		
		\begin{enumerate}
			\item Given $x_0\in X$, there is a $K_{\ref{thm-emb of T3}}$-qi embedding $\phi:T_3\map X$ such that $x_0=\phi(v_0)$ for some $v_0\in V(T_3)$ and $\phi(u)\ne\phi(v)$ for all distinct vertices $u,v\in T_3$.
			
			\item Moreover, for all distinct vertices $u,v\in T_3$, we have $$Bary_X^{-1}(\phi(u))\cap Bary_X^{-1}(\phi(v))=\emptyset.$$ In other words, images of the vertices of $T_3$ can be realized as distinct barycenters in $X$.
		\end{enumerate}
	\end{theorem}

	\begin{proof}
Let us define some constants depending only on $\dl$ and $L$ as follows:
		
		\begin{itemize}
			\item $C=C_{\ref{lem-finding ind bary}}(\dl,L)$,
			
			\item $K=K_{\ref{local to global}}(\dl,C)\ge1$,
			
			\item $D=D_{\ref{local to global}}(\dl,C)+K^2+K.R_{\ref{barycenter}}(\dl)+1$, and
			
			\item $K_{\ref{thm-emb of T3}}(\dl,L):=D$.
		\end{itemize}
		Note that the constants $C$, $K$, $D$ and $K_{\ref{thm-emb of T3}}(\dl,L)$ depend only on $\delta$ and $L$. (Moreover, if we trace back the definitions of these constants, we observe that $D > L$.)
		
		Let $v_0$ be a root vertex of $T_3$. Let us fix some notations: $B_{T_3}(v_0,n):=\{p\in T_3:d_{T_3}(v_0,p)\le n\}$ and $S_{T_3}(v_0,n):=\{v\in V(T_3):d_{T_3}(v_0,v)=n\}$ where $n\in\N$. We put the path metric on $B_{T_3}(v_0,n)$ induced from $T_3$. Then the inclusion $B_{T_3}(v_0,n)\map T_3$ is an isometric embedding.\smallskip
		
		We will define the map $\phi:T_3\map X$ inductively. At $n^{\text{th}}$-step, we define $\phi_n:B_{T_3}(v_0,n)\map X$ such that $\phi_n|_{B_{T_3}(v_0,n-1)}=\phi_{n-1}$. 
		
		Define $\phi_0(v_0)=x_0$. 
		
\noindent\underline{Initial step}: We will define a map $\phi_1:B_{T_3}(v_0,1)\map X$ such that $\phi_1|_{\{v_0\}}=\phi_0$. For that first we define $\phi_1(v_0)=\phi_0(v_0)=x_0$. Let $S_{T_3}(v_0,1)=\{a_1,a_2,a_3\}$. By Lemma \ref{lem-finding ind bary} $(2)$, we have $b_1,b_2,b_3\in X$ such that $d_X(x_0,b_i)=D$ for all $i\in\{1,2,3\}$, and $(b_i,b_j)_{x_0}\le C$ for all distinct $i,j\in\{1,2,3\}$.  For all $i\in\{1,2,3\}$, we set $\phi_1(a_i)=b_i$ and define that the edge $[v_0,a_i]_{T_3}$ is mapped linearly onto a geodesic segment $[x_0,b_i]_X$ under $\phi_1$. In other words, $\phi_1$ satisfies the following properties. We emphasize them here, as they will serve as the induction hypotheses for $\phi_n$ below.
		
\begin{enumerate}
\item $d_X(\phi_1(v_0),\phi_1(a_i))=D$ for all $i\in\{1,2,3\}$.

\item $(\phi_1(a_i),\phi_1(a_j))_{\phi_1(v_0)}\le C$ for all distinct $i,j\in\{1,2,3\}$.
\end{enumerate}


\noindent\underline{Induction hypotheses}: Let $n\in\N$. Suppose that we have defined $\phi_n:B_{T_3}(v_0,n)\to X$.
We also require that $\phi_n$ satisfy two properties analogous to those established in the initial step of the induction. To state these properties precisely, we first fix the following notation.

Let $S_{T_3}(v_0,n)=\{w_1,w_2,\dots,w_l\}$. (Note that $l=3.2^{n-1}$, $n\ge1$ and the cardinality of $S_{T_3}(v_0,n+1)$ is $2l$.) Let $S_{T_3}(v_0,n+1)=\{u_1,v_1,u_2,v_2,\dots,u_l,v_l\}$ be such that $d_{T_3}(w_i,u_i)=d_{T_3}(w_i,v_i)=1$ for all $i\in\{1,\dots,l\}$.

 Let $\phi_n(w_i)=x_i$ for all $i\in\{1,\dots,l\}$. Let $w'_i\in V(B_{T_3}(v_0,n))$ be the adjacent vertex to $w_i$ for all $i\in\{1,\dots,l\}$. (Note that for any $w'_i$ there exists exactly one $j\ne i$ such that $w'_i=w'_j$. For ease of exposition, we are not making this distinction.) Then we have the following.
		\begin{enumerate}
			\item Let $[u,v]_{T_3}$ be an edge joining $u,v\in V(B_{T_3}(v_0,n))$. Then $[u,v]_{T_3}$ is mapped linearly onto a geodesic segment $[\phi_n(u),\phi_n(v)]_X$ such that $d_X(\phi_n(u),\phi_n(v))=D$.
			
			\item Let $u,v,w\in V(B_{T_3}(v_0,n))$ be such that $u$ and $v$ are distinct vertices adjacent to $w$. Then $(\phi_n(u),\phi_n(v))_{\phi_n(w)}\le C$.
		\end{enumerate} 
		
		
		\noindent\underline{Induction step}: Now we will define a map $\phi_{n+1}:B_{T_3}(v_0,n+1)\map X$. We define $\phi_{n+1}|_{B_{T_3}(v_0,n)}=\phi_{n}$. Now we will apply Lemma \ref{lem-finding ind bary} $(1)$ at the point $x_i$ for all $i\in\{1,\dots,l\}$.  Let $\phi_n(w'_i)=x'_i$ for all $i\in\{1,\dots,l\}$. (Note that $d_X(x_i,x'_i)=D$.) By Lemma \ref{lem-finding ind bary} $(1)$, we have $y_i,z_i\in X$ such that $d_X(x_i,y_i)=D$, $d_X(x_i,z_i)=D$ and $$max\{(y_i,z_i)_{x_i},(y_i,x'_i)_{x_i},(z_i,x'_i)_{x_i}\}\le C \text{ for all }i\in\{1,\dots,l\}.$$
		
		\noindent We define $\phi_{n+1}(u_i)=y_i$ and $\phi_{n+1}(v_i)=z_i$. We define that the edge $[w_i,u_i]_{T_3}$ (respectively, $[w_i,v_i]_{T_3}$) is mapped linearly onto a geodesic segment $[x_i,y_i]_X$ (respectively, $[x_i,z_i]_X$).
		
		\noindent Since $\phi_{n+1}|_{B_{T_3}(v_0,n)}=\phi_n$, and by the construction in the induction step, we conclude that the properties $(1)$ and $(2)$ above are satisfied by  $\phi_{n+1}$ for all vertices in $B_{T_3}(v_0,n+1)$.
		
		Therefore, (by induction) we define a map $\phi=\bigcup_{n\in\N\cup\{0\}}\phi_n:T_3\map X$ as follows. Let $v\in T_3$, and $v\in B_{T_3}(v_0,n)$ for some $n\in\N\cup\{0\}$. Then $\phi(v)=\phi_n(v)$.\smallskip
		
		Now we will show that $\phi$ is $D$–qi embedding, and in the process, establish the remaining statements of the proposition.
		
		Let $u,v\in V(T_3)$, and $u,v\in B_{T_3}(v_0,n)$ for some $n\in\N$. Let $d_{T_3}(u,v)=m$. Let $u=p_0,p_1\dots,p_m=v$ be the consecutive vertices on the geodesic segment $[u,v]_{T_3}$. By our construction of $\phi_n$, we have 
		
		\begin{itemize}
			\item $d_X(\phi_n(p_{i-1}),\phi_n(p_i))=D\ge D_{\ref{local to global}}(\dl,C)$ for all $i\in\{1,\dots,m\}$, and
			
			\item  $(\phi_n(p_{i-1}),\phi_n(p_{i+1}))_{\phi_n(p_i)}\le C$ for all $i\in\{1,\dots,m-1\}$.
		\end{itemize}
		Hence by Lemma \ref{local to global}, the arclength parametrization of the concatenation $$[\phi_n(p_0),\phi_n(p_1)]_X*\dots*[\phi_n(p_{m-1}),\phi_n(p_m)]_X$$ is a $K$-quasigeodesic. Thus  $$\frac{1}{K}mD-K\le d_X(\phi_n(u),\phi_n(v))\hspace{5mm}(*)$$
		
		\noindent On the other hand, by triangle inequality, we have $d_X(\phi_n(u),\phi_n(v))\le Dm=Dd_{T_3}(u,v)$. Note that $D>1$, $D>K$ and $\phi_n(u)=\phi(u),~\phi_n(v)=\phi(v)$. Therefore, combining the inequality $(*)$, we have $$\frac{1}{D}d_{T_3}(u,v)-D\le d_X(\phi(u),\phi(v))\le Dd_{T_3}(u,v)+D\hspace{3mm} (**)$$ Note that the edges of $T_3$ are mapped linearly onto geodesic segments of length $D$. For any two points in $T_3$ -- not necessarily vertices -- a similar calculation verifies the inequality $(**)$, establishing that $\phi$ is $D$-qi embedding.

		Moreover, $\phi(u)\ne\phi(v)$ for all distinct vertices $u, v\in V(T_3)$. Indeed, let $u,v\in B_{T_3}(v_0,n)$ for some $n\in\N$. If $\phi(u)=\phi(v)$ then $\phi_n(u)=\phi_n(v)$, and by the above inequality $(*)$, we have $\frac{1}{K}D-K\le\frac{1}{K}d_{T_3}(u,v)-K\le 0$, and so $D\le K^2$ $-$ contradicting to the choice of $D$. This completes the proof of $(1)$.

		To prove $(2)$, we only need to show that $d_X(\phi(u),\phi(v))>R_{\ref{barycenter}}(\dl)$ for all distinct vertices $u,v\in T_3$ (see Lemma \ref{barycenter}). Indeed, let $u,v\in B_{T_3}(v_0,n)$ for some $n\in\N$ and $d_{T_3}(u,v)=m\in\N$. If $d_X(\phi(u),\phi(v))\le R_{\ref{barycenter}}(\dl)$, then by the above inequality $(*)$, we will have $$\frac{1}{K}mD-K\le R_{\ref{barycenter}}(\dl).$$ This implies that $D\le mD\le KR_{\ref{barycenter}}(\dl)+K^2$ $-$ contradicting to the choice of $D$. This completes the proof of $(2)$.
	\end{proof}

	\subsection{Growth}\label{subsec-expo growth}
	Growth of groups is a well-studied topic in geometric group theory (see \cite{mann-growbook}, \cite{grigor-intermegrowth-intro}). In this subsection, we study the growth of metric graphs with bounded valence (see Definition \ref{defn-valence}), which need not be Cayley graphs of groups. Although groups can exhibit various types of growth, we focus exclusively on exponential growth of metric graphs, as this is the primary focus of the article. We also mention that the definition of exponential growth (Definition \ref{defn-exponential growth}) coincides with that for groups (see \cite[Section $2.2$]{mann-growbook}).
	
	\begin{defn}\label{defn-valence}
For a metric graph $\Gamma$, the valence at a vertex $u\in V(\Gamma)$ is the number of edges adjacent to $u$. We say that the graph $\Gamma$ has bounded valence if there is a constant $D\ge0$ such that the valence at each vertex of $\Gamma$ is bounded above by $D$.
		
We say that a collection of metric graphs $\{\Gamma_{\lm}:\lm\in\Lambda\}$ have uniformly bounded valence if there is a constant $D\ge0$ such that for every $\lm\in\Lambda$, the graph $\Gamma_{\lm}$ has bounded valence with the constant $D$ mentioned above.
	\end{defn}
	
\begin{remark}
Suppose $\Gamma$ is a metric graph with bounded valence. Further, assume that $D$ is the constant such that the valence at each vertex of $\Gamma$ is bounded above by $D$. Then for every vertex $u$ of $\Gamma$, the number of vertices adjacent to $u$ is bounded above by $D$. In what follows, we focus on counting the number of vertices adjacent to a given vertex.
\end{remark}	
	
	
Let $\Gamma$ be a metric graph $\Gamma$ and $u\in V(\Gamma)$. Let $n\in\N$. We denote the ball of radius $n$ at $u$ by $B(u,n)$ and it is the set $\{v\in\Gamma:~d_{\Gamma}(u,v)\le n\}$. Let $||B(u,n)||$ denote the number of vertices in $B(u,n)$.
	
	\begin{defn}[Exponential growth]\label{defn-exponential growth}
Suppose $\Gamma$ is a metric graph of bounded valence. We say that $\Gamma$ has {\em exponential growth witnessed by a pair $(a,b)$} if $a>0$, $b>1$, and $$||B(u,n)||\ge ab^n$$for every vertex $u\in\Gamma$ and every $n\in\mathbb N$. We say that $\Gamma$ has {\em exponential growth} if it has exponential growth witnessed by a pair $(a,b)$ for some $a>0$ and $b>1$. Note that in the pair $(a,b)$, the constant $b$ appears in the exponential factor.

Let $\{\Gamma_{\lambda}:\lambda\in\Lambda\}$ be a collection of metric graphs with uniformly bounded valence. We say that $\{\Gamma_{\lambda}:\lambda\in\Lambda\}$ has {\em uniform exponential growth} if there exist constants $a>0$ and $b>1$ such that the exponential growth of every $\Gamma_{\lambda}$ is witnessed by the same pair $(a,b)$.
\end{defn}

It is very well possible that $\Gamma$ has exponential growth witnessed by a pair $(a,b)$  as well as by a pair $(a',b')$ where $a'>a$ and $b'>b$. Thus witnessing pair gives only a blower bound of the exponential growth.

	\begin{example}\label{exmp-T_3 expo}
Suppose $T_i$ denotes the (simplicial) metric $m$-valent tree where $m\in\N$ (i.e., the valence at each vertex is exactly $m$). One can easily check that $T_3$ has exponential growth witnessed by a pair $(3/2,2)$. That is, the number of vertices in a ball of radius $n\in\N$ in $T_3$ is at least $\frac{3}{2}2^n$.
	\end{example}
	
	
	\begin{lemma}\label{lem-qi implies exponential growth}
Let $a>0,~b>1,~k\ge1$ and $D\ge2$. Then we have constants $a_{\ref{lem-qi implies exponential growth}}=a_{\ref{lem-qi implies exponential growth}}(a,b,k,D)>0$ and $b_{\ref{lem-qi implies exponential growth}}=b_{\ref{lem-qi implies exponential growth}}(k,b)=b^{\frac{1}{k+2}}>1$ such that the following holds.
		
Suppose $(\Gamma_i, d_i)$ is a metric graph for $i = 1, 2$, and let $\phi : \Gamma_1 \to \Gamma_2$ be a $k$-qi embedding. Assume that the valence at each vertex of $\Gamma_1$ is bounded above by $D$, and that $\Gamma_2$ is proper. Further, assume that $\Gamma_1$ has exponential growth witnessed by a pair $(a,b)$. Let $v\in V(\Gamma_2)$ be such that $d_2(\phi(u),v)\le k$ for some $u\in V(\Gamma_1)$. Then for all $m\in\N$, we have $$||B_2(v,m)||\ge a_{\ref{lem-qi implies exponential growth}}b_{\ref{lem-qi implies exponential growth}}^m$$where $B_2(v,m)$ denotes the $m$-radius ball centered at $v$ in $\Gamma_2$.
	\end{lemma}
	
	\begin{proof}
From $\phi$, we first define a qi embedding map $\phi':\Gamma_1\to\Gamma_2$ which sends vertices to vertices. For $x\in\Gamma_1$, define
\[
\phi'(x) =
\begin{cases}
	\phi(x) & \text{if } x \in\Gamma_1 \setminus V(\Gamma_1), \\
	\phi(x) & \text{if } x \in V(\Gamma_1) \text{ and } \phi(x) \in V(\Gamma_2), \\
	\text{a vertex of the edge containing } \phi(x) & \text{otherwise}.
\end{cases}
\]
Note that for any $x\in\Gamma_1$, $d_2(\phi(x),\phi'(x))\le1$ as the edges of $\Gamma_2$ have length $1$. Let $\lm=k+2$. Hence, $\phi'$ is a $\lm$-qi embedding as $\phi$ is $k$-qi embedding. Moreover, $\phi'$ sends vertices to vertices .

Now we will prove the lower bound for the growth of $\Gamma_2$ as claimed in the statement. Let $m\in\N$, $m\ge3\lm$ and $n=[\frac{m-2\lm}{\lm}]\in\N$ where $[r]$ denotes the greatest integer not greater than $r\in\R$. Let $d=[\lm^2]+1$. Let $B_1(u,n)$ denote the $n$-radius ball centered at $u$ in $\Gamma_1$. Let $A$ be a maximal $d$-separated subset of $V(B_1(u,n))$ (with respect to inclusion). Then for all $p\in B_1(u,n)$, we have $d_1(p,x)\le d+1$ for some $x\in A$.
		
By the assumptions we have $||B_1(u,n)||\ge ab^n$ and $||B_1(p,d+1)||\le D^{d+1}$ for any $p\in V(\Gamma_1)$. Again $B_1(u,n)\sse \bigcup_{x\in A}B_1(x,d+1)$ implies $ab^n\le ||B_1(u,n)||\le ||A|| D^{d+1}$. Thus $||A||\ge aD^{-(d+1)}b^n$.
		
Note that, as $\phi'$ is $\lm$-qi embedding, $B_1(u,n)$ is mapped under $\phi'$ in $B_2(\phi'(u),n\lm+\lm)$. Since $d_2(\phi(u),v)\le k$ implies $d_2(\phi'(u),v)\le k+1\le\lm$, we have $\phi'(B_1(u,n))\sse B_2(v,n\lm+2\lm)$, and so $\phi'(B_1(u,n))\sse B_2(v,m)$ as $m\ge n\lm+2\lm$. On the other hand, for all distinct $x,x'\in A$, we have $\phi'(x)\ne\phi'(x')$. Indeed, suppose $\phi'(x) = \phi'(x')$. Then, since $d_1(x, x') \ge d$, the $\lm$-qi embedding condition gives $\frac{1}{\lm}d-\lm\le 0$, which implies $d \le \lm^2$ -- contradicting to the choice of $d$. 
		
Hence $||B_2(v,m)||\ge||A||\ge aD^{-(d+1)}b^n\ge aD^{-(d+1)}b^{\frac{m-2\lm}{\lm}-1}=(aD^{-([\lm^2]+2)}b^{-3})(b^{1/\lm})^m$  as $d=[\lm^2]+1$.

Note that $\Gamma_2$ is connected. Let $t\in\N$ and $t< 3\lm$. Then $||B_2(v,t)||\ge t+1\ge 1\ge b^{t-3\lm}\ge b^{-3\lm}b^t\ge b^{-3\lm}(b^{1/\lm})^t$ as $\lm\ge1$ and $b\ge1$. Now we fix $a_{\ref{lem-qi implies exponential growth}}(a,b,k,D)=min\{aD^{-([\lm^2]+2)}b^{-3},b^{-3\lm}\}>0$ and $b_{\ref{lem-qi implies exponential growth}}(b,k)=b^{1/\lm}=b^{\frac{1}{k+2}}>1$. Therefore, combining both the inequalities, for all $m\in\N$, we have $||B_2(v,m)||\ge a_{\ref{lem-qi implies exponential growth}}b_{\ref{lem-qi implies exponential growth}}^m$. This completes the proof.
	\end{proof}
	As a consequence of Lemma \ref{lem-qi implies exponential growth}, we have the following corollary. It says that the exponential growth of metric graphs is quasiisometry invariant.
	\begin{cor}\label{cor-exponential growth QI inv}
Suppose $(\Gamma_i, d_i)$ is a metric graph for $i = 1, 2$, and let $\phi : \Gamma_1 \to \Gamma_2$ be a quasiisometry. Assume that each $\Gamma_i$ has bounded valence. Then $\Gamma_1$ has exponential growth if and only if $\Gamma_2$ has exponential growth.
		
Moreover, the exponential growth of $\Gamma_1$ is witnessed by a pair depending only on a pair witnessing the exponential growth of $\Gamma_2$, the quasiisometry constants and the constant that bounds the valence of each vertex of $\Gamma_2$ from above.
	\end{cor}

	Since Cayley graphs corresponding to different finite generating sets of a group are quasiisometric to each other, the above corollary implies that the notion of exponential growth is a well-defined property of groups.

	\begin{theorem}\textup{(\cite[Theorem 1.1]{malik-koubi})}\label{thm-exp growth rate hyp grp}
		Suppose $G$ is a nonelementary hyperbolic group. For any finite generating set $S$ of $G$, the Cayley graph of $G$ with respect to $S$ has exponential growth.
	\end{theorem}
	
The following result (Theorem \ref{thm-bary imp expo growth}) provides a sufficient condition under which a hyperbolic metric graph (not necessarily a Cayley graph of a hyperbolic group) exhibits exponential growth. We believe that this condition may characterize exponential growth, but we will not pursue that direction here, as it may be of independent interest. (One needs to keep in mind that, a pair witnessing exponential growth must be independent of the choice of base point.) 
Note that Theorem \ref{thm-bary imp expo growth} generalizes Theorem \ref{thm-exp growth rate hyp grp} to any (bounded valence) metric graph and it also provides a different approach of proving Theorem \ref{thm-exp growth rate hyp grp}.
	
	\begin{theorem}\label{thm-bary imp expo growth}
		Suppose $\Gamma$ is a hyperbolic metric graph and the barycenter map $\pa^3\Gamma\map\Gamma$ is coarsely surjective. Further, suppose that $\Gamma$ has bounded valence. Then $\Gamma$ has exponential growth.
		
Moreover, the exponential growth of $\Gamma$ is witnessed by a pair depending only on the hyperbolicity constant of $\Gamma$ and the coarse surjectivity constant of the barycenter map. 
	\end{theorem}
	
\begin{proof}
Let $u\in V(\Gamma)$. Suppose $\Gamma$ is $\dl$-hyperbolic and the barycenter map $\pa^3\Gamma\map\Gamma$ is $L$-coarsely surjective, for some $\dl\ge0$ and $L\ge0$. Let $k=K_{\ref{thm-emb of T3}}(\dl,L)\ge1$ be the constant depending only on $\dl$ and $L$, as appearing in Theorem \ref{thm-emb of T3}. Let $T_3$ denote the trivalent tree with a root vertex $u_0$. Let $u\in V(\Gamma)$.
		
Note that $\Gamma$ is a proper metric graph. Then by Theorem \ref{thm-emb of T3}, we have a $k$-qi embedding $\phi:T_3\map\Gamma$ such that $\phi(u_0)=u$. Then by Example \ref{exmp-T_3 expo} and Lemma \ref{lem-qi implies exponential growth}, we have $a>0$ and $b>1$ depending only on $k$ such that $$||B_{\Gamma}(u,n)||\ge ab^n\text{ for all }n\in\N$$ where $B_{\Gamma}(u,n)$ denotes the $n$-radius ball centered at $u$ in $\Gamma$. Note that $u\in V(\Gamma)$ is arbitrary. Therefore, $\Gamma$ has exponential growth witnessed by a pair $(a,b)$ where $a$ and $b$ depend only on $k$, hence, only on $\dl$ and $L$. This completes the proof.
	\end{proof}
	\section{Metric graph bundles}\label{metric graph bundle sec}
	
In this section, we recall relevant results on metric graph bundles from \cite{pranab-mahan}. Following \cite[Proposition 6.6]{ps-krishna}, we give a detailed description of the boundary of a metric graph bundle over $[0,\infty)$ in Proposition \ref{description of pa X}. We also establish several related results needed in the next section. In particular, the notion of the `flow' of a subset (of a fiber) is introduced in Subsection \ref{subsec-flow}.	
	
	\begin{defn}\label{defn-proper embedding}
Suppose $\phi:Y\map X$ is a map between metric spaces and $f:\R_{\ge0}\map\R_{\ge0}$ is a map. We say that $\phi$ is metrically $f$-proper (or simply $f$-proper) if for all $y,y'\in Y$ and $r\in\R_{\ge0}$ with $d_X(\phi(y),\phi(y'))\le r$ we have $d_Y(y,y')\le f(r)$.
	\end{defn}

	\begin{defn}\textup{(\cite[Definition 1.5]{pranab-mahan})}\label{defn-metric graph bundle}
		Suppose $X$ and $B$ are metric graphs, and $f:\N\map\N$ is a map such that $f(n)\map\infty$ as $n\map\infty$. We say that $X$ is an $f$-metric graph bundle over $B$ if there is a simplicial, surjective (and $1$-Lipschitz) map $\pi:X\map B$ such that the following hold.
		\begin{enumerate}
			\item For all $b\in V(B)$, $F_b:=\pi^{-1}(b)$, called fiber, is a connected subgraph of $X$ with respect to the path metric $d_b$ from $X$. The inclusion maps $(F_b,d_b)\map X$ are $f$-proper embedding.
			
			\item Let $b_1,b_2\in V(B)$ be adjacent vertices, and let $x\in V(F_{b_1})$. Then  $x$ is connected by an edge in $X$ to a point in $V(F_{b_2})$.
		\end{enumerate}
		
		Most of the time, we say that $\pi:X\map B$ is an $f$-metric graph bundle. Sometimes, we say that $X$ is a metric graph bundle over $B$ by making $f$ implicit.
	\end{defn}
	
	
	\begin{lemma}\textup{(\cite[Proposition 1.7]{pranab-mahan})}\label{lem-natural map QI}
		Suppose $\pi:X\map B$ is an $f$-metric graph bundle. 
		Let $b_1,b_2\in V(B)$ such that $d_B(b_1,b_2)=1$. Let $\phi:F_{b_1}\map F_{b_2}$ be any map that sends $x\in V(F_{b_1})$ to $\phi(x)\in V(F_{b_2})$ such that $x$ and $\phi(x)$ are joined by an edge in $X$. Then $\phi$ is a $k_{\ref{lem-natural map QI}}(f)$-quasiisometry where $k_{\ref{lem-natural map QI}}(f)\ge1$ is a constant depending only on the function $f$.
	\end{lemma}
	
	\begin{defn}[Quasiisometric (in short, qi) lift]\label{qi lift}
		Let $k\ge1$. Suppose $\pi:X\map B$ is an $f$-metric graph bundle. Let $\gm:I\sse\R\map B$ be a geodesic joining two vertices in $B$. By a $k$-qi lift of $\gm$ we mean a $k$-qi embedding $\tilde{\gm}:I\map X$ such that $\pi\circ\tilde{\gm}=\gm$ on $I\cap\Z$.
		
	\end{defn}
	
In a hyperbolic metric space, (quasi)geodesics diverge exponentially. In an $f$-metric graph bundle $\pi:X\map B$, qi lifts are quasigeodesics. So, they diverge exponentially provided $X$ is hyperbolic. This property is captured in Definition \ref{flaring condition} for special types of quasigeodesics, namely, qi lifts. This definition is a generalization of the Bestvina--Feighn hallway flaring condition (\cite{BF}) in a natural way for metric graph bundles.
	
	\begin{defn}\textup{(\cite[Definition 1.12]{pranab-mahan})}\label{flaring condition}
		Suppose $\pi:X\map B$ is an $f$-metric graph bundle. Let $k\ge1$. We say that $\pi:X\map B$ satisfies a $k$-flaring condition if there are constants $M_k>0$ and $n_k\in\N$ depending on $k$, and $\lm>1$ such that for any geodesic $\gm:[-n_k,n_k]\map B$ joining two vertices and for any two $k$-qi lifts $\gm_0,~\gm_1$ of $\gm$ with $d_{\gm(0)}(\gm_0(0),\gm_1(0))>M_k$ we have $$\lm~d_{\gm(0)}(\gm_0(0),\gm_1(0))<\text{ max }\{d_{\gm(n_k)}(\gm_0(n_k),\gm_1(n_k)),d_{\gm(-n_k)}(\gm_0(-n_k),\gm_1(-n_k))\}.$$
	\end{defn}

	\begin{defn}\textup{(\cite[Definition $2.4$]{NirMaMj-commen})}\label{controlled fibers defn}
		We say that an $f$-metric graph bundle $\pi:X\map B$ has controlled hyperbolic fibers if there are $\dl\ge0$ and $L\ge0$ such that for all $b\in V(B)$,
		
		\begin{enumerate}
			\item $F_b$ is $\dl$-hyperbolic geodesic metric graph, and
			
			\item the barycenter map $\pa^3F_b\map F_b$ is $L$-coarsely surjective $($i.e., $L$-neighbourhood of the image is $F_b)$.
		\end{enumerate} 
		
		The constants $\dl$ and $L$ are called parameters.
	\end{defn}
	
	\begin{remark}
		Here in Definition \ref{controlled fibers defn}, one does not require $F_b$ to be a proper metric space. In that case, the barycenter map is defined for an ideal quasigeodesic triangle instead ideal geodesic triangle (see \cite[Lemma $2.7$]{pranab-mahan}).
	\end{remark}

	\begin{theorem}\label{thm-metric bundle comb}
		Suppose $\pi:X\map B$ is an $f$-metric graph bundle. Then we have the following.
		
		\begin{enumerate}
			\item \textup{(\cite[Proposition $5.8$]{pranab-mahan})} Suppose that the fibers and the total space $X$ are uniformly hyperbolic. Then $\pi:X\map B$ satisfies the $k$-flaring condition for all $k\ge1$.
			
			\item \textup{(\cite[Theorem $4.3$]{pranab-mahan})} Suppose the metric graph bundle $\pi : X \to B$ has controlled hyperbolic fibers and satisfies the $k$-flaring condition for all $k \ge 1$. Then $X$ is hyperbolic if and only if $B$ is.
		\end{enumerate}
	\end{theorem}
	
	\begin{remark}\label{crucial remark}
		Suppose $\pi:X\map B$ is an $f$-metric graph bundle.
		
		\noindent$(1)$ Let $A$ be any (connected) subgraph of $B$. Let $\pi^{-1}(A)=:X_A$. We consider both $(A,d_A)$ and $(X_A,d_{X_A})$ equipped with their path metrics induced from $B$ and $X$ respectively. Then for all $x,x'\in X_A$, we have $d_X(x,x')\le d_{X_A}(x,x')$. Hence the restriction of $\pi$ to $X_A$, that is, $\pi|_{X_A}:X_A\map A$ is also an $f$-metric graph bundle.
		
		\noindent$(2)$ \textup{(\cite[Remark 4.4]{pranab-mahan})} 
		Suppose $\pi: X \to B$ satisfies Theorem \ref{thm-metric bundle comb} $(2)$. Moreover, let $A$ be a qi embedded subgraph of $B$. Then the restricted bundle $\pi|_{X_A}:X_A\map A\sse B$ satisfies the $k$-flaring condition for all $k\ge1$, possibly with different constants. (This follows essentially from the bounded flaring condition (\cite[Corollary $1.14$]{pranab-mahan}) and the flaring condition for $\pi:X\map B$.) Note that $A$ is hyperbolic. Since the properties of fibers remain unchanged, the restricted bundle $\pi|_{X_A}:X_A\map A\sse B$ continues to have controlled hyperbolic fibers. Thus, it follows from Theorem \ref{thm-metric bundle comb} $(2)$ that $X_A$ is hyperbolic.

	\end{remark}

	\subsection{Metric graph bundles over geodesic rays}
	In this subsection, we will work with the metric graph bundle over $[0,\infty)$ such that fibers are {\em proper metric spaces}. Hence, the total space is also proper. The properness of fibers is needed in Propositions \ref{mild imp special} and \ref{description of pa X}. \smallskip
	
	{\em \bf Convention}: {\em Let $\pi:X\map[0,\infty)$ be an $f$-metric graph bundle. We will denote the induced path metric on $F_i=\pi^{-1}(i)$ by $d_i$ for all $i\in\N\cup\{0\}$. For all $i\in\N\cup\{0\}$, we denote a natural quasiisometry map $F_i\map F_{i+1}$ by $\phi_i$ as defined in Lemma \ref{lem-natural map QI}.} 
	
	%
	%
	%
	%
	%
	%
	%
	%
	\begin{remark}\label{c-qi lift}
Let $\pi:X\map [0,\infty)$ be an $f$-metric graph bundle. Then for any $x\in V(X)$, there is a $(1,0)$-qi section (i.e., isometric section) over $[0,\infty)$ through $x$. In particular, there is a $1$-qi section over $[0,\infty)$ through each point.
		
	\end{remark}
	
	Remark \ref{c-qi lift} states that the existence of a qi section follows directly from the definition of a metric graph bundle when the base is $[0,\infty)$. This contrasts with the more intricate construction in \cite{pranab-mahan}, where the base is an arbitrary hyperbolic space. There, the authors constructed a qi section by fixing an ideal triangle in a fiber, flowing it across all fibers, and taking barycenters of the resulting ideal triangles in each fiber -- see \cite[Propositions $2.10$ and $2.12$]{pranab-mahan}. In this paper, such qi sections over $[0,\infty)$ will be referred to as {\em good qi sections} (see Definition \ref{good qi section}). The following result is not stated in the same form as \cite[Proposition $2.10$]{pranab-mahan}, but its content is essentially derived from the proof of that proposition.
	
	\begin{lemma}\textup{(\cite[Propositions $2.10$]{pranab-mahan})}\label{for special qi sec}
		Suppose $\pi:X\map[0,\infty)$ is an $f$-metric graph bundle such that the fibers are all $\dl$-hyperbolic for some $\dl\ge0$. Let $i\in\N\cup\{0\}$ and $\xi_i=(\xi_{1,i},\xi_{2,i},\xi_{3,i})\in\pa^3F_i$ such that $\pa\phi_{i}(\xi_{j,i})=\xi_{j,i+1}$ for $j\in\{1,2,3\}$ where $\phi_i:F_{i}\map F_{i+1}$ is a natural quasiisometry. Define $\Sigma:[0,\infty)\map X$ such that $\Sigma(i)= Bary_{F_i}(\xi_i)$ for all $i\in\N\cup\{0\}$ and for all $t\in[i,i+1)$, $\Sigma(t)=Bary_{F_i}(\xi_i)$.
		
		Then $\Sigma$ is a $K_{\ref{for special qi sec}}(\dl,f)$-qi section for some constant $K_{\ref{for special qi sec}}(\dl,f)\ge1$ depending only on $\dl$ and $f$.
	\end{lemma}
	
%
%
	The following definition is motivated by Lemma \ref{for special qi sec}. 
	
	\begin{defn}\label{good qi section}
		Suppose $\pi:X\map[0,\infty)$ is an $f$-metric graph bundle with uniformly hyperbolic fibers. Let $K\ge1$. Let $\Sigma$ be a $K$-qi section over $[0,\infty)$. We say $\Sigma$ is a $K$-good qi section if $\Sigma(i)$ is a $K$-barycenter of the ideal triangle $\triangle(\xi_{1,i},\xi_{2,i},\xi_{3,i})$ in $F_i$ for all $i\in\N\cup\{0\}$ and $\pa\phi_i(\xi_{j,i})=\xi_{j,i+1}$ for $j\in\{1,2,3\}$.
		
		
We say that $\gm$ is a $K$-good qi lift of some interval $[m,n]$ where $m,n\in\N\cup\{0\}$ if it is the restriction of a $K$-good qi section over $[0,\infty)$.
	\end{defn}
	As a corollary of Lemma \ref{for special qi sec}, we have the following.
	\begin{cor}\label{cor-existence good qi}
		Suppose $\pi:X\map[0,\infty)$ is an $f$-metric graph bundle with controlled hyperbolic fibers with parameters $\dl\ge0$ and $L\ge0$. Then through each point of $X$ there is a $K_{\ref{cor-existence good qi}}$-good qi section where $K_{\ref{cor-existence good qi}}=K_{\ref{cor-existence good qi}}(\dl,L,f)=K_{\ref{for special qi sec}}(\dl,f)+L+1$.
	\end{cor}
	
	\begin{lemma}\textup{(\cite[Lemma $3.1$]{pranab-mahan},~Remark \ref{rmk-not needed bary})}\label{lem-qi section in ladder}
		Suppose $\pi:X\map[0,\infty)$ is an $f$-metric graph bundle such that the fibers are all $\dl$-hyperbolic for some $\dl\ge0$. Then given any $k'\ge1$ there is $k_{\ref{lem-qi section in ladder}}=k_{\ref{lem-qi section in ladder}}(\dl,k',f)\ge1$ depending on $\dl$, $k'$ and $f$ such that the following holds.
		
		Let $\Sigma$ and $\Sigma'$ be two $k'$-qi sections over $[0,\infty)$. Let $x\in[\Sigma(i),\Sigma'(i)]_{F_i}$ for some $i\in\N\cup\{0\}$. Then there is a $k_{\ref{lem-qi section in ladder}}$-qi section, say $\Sigma''$, through $x$ over $[0,\infty)$ such that $\Sigma''(i)\in[\Sigma(i),\Sigma'(i)]_{F_i}$ for all $i\in\N\cup\{0\}$.
	\end{lemma}
	
	\begin{remark}\label{rmk-not needed bary}
		In the proof of Lemma \ref{lem-qi section in ladder}, we do not require the metric graph bundle to have controlled hyperbolic fibers, unlike in the assumption of \cite[Lemma $3.1$]{pranab-mahan}. Since the base is $[0,\infty)$ it follows from \cite[Theorem $3.8$]{mitra-trees}. 
		%
	\end{remark}
	
	%
	
	The following result is a specialization to the metric graph bundle context of the fact that geodesics in a hyperbolic metric space diverge exponentially.
	
	
	\begin{lemma}\textup{(\cite[Lemma 5.9]{pranab-mahan})}\label{once start decreasing is decreasing}
		Let $\dl\ge0$, $k\ge 1$ and $C>0$. Let $\pi:X\map[0,\infty)$ be an $f$-metric graph bundle such that $X$ is $\dl$-hyperbolic. Then there are constants $a_{\ref{once start decreasing is decreasing}}=a_{\ref{once start decreasing is decreasing}}(\dl,k,C)>0$ depending on $\dl,~k$ and $C$, and $b=2^{\frac{1}{\dl+1}}>1$ such that the following holds.
		
		Suppose $\Sigma,\Sigma':[0,\infty)\map X$ are two $k$-qi sections such that $d_i(\Sigma(i),\Sigma'(i))>M_k$ for all $i\in\N\cup\{0\}$, where $M_k$ is coming from $k$-flaring condition $($see Definition \ref{flaring condition}$)$. Further, suppose that $d_0(\Sigma(0),\Sigma'(0))\le C$. Then $d_{i}(\Sigma(i),\Sigma'(i))\ge a_{\ref{once start decreasing is decreasing}}b^i$ for all $i\in\N\cup\{0\}$. 
	\end{lemma}

	\subsubsection{Boundary of metric graph bundle}\label{boundary des sec}
	Let us start by recalling the following result.
	
	\begin{prop}\textup{(\cite[Proposition 2.3.3]{bowditch-stacks}, see also \cite[Proposition 6.6]{ps-krishna})}\label{des pa X}
		Suppose $\pi:X\map [0,\infty)$ is an $f$-metric graph bundle such that the fibers are uniformly hyperbolic and proper metric spaces. $($Note that $X$ is also proper$.)$ Assume that $X$ is hyperbolic. Then, as a set, $\pa X=\Lambda_X(F_0)\cup\pa_{qi}X$ where $\pa_{qi} X:=\{\Sigma(\infty):\Sigma\text{ is a 1-qi section over }[0,\infty)\}$.
	\end{prop}
	
	Now we will give a more detailed description of $\pa X$ in the context of controlled hyperbolic fibers (in Proposition \ref{description of pa X}) where the qi sections are replaced by {\em good qi sections}. The idea underlying this description can be traced back to \cite[Proposition 2.3.3]{bowditch-stacks} (see also \cite[Proposition 6.6]{ps-krishna}). While the description in \cite[Proposition 6.6]{ps-krishna} applies to an arbitrary hyperbolic base, we focus on the specific case where the base is $[0,\infty)$, as this is our requirement. Let us first prove the following result that is used in the proof of Proposition \ref{description of pa X}. 
	
	
	\begin{lemma}\label{unbdd imply in fiber}
		Suppose $\pi:X\map[0,\infty)$ is an $f$-metric graph bundle such that $X$ is hyperbolic. Let $k\ge 1$ and $\{n_i\}\sse\N$. Let $\Sigma_i$ be a $k$-qi lift of $[0,n_i]$. $($We have no condition on (un)boundedness of $\{n_i\}.)$ Let $x\in F_0$ and $\lim_{i\map\infty}d_0(x,\Sigma_i(0))=\infty$. Moreover, suppose that $\lim^X_{i\map\infty}\Sigma_i(n_i)$ exists in $\pa X$. Then $\lim^{X}_{i\map\infty}\Sigma_i(0)$ exists in $\pa X$ and $\lim^X_{i\map\infty}\Sigma_i(0)=\lim^X_{i\map\infty}\Sigma_i(n_i)$ in $\pa X$.
	\end{lemma}	
	
	\begin{proof}
		First of all $\lim_{i\map\infty}d_X(x,\Sigma_i)=\infty$. Indeed, if not, let $d_X(x,x_i)\le D$ for some $x_i\in\Sigma_i$ and $D\ge0$. Then $d(0,\pi(x_i))\le D$ since $\pi$ is $1$-Lipschitz. Thus $d_X(\Sigma_i(0),x_i)\le kD+k$ since the restriction of $\Sigma_i$ over $[0,\pi(x_i)]$ is a $k$-qi embedding in $X$. Hence $d_X(x,\Sigma_i(0))\le d_X(x,x_i)+d_X(x_i,\Sigma_i(0))\le (k+1)D+k$. Since fibers are $f$-proper embedding in $X$, so $d_0(x,\Sigma_i(0))\le f((k+1)D+k)$. This contradicts to the assumption that $\lim_{i\map\infty} d_0(x,\Sigma_i(0))=\infty$.
		
		Let $\al_i$ be a geodesic joining $\Sigma_i(0)$ and $\Sigma_i(n_i)$. Then by the stability of quasigeodesic (Lemma \ref{stability of qc}), we have $\lim_{i\map\infty}d_X(x,\al_i)=\infty$. Hence by Lemma \ref{require in bary}, $\lim^X_{i\map\infty}\Sigma_i(0)$ exists in $\pa X$ and $\lim^X_{i\map\infty}\Sigma_i(0)=\lim^X_{i\map\infty}\Sigma_i(n_i)$ in $\pa X$.
	\end{proof}
	The proof of the following result is a bit different from that of \cite[Proposition $6.6$]{ps-krishna} as we want a more detailed description of $\partial X$.
	\begin{prop}\label{mild imp special}
		Suppose $\pi:X\map [0,\infty)$ is an $f$-metric graph bundle such that the fibers are $\dl$-hyperbolic for some $\dl\ge0$. Suppose that $X$ is also $\dl$-hyperbolic, and that the fibers are proper metric spaces. Let $\{p_i\}\sse\N$ be such that $\lim_{i\to\infty}p_i=\infty$. Suppose $\Sigma_{i}$ is a $K$-good qi lift of $[0,p_i]$ for some $K\ge1$ such that $diam~\{\Sigma_i(0):i\in\N\}$ is bounded.
		
		Then there is a subsequence $\{p_{i_n}\}_n\sse \{p_i\}_i$ such that $\LMX \Sigma_{i_n}(p_{i_n})$ exists in $\pa X$. Moreover, there is a $K_{\ref{for special qi sec}}(\dl,f)$-good qi section, say $\Sigma$, over $[0,\infty)$ such that $\Sigma(\infty)=\lim^X_{n\map\infty}\Sigma_{i_n}(p_{i_n})$.
	\end{prop}
	
	\begin{proof}
Note that $\Sigma_i$ is the restriction of a $K$-good qi section over $[0,\infty)$. Thus for each $i\in\N$ (as in the statement) and for each $l\in\N\cup\{0\}$, we have $\xi^{(i)}_l=(\xi^{(i)}_{1,l},\xi^{(i)}_{2,l},\xi^{(i)}_{3,l})\in\pa^3F_l$ and $\pa\phi_l(\xi^{(i)}_{j,l})=\xi^{(i)}_{j,l+1}$ where $j\in\{1,2,3\}$ and $\phi:F_l\to F_{l+1}$ is a natural quasiisometry (see Lemma \ref{lem-natural map QI}) satisfying the following. For $0\le l\le p_i$, $\Sigma_i(l)$ is a $K$-barycenter of the ideal triangle $\triangle(\xi^{(i)}_{1,l},\xi^{(i)}_{2,l},\xi^{(i)}_{3,l})$ in $F_l$. 
Now by Lemma \ref{D-barycenter}, we have $$d_l(Bary_{F_l}(\xi^{(i)}_l),\Sigma_i(l))\le D_{\ref{D-barycenter}}(\dl,max\{K,D_{\ref{barycenter}}(\dl)\})=D \text{ (say)}$$ for all $0\le l\le p_i$.\smallskip
		
\noindent{\em Claim}: {\em Let $l\in\N\cup\{0\}$. Then $diam\{Bary_{F_l}(\xi^{(i)}_l):i\in\N\text{ and }p_i\ge l\}$ is bounded in $F_l$. (The bound need not be uniform.)}\smallskip
		
\noindent{\em Proof of Claim}: Let $diam~\{\Sigma_i(0):i\in\N\}\le D'$ for some $D'\ge0$. Then for all $i,j\in\N$ with $l\le p_i$ and $l\le p_j$, we have $d_X(\Sigma_i(l),\Sigma_j(l))\le d_X(\Sigma_i(l),\Sigma_i(0))+d_X(\Sigma_i(0),\Sigma_j(0))+d_X(\Sigma_j(0),\Sigma_j(l))\le2Kl+2K+D'$ as $\Sigma_i,~\Sigma_j$ are $K$-qi embeddings. Hence, combining the above inequality, for all $i,j\in\N$ with $l\le p_i,l\le p_j$, we have $d_X(Bary_{F_l}(\xi^{(i)}_l),Bary_{F_l}(\xi^{(j)}_l))\le 2D+D''$ where $D''=2Kl+2K+D'$. Thus $d_l(Bary_{F_l}(\xi^{(i)}_l),Bary_{F_l}(\xi^{(j)}_l))\le f(2D+D'')$. This completes the proof of claim.\qed\smallskip
		
Now we will find a subsequence, say $\{p_{i_n}\}_n\sse \{p_i\}_i$, such that the following hold.
		
\begin{enumerate}
			\item For all $m\in\N\cup\{0\}$ and $j\in\{1,2,3\}$, $\lim^{F_m}_{n\map\infty}\xi^{(i_n)}_{j,m}=\eta_{j,m}\in\pa F_m$.
			
			\item Moreover, $\eta_m=(\eta_{1,m},\eta_{2,m},\eta_{3,m})\in\pa^3F_m$ for all $m\in\N\cup\{0\}$.
			
			\item Finally, barycenters of ideal triangles $\triangle(\eta_{1,m},\eta_{2,m},\eta_{3,m})$ in $F_m$ form the required qi section.
\end{enumerate}
		
First, we construct subsequences in each of the fibers satisfying properties $(1)$ and $(2)$, ensuring that each subsequent subsequence is contained in the preceding one. Then, we apply a diagonal argument to obtain properties $(1)$, $(2)$ and $(3)$ as stated above. The construction is by induction on the distance of the fiber from the initial one.
		
		\noindent\underline{Initial step}: Note that
		\begin{itemize}
			\item $F_0$ is proper hyperbolic metric space, (so $\pa F_0$ is compact metrizable space) and
			
			\item $diam\{Bary_{F_0}(\xi^{(i)}_0):i\in\N\}$ is bounded. 
		\end{itemize} Then by Lemma \ref{bdd barycen conv}, after passing to a subsequence if necessary, we will have a subsequence, say $\{i_{n,0}\}_n$, of $\{i\}_i$ (i.e., $\{p_{i_{n,0}}\}_n\sse \{p_i\}_i$) such that $\lim^{F_0}_{n\map\infty}\xi^{(i_{n,0})}_{j,0}=\eta_{j,0}$ for $j\in\{1,2,3\}$, and $\eta_0=(\eta_{1,0},\eta_{2,0},\eta_{3,0})\in\pa^3F_0$ and $d_0(Bary_{F_0}(\xi^{(i_{n,0})}_0),Bary_{F_0}(\eta_0))\le D_{\ref{bdd barycen conv}}(\dl)$ for all $n\in\N$.

		\noindent\underline{Induction hypothesis}: Suppose we have constructed the sequence till $m^{\text{th}}$-fiber where  $m\in\N$. That is, we have a subsequence, $\{i_{n,m}\}_n\sse\{i_{n,m-1}\}_n$ say, such that
		\begin{itemize}
			\item $\lim^{F_m}_{n\map\infty}\xi^{(i_{n,m})}_{j,m}=\eta_{j,m}$ for $j\in\{1,2,3\}$,
			
			\item $\eta_m=(\eta_{1,m},\eta_{2,m},\eta_{3,m})\in\pa^3F_m$ and
			
			\item $d_m(Bary_{F_m}(\xi^{(i_{n,m})}_m),Bary_{F_m}(\eta_m))\le D_{\ref{bdd barycen conv}}(\dl)$ for all $n\in\N$.
		\end{itemize}
		
		\noindent\underline{Induction step}: Consider a natural quasiisometry map $\phi_m:F_m\map F_{m+1}$ (see Lemma \ref{lem-natural map QI}). Since $\Sigma_i$'s are $K$-good qi lift, by definition, we have $\pa\phi_m(\xi^{(i_{n,m})}_{j,m})=\xi^{(i_{n,m})}_{j,m+1}$ for all $n\in\N$ and for all $j\in\{1,2,3\}$. Now $F_{m+1}$ is proper hyperbolic metric space and $diam\{Bary_{F_{m+1}}(\xi^{(i_{n,m})}_{m+1}):n\in\N\}$ is bounded by Claim above. So after passing to a subsequence if necessary, (as in the initial step), we have a subsequence, $\{i_{n,m+1}\}_n\sse\{i_{n,m}\}_n$ say, such that $\lim^{F_{m+1}}_{n\map\infty}\xi^{(i_{n,m+1})}_{j,m+1}=\eta_{j,m+1}\in\pa F_{m+1}$ for $j\in\{1,2,3\}$, and $\eta_{m+1}=(\eta_{1,m+1},\eta_{2,m+1},\eta_{3,m+1})\in\pa^3F_{m+1}$ and $$d_{m+1}(Bary_{F_{m+1}}(\xi^{(i_{n,m+1})}_{m+1}),Bary_{F_{m+1}}(\eta_{m+1}))\le D_{\ref{bdd barycen conv}}(\dl)$$ for all $n\in\N$. This completes the induction process.\smallskip
		
		\noindent Now, we set $i_n=i_{n,n}$ for $n\in\N$. So we get a subsequence $\{i_n\}_n\sse\{i\}_i$ (i.e., $\{p_{i_n}\}_n\sse \{p_i\}_i\}$) satisfying $(1)$ and $(2)$ above. Finally, we will define the required qi section $\Sigma$ as mentioned in $(3)$ and complete the proof.
		
		Since $\phi_n:F_n\map F_{n+1}$ induces a homeomorphism $\pa \phi_n:\pa F_n\map\pa F_{n+1}$, so $\pa \phi_n(\eta_{j,n})=\eta_{j,n+1}$ for all $j\in\{1,2,3\}$. Note that $Bary_{F_n}(\eta_n)$ is a $D_{\ref{barycenter}}(\dl)$-barycenter of the ideal triangle $(\eta_{1,n},\eta_{2,n},\eta_{3,n})$ (see Lemma \ref{barycenter}). Define $\Sigma:[0,\infty)\map X$ such that $\Sigma(n)=Bary_{F_n}(\eta_n)$ for all $n\in\N\cup\{0\}$ and for all $t\in[n,n+1)$, $\Sigma(t)=\Sigma(n)$. 
		Hence by Lemma \ref{for special qi sec}, $\Sigma$ is a $K_{\ref{for special qi sec}}(\dl,f)$-good qi section (see Definition \ref{good qi section}).
		
		Note that $i_n\ge n$. Then from the above construction, we have $d_X(\Sigma_{i_n}(n),\Sigma(n))\le d_n(\Sigma_{i_n}(n),Bary_{F_n}(\eta_n))\le d_n(\Sigma_{i_n}(n),Bary_{F_n}(\xi^{(i_n)}_n))+d_n(Bary_{F_n}(\xi^{(i_n)}_n),Bary_{F_n}(\eta_n))\le D+D_{\ref{bdd barycen conv}}(\dl)$. Then by Lemma \ref{close give same limit}, $\LMX\Sigma_{i_n}(n)=\Sigma(\infty)$. Since $\{\Sigma_{i_n}(n):n\in\N\}$ is unbounded, by Lemma \ref{on qc imp same}, $\LMX\Sigma_{i_n}(p_{i_n})=\LMX\Sigma_{i_n}(n)$. Therefore, $\LMX\Sigma_{i_n}(i_n)=\Sigma(\infty)$. This completes the proof of Proposition \ref{mild imp special}.
	\end{proof}

	Now we will prove the main result of this subsection. Suppose $\pa_{good}X:=\{\Sigma(\infty):\Sigma\text{ is a }K_{\ref{for special qi sec}}(\dl,f) \text{-good qi section over }[0,\infty)\}$. With this notation we have the following.
	
	\begin{prop}\label{description of pa X}
		Suppose $\pi:X\map [0,\infty)$ is an $f$-metric graph bundle with controlled hyperbolic fibers, and suppose that $X$ is hyperbolic. Further, assume that the fibers are proper metric spaces. Then, as a set, $\pa X=\Lambda_X(F_0)\cup\pa_{good}X$.
	\end{prop}
	
	\begin{proof}
		Suppose the fibers and the total space $X$ are $\dl$-hyperbolic for some $\dl\ge0$. Let the barycenter maps for fibers be $L$-coarsely surjective for some $L\ge0$. Suppose $\al$ is a geodesic ray in $X$. 
		Let $n\in\N$. Then by Corollary \ref{cor-existence good qi}, there is a $K$-good qi lift, say $\Sigma_n$, over $[0,\pi(\al(n))]$ through $\al(n)$, where $K=K_{\ref{cor-existence good qi}}(\dl,L,f)$. So $d_X(\Sigma_n(0),\al(n))\le K~d(0,\pi(\al(n)))+K$.\smallskip
		
		\noindent {\em Case} 1: Suppose $diam~\{\Sigma_n(0):n\in\N\}$ is unbounded. Then by Lemma \ref{unbdd imply in fiber}, $\al(\infty)=\LMX \al(n)=\LMX\Sigma_n(0)\in\Lambda_X(F_0)$.\smallskip
		
		\noindent {\em Case} 2: Suppose $diam~\{\Sigma_n(0):n\in\N\}$ is bounded. Note that $\{\pi(\al(n)):n\in\N\}$ is unbounded. Indeed, if not, let $d(0,\pi(\al(n)))\le D$ for some $D\ge0$ and for all $n\in\N$. Then $d_X(\Sigma_n(0),\al(n))\le KD+K$ for all $n\in\N$. Since $diam~\{\Sigma_n(0):n\in\N\}$ is bounded, we have that $\{d_X(x,\al(n)):n\in\N\}$ is bounded for any fixed point $x\in F_0$. This contradicts the fact that $\al$ is a geodesic ray.

		Then by Proposition \ref{mild imp special}, there is a $K_{\ref{for special qi sec}}(\dl,f)$-good qi section, say $\Sigma$, over $[0,\infty)$ and a subsequence $\{n_i\}\sse\N$ such that $\Sigma(\infty)=\lim^X_{i\map\infty}\Sigma_{n_i}(n_i)$. Again, $\lim^X_{i\map\infty}\al(n_i)=\al(\infty)$ implies that $\Sigma(\infty)=\al(\infty)$ in $\pa X$. This completes the proof of Proposition \ref{description of pa X}.
	\end{proof}
	
	We end this subsection by proving the following result, which is needed in the proof of Theorems \ref{thm-main one ended} and \ref{main thm over ray graph gen}.
	
	\begin{lemma}\label{need in main thms}
		Suppose $\pi:X\map[0,\infty)$ is an $f$-metric graph bundle such that $X$ is hyperbolic. Let $k\ge1$, $\{i_n\}\sse\N$, and let $\Sigma$ be a $k$-qi section over $[0,\infty)$. For all $n\in\N$, let $\Sigma_n$ be a $k$-qi lift of $[0,i_n]$ such that $d_0(x,\Sigma_n(0))\map\infty$ as $n\map\infty$ for a fixed $x\in F_0$. 
		Moreover, assume that there is $M\ge0$ with $d_{i_n}(\Sigma_n(i_n),\Sigma(i_n))\le M$ for all $n\in\N$. Then we have $\LMX \Sigma_n(0)=\Sigma(\infty)$.
	\end{lemma}
	
	\begin{proof}
		First, we will prove that $i_n\map\infty$ as $n\map\infty$. If not, after passing through a subsequences, if necessary, we may assume that $i_n\le D$ for all $n\in\N$ and for some $D\ge0$. Now restrict both $\Sigma$ and $\Sigma_n$ on $[0,i_n]$. Thus $d_X(\Sigma(0),\Sigma(i_n))\le kD+k$ and $d_X(\Sigma_n(0),\Sigma_n(i_n))\le kD+k$. Hence by triangle inequality, $d_X(\Sigma(0),\Sigma_n(0))\le d_X(\Sigma(0),\Sigma(i_n))+d_X(\Sigma(i_n),\Sigma_n(i_n))+d_X(\Sigma_n(i_n),\Sigma_n(0))\le2(kD+k)+M=D'$ (say). Since the fibers are $f$-proper embedding in $X$, we have $d_0(\Sigma(0),\Sigma_n(0))\le f(D')$. This contradicts to the assumption that $d_0(x,\Sigma_n(0))\map\infty$ as $n\map\infty$ for a fixed $x\in F_0$.
		
		Again by Lemma \ref{close give same limit}, $\LMX\Sigma_n(i_n)$ exists in $\pa X$ and $\LMX\Sigma_n(i_n)=\Sigma(\infty)$. Now $\lim_{n\map\infty}d_0(x,\Sigma_n(0))=\infty$ implies, by Lemma \ref{unbdd imply in fiber}, $\lim^X_{n\map\infty}\Sigma_n(0)=\lim^X_{n\map\infty}\Sigma_n(i_n)$. Therefore, $\Sigma(\infty)=\lim^X_{n\map\infty}\Sigma_n(0)$. This completes the proof.
	\end{proof}
	
	\subsection{Flow}\label{subsec-flow}
	
	
	Now we define the notion `flow' of a subset of a fiber. This version of flow differs slightly from the one introduced in \cite[Chapter 3]{ps-kap} (see also \cite[Section $6$]{halder-com}); however, after taking the quasiconvex hull in each fiber, the two definitions coincide. Since our focus is on counting the number of vertices in the flow, the present definition is more suitable for our purposes.
	\begin{defn}[Flow]\label{flow}
		Suppose $\pi:X\map[0,\infty)$ is an $f$-metric graph bundle such that the fibers have uniformly bounded valence.
		
		Suppose $A\sse V(F_0)$ is a subset. Let $k\ge1$. The $k$-flow of $A$ is denoted by $\F l_k(A)$ and is defined to be the union of all $k$-qi sections through points in $A$ over $[0,\infty)$.
	\end{defn}
	Suppose the number of vertices in any subset $C\sse X$ is denoted by $||C||$. The following result states that the flow of any subset can contain vertices from the fibers at most exponentially in distance.
	\begin{lemma}\label{exponentially bdd}
Let $k\ge1$ and $D\ge2$. Suppose $\pi:X\map[0,\infty)$ is an $f$-metric graph bundle such that the valence at each vertex in the fiber is bounded above by $D$.
		
		Then there is $c=D^{f(4k)+1}>1$ such that for any subset $A\sse V(F_0)$, we have $||\F l_k(A)\cap F_i||\le ||A||c^i$ for all $i\in\N\cup\{0\}$.
	\end{lemma}	
	\begin{proof}
		Note that for any vertex $v\in A$, $diam (\F l_k(v)\cap F_1)\le4k$ in $X$-metric, and so $diam(\F l_k(v)\cap F_1)\le f(4k)$ in $F_1$-metric. Since the valence at each vertex in the fiber is bounded above by $D$, we have $||\F l_k(v)\cap F_1||\le D^{f(4k)+1}=c>1$. Then $||\F l_k(A)\cap F_1||\le||A||c$. Inductively, we get $||\F l_k(A)\cap F_i||\le ||A|| c^i$ for all $i\in\N\cup\{0\}$.
	\end{proof}	
	
	\section{Cannon--Thurston maps are surjective when the bases are $[0,\infty)$}\label{sec-main theorems over rays}
	
	
In this section, we will prove Theorems \ref{thm-combo1}, \ref{main thm over ray graph gen}, and Corollary \ref{cor-main gp over ray new intro} when the base $B$ of the metric graph bundle is $[0,\infty)$. We will first prove Theorem \ref{thm-combo1} $(B)$. In the proof of Theorem \ref{thm-combo1} $(B)$, although the graph structure itself does not play a crucial role, we choose to work with metric graph bundles for simplicity. As mentioned in Remark \ref{some remarks} $(2)$ in the introduction, for a version of Theorem \ref{thm-combo1} $(B)$ in metric bundles, we refer the reader to Theorem \ref{thm-all in one appendix} in the Appendix. 

\subsection{When the fibers are one-ended}
	
\begin{theorem}\label{thm-main one ended}
Let $\dl\ge0$, $L\ge0$. Suppose $\pi:X\to [0,\infty)$ is an $f$-metric graph bundle satisfying the following conditions.

	\begin{enumerate}
\item The fibers $F_i:=\pi^{-1}(i),~i\in\N\cup\{0\}$ and the total space $X$ are $\dl$-hyperbolic.
	
\item The barycenter maps $\pa^3F_i\map F_i,~i\in\N\cup\{0\}$ are $L$-coarsely surjective.
	
\item The fibers are one-ended and proper metric spaces.
\end{enumerate}
Then the CT map $\pa F_0\to\pa X$ is surjective.
\end{theorem}

\proof. We first  note that $\pi:X\map [0,\infty)$ satisfies $k$-flaring condition for all $k\ge1$ by Theorem \ref{thm-metric bundle comb} $(1)$. Let $K=K_{\ref{cor-existence good qi}}(\dl,L,f)$ be the constant depending only on $\dl$, $L$ and $f$ as in Corollary \ref{cor-existence good qi}.
	
	Since the inclusion $F_0\map X$ admits a CT map (\cite{mitra-trees}; see also \cite[Section $2$]{bowditch-stacks}), by Lemma \ref{CT image is limit set}, $\Lambda_X(F_0)=\pa i_{F_0,X}(\pa F_0)$. Note that by Corollary \ref{cor-existence good qi}, through each point of $X$ there is a $K$-good qi section. Let $\Sigma:[0,\infty)\map X$ be a $K$-good qi section over $[0,\infty)$. Hence, by Proposition \ref{description of pa X} (and Lemma \ref{CT image is limit set}), it is enough to show that there is $\{p_n\}\sse F_0$ such that $\LMX p_n=\Sigma(\infty)$. Now the rest is devoted to proving this.\smallskip
	
For each $i\in\N\cup\{0\}$, fix a natural $k_{\ref{lem-natural map QI}}(f)$-quasiisometry $\phi_i:F_i\map F_{i+1}$ as in Lemma \ref{lem-natural map QI}.	Suppose $\xi_i=(\xi_{1,i},\xi_{2,i},\xi_{3,i})\in\pa^3F_i$ such that $\Sigma(i)$ is a $K$-barycenter of the ideal triangle $\triangle(\xi_{1,i},\xi_{2,i},\xi_{3,i})$ in $F_i$ and $\pa\phi_{i}(\xi_{j,i})=\xi_{j,i+1}$ where $i\in\N\cup\{0\}$ and $j\in\{1,2,3\}$.
	
 Let $\bt_i$ be a geodesic line joining $\xi_{1,i}$ and $\xi_{2,i}$ in $F_i$ where $i\in\N\cup\{0\}$. Let $D=D_{\ref{lem-geo lying out imply out}}(\dl,k_{\ref{lem-natural map QI}}(f))$ be the constant as in Lemma \ref{lem-geo lying out imply out} $(a)$ depending only on $\dl$ and $k_{\ref{lem-natural map QI}}(f)$.
	
	Let $n\in\N$, and let $\al_{n,0}$ be a continuous rectifiable path joining two points belonging to $\bt_0$ and lying outside $n$-radius ball centered at $\Sigma(0)$ in $F_0$. This is possible by the assumption that fibers are one-ended. Then for all $i\in\N\cup\{0\}$, using Lemma \ref{lem-geo lying out imply out} $(a)$, we will inductively have a continuous rectifiable path $\al_{n,i}$ in $F_i$ joining two points belonging to $\bt_i$ such that $$Hd_{i+1}(\phi_i(\al_{n,i}),\al_{n,i+1})\le D\hspace{5mm} (*)$$Moreover, we have the following note.\smallskip
	
	\noindent{\em Note} $1$: {\em By Lemma \ref{lem-geo lying out imply out} $(b)$, we have a constant $A=A_{\ref{lem-geo lying out imply out}}(\dl,k_{\ref{lem-natural map QI}}(f))$ depending only on $\dl$ and $k_{\ref{lem-natural map QI}}(f)$ such that length $(\al_{n,i+1})\le A~length~(\al_{n,i})+A$ for all $i\in\N\cup\{0\}$. Thus, inductively, we have $length~(\al_{n,i})\le A^i~length~(\al_{n,0})+A^i+A^{i-1}+\dots+A\le A^i~length~(\al_{n,0})+A(\frac{A^i-1}{A-1})\le A^i~length~(\al_{n,0})+A^i(\frac{A}{A-1})\le A^i~length~(\al_{n,0})+2A^i\le A^i~(length ~(\al_{n,0})+2)$ for all $i\in\N$}, where we have used the fact that $A\ge 2$ (see Lemma \ref{lem-geo lying out imply out}).

	\medskip
	\noindent{\em Claim} $1$: {\em Let $i\in\N\cup\{0\}$ and $x\in\al_{n,i}$. Then there is a $2(1+D)$-qi section, say $\Sigma'$, over $[0,\infty)$ such that $\Sigma'(j)\in\al_{n,j}$ for all $j\in\N\cup\{0\}$ and $\Sigma'(i)=x$}.\smallskip
	
\noindent{\em Proof of Claim} $1$: We first, inductively, define $\Sigma'$ on $\N\cup\{0\}$, and then for any interior point $t\in(j,j+1)$, we define $\Sigma'(t):=\Sigma'(j)$ where $j\in\N\cup\{0\}$. We set $\Sigma'(i)=x\in\al_{n,i}$.

For $j\ge i$, we assume $\Sigma'(j)\in\al_{n,j}$ has been defined and thus we define $\Sigma'(j+1)\in\al_{n,j+1}$. Since $\Sigma'(j)\in\al_{n,j}$, by $(*)$ above, there is $x'\in\al_{n,j+1}$ such that $d_{j+1}(\phi_{j}(\Sigma'(j)),x')\le D$, and we set $\Sigma'(j+1)=x'$. Note that $d_X(\Sigma'(j),\Sigma'(j+1))\le D\le1+D$.

For $0< j\le i$, we assume $\Sigma'(j)\in\al_{n,j}$ has been defined and thus we define $\Sigma'(j-1)\in\al_{n,j-1}$. Since $\Sigma'(j)\in\al_{n,j}$, by $(*)$ above, we take $x'\in\phi_{j-1}(\al_{n,j-1})$ such that $d_j(x',\Sigma'(j))\le D$. 
Now we take $\Sigma'(j-1)\in\al_{n,j-1}$ such that $\phi_j(\Sigma'(j-1))=x'$. Thus $d_X(\Sigma'(j-1),\Sigma'(j))\le D+1$. This completes, by induction, the definition of $\Sigma'$ over $\N\cup\{0\}$.
	
Now for any $j,j'\in\N\cup\{0\}$, we have $d_X(\Sigma'(j),\Sigma'(j'))\le (D+1)|j-j'|$. As mentioned above, we define $\Sigma'(t):=\Sigma'(j)$ for $t\in(j,j+1)$. Then for any $t\in(j,j+1)$ and $s\in(j',j'+1)$, we have $d_X(\Sigma'(t),\Sigma'(s))\le (D+1)|t-s|+2(D+1)$. Since $\pi$ is $1$-Lipschitz, $\Sigma'$ is a $2(1+D)$-qi section over $[0,\infty)$. Note that we have $\Sigma'(j)\in\al_{n,j}$ for all $j\in\N\cup\{0\}$ and $\Sigma'(i)=x$. This completes the proof of Claim $1$.\medskip

Let $K'=2(1+D)$, $K''=max\{K',K\}$. Let $M_{K''}$ be the constant coming from $K''$-flaring condition (see Definition \ref{flaring condition}).\smallskip
	
\noindent{\em Claim} $2$: {\em There is an unbounded subsequence $\{n_k\}_k\sse\N$ and a number $l_k=l(n_k)\in\N$ (depending on $n_k$) for each $k\in\N$ such that $d_{l_k}(x_{l_k},\Sigma(l_k))\le M_{K''}$ for some $x_{l_k}\in\al_{n_k,l_k}$.}\smallskip

One can think of $l_k$ as recording the time at which the flow path $\alpha_{n_k,l_k}$ (of $\al_{n_k,0}$) comes within distance $M_{K''}$ of $\Sigma(l_k)$ in the fiber $F_{l_k}$. In particular, $l_k\to\infty$ as $n_k\to\infty$ (equivalently, as $k\to\infty$) can be seen from the next paragraph and the first line of the proof of Lemma \ref{need in main thms}.\smallskip

Assuming Claim 2 for the moment, we now use it to complete the proof of Theorem \ref{thm-main one ended}. Let $\Sigma_k$ be a $K'$-qi section through $x_{l_k}$ such that $\Sigma_k(i)\in\al_{n_k,i}$ for all $i\in\N\cup\{0\}$ (see Claim $1$). Now $d_0(\Sigma(0),\al_{n,0})\ge n$ implies $d_0(\Sigma(0),\Sigma_k(0))\ge n_k$, and we restrict $\Sigma_k$ over $[0,l_k]$ and we have $d_{l_k}(\Sigma_k(l_k),\Sigma(l_k))\le M_{K''}$. Hence by Lemma \ref{need in main thms}, $\Sigma(\infty)=\lim^X_{k\map\infty}\Sigma_k(0)\in\Lambda_X(F_0)$. This completes the proof of Theorem \ref{thm-main one ended}.\smallskip

	


\noindent{\em Proof of Claim} $2$: Suppose Claim $2$ is not true. Then for all large $n\in\N$, $d_i(\al_{n,i},\Sigma(i))>M_{K''}$ for all $i\in\N\cup\{0\}$. {\bf We fix} $\bm {n}\in{\bm\N}$ such that $d_i(\al_{n,i},\Sigma(i))>M_{K''}$ for all $i\in\N\cup\{0\}$. For rest of the proof, $n$ is fixed with the stated property. In particular, let $\Sigma_n$ be a $K'$-qi section such that $\Sigma_n(i)\in\al_{n,i}$ for all $i\in\N\cup\{0\}$ (see Claim $1$); then for all $i\in\N\cup\{0\}$, we have $d_i(\Sigma_n(i),\Sigma(i))>M_{K''}$. Note that both $\Sigma$ and $\Sigma_n$ are $K''$-qi sections.
	
Let $C=~max~\{d_0(\Sigma(0),x):x\in\al_{n,0}\}$. Let $a=a_{\ref{once start decreasing is decreasing}}(\dl,K'',C)>0$ be the constant depending only on $\dl$, $K''$ and $C$ as in Lemma \ref{once start decreasing is decreasing}, and let $b=2^{\frac{1}{\dl+1}}$. Since $\pi:X\map[0,\infty)$ also satisfies $K''$-flaring condition as mentioned above (see Theorem \ref{thm-metric bundle comb} $(1)$), by Lemma \ref{once start decreasing is decreasing}, we have $d_i(\Sigma_n(i),\Sigma(i))\ge ab^i$ for all $i\in\N\cup\{0\}$. Note that $n$ is fixed and so is $\al_{n,0}$. Thus $a$ is a fixed number independent of $i$. Again $n$ is fixed but the choice of $\Sigma_n$ is arbitrary through a point belonging to $\al_{n,i}$ for all $i\in\N\cup\{0\}$, therefore, we have $d_i(\Sigma(i),\al_{n,i})\ge ab^i$ for all $i\in\N\cup\{0\}$.
	
	
	\noindent Since $d_i(\Sigma(i),\bt_i)\le K$ for all $i\in\N\cup\{0\}$, we have $d_i(x_i,\al_{n,i})\ge ab^i-K$ for some $x_{i}\in\bt_{i}$. Then by Lemma \ref{geo line div exp}, $$length~(\al_{n,i})\ge b^{ab^i-K-1}=B(b^a)^{b^i}\hspace{5mm} (**)$$ for some constant $B>0$ independent of $i$ and note that $b=2^{\frac{1}{\dl+1}}>1$, $a>0$. Note also that $n$ is fixed. This contradicts Note $1$ above, which states that the length of $\alpha_{n,i}$ grows at most exponentially in $i$, whereas $(**)$ shows that its growth is doubly exponential. This completes the proof of Claim $2$.\qed

	\subsection{When the fibers are not necessarily one-ended}\label{main thm over ray graph gen sec} 
	
	In this subsection, we prove surjectivity of the CT map in a more general setting (Theorem \ref{main thm over ray graph gen}). As a consequence of this result, we will prove Theorem \ref{thm-combo1} $(A)$ when the base $B$ is $[0,\infty)$ in Theorem \ref{thm-main bdd valence}. For the notion of uniform exponential growth and uniform bounded valence, we refer the reader to Definitions \ref{defn-exponential growth} and \ref{defn-valence} respectively.
	
	\begin{theorem}\label{main thm over ray graph gen}
		Suppose $\pi:X\map[0,\infty)$ is an $f$-metric graph bundle such that the fibers are uniformly hyperbolic. Further, suppose that the fibers have uniformly bounded valence. Moreover, assume that the collection $\{F_i:i\in\N\cup\{0\}\}$ have uniform exponential growth. Finally, let $X$ be hyperbolic. Then the CT map $\pa\pi^{-1}(0)\map\pa X$ is surjective.
	\end{theorem}
	
	\begin{proof}
		The idea of proof of Theorem \ref{main thm over ray graph gen} is somewhat similar to that of Theorem \ref{thm-main one ended}. In the proof of Theorem \ref{thm-main one ended}, to show that the good qi sections can indeed be realized as limit points of the fibers, we took paths in $F_0$ and `flowed' them in all the fibers. To conclude such a statement here, we analyse the flow of a fixed subset of $F_0$ (introduced in Subsection \ref{subsec-flow}) and invoke the exponential growth property of the fibers.
		
		\noindent According to the assumptions, we have the following.
		
		\begin{enumerate}
			\item Without loss of generality, we assume that the fibers and the total space $X$ are $\dl$-hyperbolic for some $\dl\ge0$.
			
			\item We have $D\ge1$ such that for all $i\in \N\cup\{0\}$ and for all $u\in V(F_i)$, the valence at $u$ in $F_i$ is bounded above by $D$.
			
			
			\item There are constant $s>0$ and $t>1$ satisfying the following. Let $i\in\N\cup\{0\}$ and $v\in V(F_i)$. For a subset $C\sse V(F_i)$, let $||C||$ denote the number of vertices in $C$. Then $||B(v,n)||\ge st^n$ for all $n\in\N$ where $B(v,n)=\{w\in V(F_i):d_i(v,w)\le n\}\sse F_i$.
		\end{enumerate}
		
We also note that $\pi:X\map [0,\infty)$ satisfies $k'$-flaring condition for all $k'\ge1$ by Theorem \ref{thm-metric bundle comb} $(1)$.
		
Since the inclusion $F_0\map X$ admits a CT map (\cite{mitra-trees}; see also \cite[Section $2$]{bowditch-stacks}), by Lemma \ref{CT image is limit set}, $\Lambda_X(F_0)=\pa i_{F_0,X}(\pa F_0)$. Hence by Proposition \ref{des pa X} (and Lemma \ref{CT image is limit set}), it is enough to show that for any $1$-qi section $\Sigma:[0,\infty)\map X$, there is $\{p_n\}\sse F_0$ such that $\LMX p_n=\Sigma(\infty)$.
		
Fix $k=k_{\ref{lem-qi section in ladder}}(\dl,1,f)\ge1$ where $k_{\ref{lem-qi section in ladder}}(\dl,1,f)$ is the constant depending only on $\dl$ and $f$ as in Lemma \ref{lem-qi section in ladder}. Fix the constant $M_k>0$ coming from $k$-flaring condition (see Definition \ref{flaring condition}).\smallskip
		
\noindent{\em Claim}: {\em For each $n\in\N$ there is $p_n\in F_0$ with $d_0(\Sigma(0),p_n)\ge n$ and a $k$-qi section, $\Sigma_n$ say, through $p_n$, and a number $l_n\in\N$ such that $d_{l_n}(\Sigma_n(l_n),\Sigma(l_n))\le M_k$.}\smallskip

As we saw in Claim~2 of the proof of Theorem~\ref{thm-main one ended}, here as well one can think of $l_n$ as recording the time at which $\Sigma_n(l_n)$ comes within distance $M_k$ of $\Sigma(l_n)$ in the fiber $F_{l_n}$. In this case as well, it follows from the first line of the proof of Lemma \ref{need in main thms} and the following paragraph that $l_n\to\infty$ as $n\to\infty$.\smallskip

Suppose the above Claim is true. Now restrict $\Sigma_n$ over $[0,l_n]$. Then we are done by Lemma \ref{need in main thms}, namely $p_n=\Sigma_n(0)$, $d_{0}(\Sigma(0),\Sigma_n(0))\ge n$, $d_{l_n}(\Sigma_n(l_n),\Sigma(l_n))\le M_k$, and so $\Sigma(\infty)=\LMX p_n\in\Lambda_X(F_0)$.\smallskip
		
So we only need to prove the Claim above.\smallskip
		
\noindent{\em Proof of Claim}: Suppose the Claim is false. Then there exists $n_0\in\N$ such that for every $p\in F_0$ with $d_0(\Sigma(0),p)\ge n_0$, every $k$-qi section $\Sigma'$ through $p$ satisfies $d_i(\Sigma(i),\Sigma'(i))>M_k$ for all $i\in\N\cup\{0\}$. (See Remark~\ref{c-qi lift} for the existence of a $1$-qi section.)
		
		
Let $A=V(B(\Sigma(0),n_0-1))=\{x\in V(F_0):d_0(\Sigma(0),x)\le n_0-1\}$. ($A$ is the set of vertices in $(n_0-1)$-radius ball centered at $\Sigma(0)$ in the fiber $F_0$.) Now we consider $\F l_1(A)$ (union of all $1$-qi section through points in $A)$ as defined in Definition \ref{flow}.
		
Let $a=a_{\ref{once start decreasing is decreasing}}(\dl,k,n_0)$ be the constant as in Lemma \ref{once start decreasing is decreasing} depending only on $\dl$, $k$ and $n_0$. Let $b=2^{\frac{1}{\dl+1}}>1$. For $x\in\R_{\ge0}$, let $[x]$ denote the greatest integer not greater than $x$.\smallskip
		
\noindent{\em Subclaim}: {\em $V(B(\Sigma(i),[ab^i-1]))\sse \F l_1(A)\cap F_{i}$ for all $i\in\N\cup\{0\}$ where $B(\Sigma(i),[ab^i-1])$ is the (closed) ball of radius $[ab^i-1]$ centered at $\Sigma(i)$ in $F_i$.}\smallskip
		
\noindent{\em Proof of Subclaim}: Let $i_0\in\N\cup\{0\}$ and $y\in V(B(\Sigma(i_0),[ab^{i_0}-1]))$. We will show that $y\in\F l_1(A)\cap F_{i_0}$.

Suppose $\Sigma_y$ is a $1$-qi section through $y$. Then we will first prove that $\Sigma_y(0)\in A$. If $\Sigma_y(0)\notin A$, then $d_0(\Sigma(0),\Sigma_y(0))\ge n_0$. Let $v\in[\Sigma(0),\Sigma_y(0)]_{F_0}$ such that $d_0(\Sigma(0),v)=n_0$. Note that both $\Sigma$ and $\Sigma_y$ are $1$-qi sections, and $k=k_{\ref{lem-qi section in ladder}}(\dl,1,f)\ge1$. By Lemma \ref{lem-qi section in ladder}, we have a $k$-qi section, say $\Sigma'$, through $v$ such that $\Sigma'(i)\in[\Sigma(i),\Sigma_y(i)]_{F_i}$ for all $i\in\N\cup\{0\}$. Note also that both $\Sigma$ and $\Sigma'$ are $k$-qi sections, and $d_0(\Sigma(0),\Sigma'(0))=n_0$. By the contrary assumption of the Claim, we have $d_i(\Sigma(i),\Sigma'(i))>M_k$ for all $i\in\N\cup\{0\}$. Since $\pi:X\map[0,\infty)$ satisfies $k$-flaring condition as mentioned above (see Theorem \ref{thm-metric bundle comb} $(1)$), by Lemma \ref{once start decreasing is decreasing}, we have constants $a$ and $b$ as mentioned above such that $d_i(\Sigma'(i),\Sigma(i))\ge ab^i$ and  hence, $d_{i}(\Sigma(i),\Sigma'(i))>[ab^i-1]$ for all $i\in\N\cup\{0\}$. In particular, since $\Sigma'(i_0)\in[\Sigma(i_0),\Sigma_y(i_0)]_{F_{i_0}}$, we have, $d_{i_0}(\Sigma(i_0),y)=d_{i_0}(\Sigma(i_0),\Sigma_y(i_0))>[ab^{i_0}-1]$. This contradicts the choice of $y$. Hence $\Sigma_y(0)\in A$.
		
Therefore, $y\in\F l_1(A)\cap F_{i_0}$ since $\Sigma_y(0)\in A$ and $\Sigma_y$ is a $1$-qi section through $y$. This completes the proof of Subclaim.\smallskip
		
		
		Then by the assumption $(3)$ above, we have $||B(\Sigma(i),[ab^i-1])||\ge st^{[ab^i-1]}=A'(t^a)^{b^i}$ for some $A'>0$ independent of $i$. In particular, $||\F l_1(A)\cap F_{i}||\ge A'(t^a)^{b^i}$ for all $i\in\N\cup\{0\}$.
		
		On the other hand, by Lemma \ref{exponentially bdd}, there is $c=D^{f(4k)+1}>1$ such that $||\F l_1(A)\cap F_{i}||\le ||A|| c^i$ for all $i\in\N\cup\{0\}$. Since $n_0$ is fixed so $||A||$ is a fixed number; in fact, $||A||\le D^{n_0}$.
		
		This gives us a contradiction to the above fact that $||\F l_1(A)\cap F_{i}||$ grows at least doubly exponentially in $i$. This completes the proof of Claim, which in turn completes the proof of Theorem~\ref{main thm over ray graph gen}.
	\end{proof}
	%
	%
	%
	
As an application of Theorem \ref{main thm over ray graph gen} and Corollary \ref{cor-exponential growth QI inv}, we have the following.	
	
	\begin{cor}\label{cor-one fixed fiber}
		Suppose $\pi:X\map[0,\infty)$ is an $f$-metric graph bundle such that the fibers are uniformly quasiisometric to a fixed hyperbolic metric graph, say $Z$. Further, suppose that the fibers have uniformly bounded valence. Moreover, we assume that $Z$ has exponential growth. Finally, let $X$ be hyperbolic. Then the CT map $\pa\pi^{-1}(0)\map\pa X$ is surjective.
	\end{cor}
\begin{proof}
Since $Z$ has exponential growth and the fibers are quasiisometric to $Z$, by the first part of Corollary \ref{cor-exponential growth QI inv}, the fibers have exponential growth. On the other hand, the fibers have uniformly bounded valence and are uniformly quasiisometric to $Z$, by moreover part of Corollary \ref{cor-exponential growth QI inv}, the collection $\{F_i:i\in\N\cup\{0\}\}$ have uniform exponential growth. Then Corollary \ref{cor-one fixed fiber} follows from Theorem \ref{main thm over ray graph gen}.
\end{proof}

Now we will prove Theorem \ref{thm-combo1} $(A)$ when the base $B$ is a geodesic ray $[0,\infty)$.

\begin{theorem}\label{thm-main bdd valence}
Let $\dl\ge0$, $L\ge0$ and $D\ge1$. Suppose $\pi:X\to[0,\infty)$ is an $f$-metric graph bundle satisfying the following.
\begin{enumerate}
	\item The fibers $F_i:=\pi^{-1}(i),~i\in\N\cup\{0\}$ and the total space $X$ are $\dl$-hyperbolic.
	
	\item The barycenter maps $\pa^3F_i\map F_i,~i\in\N\cup\{0\}$ are $L$-coarsely surjective.
	
	\item The valence at each vertex of $F_i,~i\in\N\cup\{0\}$ in $F_i$ is bounded above by $D$.
\end{enumerate}
Then the CT map $\pa F_0\to\pa X$ is surjective.
\end{theorem}

\begin{proof}
By Theorem \ref{thm-bary imp expo growth}, each fiber has exponential growth, and moreover, a witnessing pair of the exponential growth depends only on $\dl$ and $L$. Hence the collection $\{F_i:i\in\N\cup\{0\}\}$ have uniform exponential growth. Note that $X$ is hyperbolic and the fibers have uniformly bounded valence. Therefore, Theorem \ref{thm-main bdd valence} follows from Theorem \ref{main thm over ray graph gen}.
\end{proof}

We now mention that Corollary \ref{cor-main gp over ray new intro} follows from Corollary \ref{cor-one fixed fiber} when the base $B$ is $[0,\infty)$. For a general base $B$, see Theorem  \ref{thm-application}. \smallskip

\noindent{\bf A different proof of Corollary \ref{cor-main gp over ray new intro} when the base $B$ is $[0,\infty)$.} Let $G$ be a nonelementary hyperbolic group. We fix a finite generating set for $G$. Let $\Gamma_G$ be the Cayley graph with respect to this generating set. Note that by Theorem \ref{thm-exp growth rate hyp grp}, $\Gamma_G$ has exponential growth. Note that the fibers have uniformly bounded valence (assumption of Corollary \ref{cor-main gp over ray new intro}). Therefore, Corollary \ref{cor-main gp over ray new intro} when the base $B$ is $[0,\infty)$ follows from Corollary \ref{cor-one fixed fiber}.\qed\smallskip

	We conclude this section with the following questions. Recall that in Theorem \ref{main thm over ray graph gen}, the assumption of uniform bounded valence of the fibers plays a crucial role in our argument. If this condition is relaxed, our argument no longer works. This motivates the following question.
	
	\begin{question}\label{qsn-general}
		Suppose $\pi:X\map[0,\infty)$ is an $f$-metric graph bundle with controlled hyperbolic fibers. Further, suppose that the fibers are proper metric spaces and are uniformly quasiisometric to a fixed hyperbolic space. Finally, assume that $X$ is hyperbolic. Is the CT map $\pa\pi^{-1}(0)\map\pa X$ {\em always} surjective?
	\end{question}
	
	\begin{remark}
One needs to be cautious about the fibers of the metric graph bundle in Question \ref{qsn-general} (see Question \ref{qsn-special}). For instance, the fibers can not be uniformly quasiisometric to $\mathbb H^n$ where $n\ge3$ (see Remark \ref{rmk-Hn can not be fiber}).
		
\noindent On the other hand, one can construct a combinatorial horoball (of any hyperbolic metric space) as in Example \ref{exp-combinatorial horoball} to obtain a surjective CT map in the context of Question \ref{qsn-general}. However, as mentioned in Remark \ref{rmk-fibers are not uniformly QI}, in such a construction, the fibers are not uniformly quasiisometric to a fixed hyperbolic space.
	\end{remark}	

As a special case of Question \ref{qsn-general}, we have the following. Let $F$ be a (simplicial) metric tree with a root vertex $u\in V(F)$. Suppose for all $v\in V(F)$ with $d_F(u,v)=n\in\N$, the valence at $v$ is $n+2$. Moreover, the valence at $u$ is a finite natural number.
		
\begin{question}\label{qsn-special}
$(1)$ Does there exist a metric graph bundle $\pi: X \to [0,\infty)$ whose fibers are uniformly quasiisometric to the tree $F$ (as described above), and such that $X$ is hyperbolic?

$(2)$ If the answer to (1) is yes, is the CT map $\partial \pi^{-1}(0) \to \partial X$ surjective?
\end{question}
%

	%
	%

\section{Main Theorem}\label{sec-main thm}	
	
The aim of this section is to prove the following main theorem of the paper.

\begin{theorem}\label{thm-application}
	Suppose $\pi:X\map B$ is an $f$-metric graph bundle with controlled hyperbolic fibers, and suppose that $X$ is hyperbolic. $($Note that $B$ is hyperbolic.$)$ Let $A$ be a qi embedded subgraph of $B$, and let $Y=\pi^{-1}(A)$. $($Note also that $Y$ is hyperbolic.$)$ Finally, we assume one of the following. 
	
	$(A)$ The fibers have uniformly bounded valence.
	
	$(B)$ The fibers are one-ended and proper metric spaces.
	
	\noindent Then CT map $\pa i_{Y,X}:\pa_s Y\map\pa_s X$ for the inclusion $i_{Y,X}:Y\to X$ is surjective.
%
%

In particular, the conclusion also holds if the ``controlled hyperbolic fibers" condition is replaced by the assumption that the fibers are uniformly quasiisometric to a fixed nonelementary hyperbolic group.
\end{theorem}
In the theorem above, we mention that the inclusion $i_{Y,X}:Y\to X$ admits a CT map by \cite{ps-krishna} (see Theorem \ref{all direction surj imply surj} (2)). For convenience, we will refer to Theorem~\ref{thm-application} under the assumption $(A)$ as Theorem~\ref{thm-application} $(A)$, and under the assumption $(B)$ as Theorem~\ref{thm-application} $(B)$ in next two sections. The following results are required in the proof of Theorem \ref{thm-application}.


\begin{theorem}\label{all direction surj imply surj}
	Suppose $\pi:X\map B$ is an $f$-metric graph bundle with controlled hyperbolic fibers, and suppose that $X$ is hyperbolic. Further, assume that the fibers are proper metric spaces. Fix $b\in V(B)$. Let $F_b=\pi^{-1}(b)$. Let $\eta\in\pa_s B$, and let $\al:[0,\infty)\map B$ be a continuous quasigeodesic ray starting at $b$ such that $\al(\infty)=\eta$. Let $X_{\al}:=\pi^{-1}(\al)$. Note that $X_{\al}$ is a (proper) hyperbolic geodesic metric space $($by Remark \ref{crucial remark} $(2))$. Then:
	
	\begin{enumerate}
		\item \textup{(\cite[Theorem 5.3]{pranab-mahan})} The inclusions $F_b\map X$ and $F_b\map X_{\al}$ admit CT maps.
		
		\item \textup{(\cite[Theorem 6.26]{ps-krishna})} The CT map $\pa i_{F_b,X_{\al}}:\pa F_b\map\pa X_{\al}$ is surjective for all $\eta\in\pa_s B$ if and only if the CT map $\pa i_{F_b,X}:\pa F_b\map \pa_sX$ is surjective.
		
		Moreover, in this case, suppose that $A$ is a qi embedded subgraph of $B$. Then $Y=\pi^{-1}(A)$ is hyperbolic by Remark \ref{crucial remark} $(2)$ and the inclusion $i_{Y,X}:Y\map X$ admits a CT map $\pa i_{Y,X}:\pa_s Y\map\pa_s X$ by \textup{\cite[Theorem 5.2]{ps-krishna}}. Finally, the map $\pa i_{Y,X}$ is surjective.
	\end{enumerate} 
\end{theorem}
We note that in the statement of \cite[Theorem 6.26]{ps-krishna} (see Theorem \ref{all direction surj imply surj} $(2)$), there is a misprint where the authors assume $\alpha$ to be a geodesic ray. However, since the space $B$ is not proper in general, geodesics may not exist representing $\eta$. In fact, their proof indicates that they are working with quasigeodesics rather than geodesics.

\subsection{Proof of Theorem \ref{thm-application}}
The proofs are simple consequences of Theorems \ref{thm-main bdd valence}, \ref{thm-main one ended} and \ref{all direction surj imply surj}. However, we will elaborate on that.


Note that by Theorem \ref{thm-metric bundle comb} $(1)$ and $(2)$, $B$ is hyperbolic (see \cite[Propositon 2.10]{pranab-mahan}). Fix $b\in V(B)$. Note also that in general, $B$ need not be a proper metric space. Let $\eta\in\pa_sB$, and let $\al:[0,\infty)\map B$ be a continuous quasigeodesic ray starting at $b\in V(B)$ such that $\al(\infty)=\eta$. Without loss of generality, we also assume that $\al$ is injective. As mentioned in Section \ref{prelims}, we think of $\al$ as lying in $B$. Let $X_{\al}=\pi^{-1}(\al)$. Note that by Remark \ref{crucial remark} $(1)$ and $(2)$, the restriction $\pi|_{X_\al}:X_\al\map \al\sse B$ is an $f$-metric graph bundle and $X_\al$ is hyperbolic respectively.

\noindent Since the properties of fibers remain unchanged, the restricted metric graph bundle $\pi|_{X_\alpha}: X_\alpha \to \alpha \subseteq B$ continues to have controlled hyperbolic fibers. Then by Theorem \ref{all direction surj imply surj} $(1)$, we have the CT map $\pa F_b\map\pa X_\al$. 
\smallskip

\noindent{\em Under the condition} $(A)$ the fibers are proper metric spaces and have uniformly bounded valence. Hence, by Theorem \ref{thm-main bdd valence}, the CT map $\partial F_b \to \partial X_\alpha$ is surjective.
\smallskip

\noindent{\em Under the condition} $(B)$ the fibers are one-ended and proper metric spaces. Hence, by Theorem \ref{thm-main one ended}, the CT map $\partial F_b \to \partial X_\alpha$ is surjective.\smallskip

\noindent Since $\eta\in\pa_sB$ was arbitrary, Theorem \ref{thm-application} follows from Theorem \ref{all direction surj imply surj} $(2)$.\smallskip

%

\noindent{\em Particular case -- when the fibers are uniformly quasiisometric to a nonelementary hyperbolic group}: Let $G$ be a nonelementary hyperbolic group, and let $\Gamma_G$ denote a Cayley graph of $G$ with respect to a finite generating set. Then the barycenter map for $\Gamma_G$ is coarsely surjective (see Remark \ref{rmk-bary map coarse surj in gp}). We assume that the fibers are $k$-quasiisometric to $\Gamma_G$ for some $k \ge 1$. It follows that the fibers are uniformly hyperbolic, and by Lemma \ref{lem-barycenter commuting}, the barycenter maps for the fibers are uniformly coarsely surjective. Therefore, we are done.\qed

	\section{Examples and Applications}\label{application and exp sec}
	
	In this section, we will see some examples and applications of our main theorem.
	\subsection{(Known) examples}\label{exp sec}
We will now see three main sources of examples -- two of which are group-theoretic (Subsection \ref{subsec-normal-commensurated}) -- where the CT map is surjective. While surjectivity of the CT map for these two group-theoretic examples follows from standard results, the example in Subsection \ref{hpbohp} does not.
	\subsubsection{Normal and commensurated subgroups}\label{subsec-normal-commensurated} For these examples and terminologies, we refer reader to \cite[Section 5]{NirMaMj-commen} and \cite[Example 1.8]{pranab-mahan} (see also \cite[Subsection $3.3.1$]{ps-krishna}). 
	
	Recall that a subgroup $H<G$ is called {\em commensurated in $G$} if for all $g\in G$, $gHg^{-1}\cap H$ has finite index in both $gHg^{-1}$ and $H$.
	
	{\em Notation}: {\em For a finitely generated group $K$, we denote a finite generating set for $K$ by $S_K$ and the Cayley graph of $K$ with respect to $S_K$ by $\Gamma_K(S_K)$.}
	
	Let $G$ be a group, and $H$ be a commensurated subgroup of $G$.
	Then by \cite[Proposition 5.12]{NirMaMj-commen} (see also \cite[Proposition 3.14]{margolis-almostnormal}), there are finite generating sets $S_H\sse S_G$ of $H$ and $G$ respectively such that $S_G\cap H=S_H$ and we have a metric graph bundle $\pi:\Gamma_G(S_G)\map\G(G,H,S)$ over the Cayley--Abels graph of the pair $(H,G)$ where fibers are isometric copies of $\Gamma_H(S_H)$.
	
	\noindent A specific example of a commensurated subgroup is a normal subgroup. See also \cite[Example 1.8]{pranab-mahan}.
	
	\noindent Now we assume that $G$ is nonelementary hyperbolic group and $H$ is nonelementary hyperbolic commensurated subgroup of $G$. 
	
	\noindent One can apply Theorem \ref{thm-application} $(A)$ to these situations.\smallskip
	
	Note that in the above example, if $H$ is an infinite normal subgroup of infinite index in $G$, then $H$ is virtually a free product of free groups and hyperbolic surface groups. See \cite[p. $379$]{mitra-endlam} for an explanation. The result is also true when $H$ is a commensurated subgroup, as proved in \cite[Theorem $A$]{NirMaMj-commen}. For particular examples: when the normal subgroup is a surface group, see \cite{mosher-hbh}, \cite{farb-mosher}, and also \cite{hamenst-word}; and when it is a free group, see \cite{BFH-lam} and \cite{geom-ghosh-gul}.

	%
	
Other examples come from {\em complexes of groups}, to which Theorem \ref{thm-application} $(A)$ can be applied. We refer the reader to \cite[Subsection 3.3.2]{ps-krishna} for such examples.
	

\noindent Theorem\ref{thm-application} $(A)$ applies in particularly interesting cases studied in \cite{min-regluing-surface} and \cite{mahan-pritam}, where the graphs of groups have all vertex and edge groups either surface groups \cite{min-regluing-surface} or free groups of rank at least three \cite{mahan-pritam}.\qed
	
	\begin{remark}
We note that in all the examples mentioned above, let $H$ be a subgroup -- that appears as fibers -- of the ambient group $G$. Then the limit set $H$ coincides with the boundary of $G$, i.e., $\Lambda_G(H)=\pa G$. See the proof of \cite[Corollary $6.33$]{ps-krishna} for this fact. Then by Lemma \ref{CT image is limit set} and Theorem \ref{all direction surj imply surj} we get the surjectivity of CT maps in all the examples above.
	\end{remark}
	
	
\subsubsection{Hyperbolic plane bundles over hyperbolic planes}\label{hpbohp} Now we mention another source of examples where Theorem \ref{thm-application} $(B)$ can be applied. Based on a work of Leininger and Schleimer (\cite{ls-disk}), Mj and Sardar constructed a metric bundle $\pi:X\map \mathbb H^2$ such that the base and fibers are uniformly quasiisometric to the hyperbolic plane (see \cite[Example 5.4, pp. 1701-1705]{pranab-mahan}). (One is referred to Definition \ref{defn-metric bundle} for metric bundle, otherwise, for the time being, one can think of this as metric graph bundle after discretization.) One can apply Theorem \ref{thm-application} $(B)$ to this example.
	
	We mention that the surjectivity of a CT map in this example follows from previously known results. Fix $b\in\mathbb H^2$ and a geodesic ray $\al$ in $\mathbb H^2$ starting at $b$. Let $\pi^{-1}(\al)=X_{\al}$. Then by the result of Bowditch (Theorem \ref{bowditch thm}), the CT map $\pa\pi^{-1}(b)\map\pa X_{\al}$ is surjective. Hence by Theorem \ref{all direction surj imply surj}, the CT map $\pa \pi^{-1}(b)\map\pa X$ is surjective (see \cite[Corollary 6.27]{ps-krishna}).\qed
	
	\begin{remark}
		If one can construct examples analogous to the one above (Subsection \ref{hpbohp}) in the setting of $\operatorname{Out}(F_n)$ and Outer space, then Theorem \ref{thm-application} $(A)$ can be applied to obtain surjective CT maps.
	\end{remark}

	\subsubsection{Examples coming from (combinatorial) horoball}
	We conclude this subsection with Example \ref{exp-combinatorial horoball}, which illustrates that the following conditions:
	
	\begin{itemize}
		\item the barycenter maps for the fibers are uniformly coarsely surjective, and
		
		\item the fibers are uniformly quasiisometric to a fixed hyperbolic space
	\end{itemize}
	
	\noindent as in Theorem \ref{thm-combo1} $(B)$ and Question \ref{main qsn} respectively, are not necessary for surjectivity of the CT map. For the example, we first need the following result (Lemma \ref{lem-uniform hyp}) which in turn relies on Proposition \ref{prop-bowditch's version} stated below. 
	
	\begin{prop}\textup{(\cite[Proposition $3.1$]{bowditch-com})}\label{prop-bowditch's version}
		Let $D\ge0$. Suppose $X$ is a (connected) metric graph, and that for each $x,y\in V(X)$, we have associated a connected subgraph $c(x,y)\sse X$ with $x,y\in c(x,y)$ such that:
		
		\begin{enumerate}
			\item For all $x,y,z\in V(X)$, $c(x,y)\sse N_D(c(y,z)\cup c(z,x))$.
			
			\item For all $x,y\in V(X)$ with $d_X(x,y)\le1$, the diameter of $c(x,y)$ is bounded above by $D$.
		\end{enumerate}
		Then $X$ is $\dl$-hyperbolic for some $\dl\ge0$ depending only on $D$. 
	\end{prop}

	\begin{lemma}\label{lem-uniform hyp}
		Let $i\in\N$. Let $F$ be a $\dl$-hyperbolic metric graph for some $\dl\ge0$.  Suppose $F_i$ is a metric graph obtained by attaching some extra edges to $F$ as follows. For all $u,v\in V(F)$, we attach an edge between $u$ and $v$ if $1<d_F(u,v)\le2^i$. Then $F_i$ is $\dl'$-hyperbolic for some $\dl'\ge0$ depending only on $\dl$. 
	\end{lemma}
	
	\begin{proof}
		We will use Bowditch's criterion as in Proposition \ref{prop-bowditch's version}. Let $x,y\in V(F_i)=V(F)$. Let $x=a_0,a_1,\dots,a_l=y$ be vertices on a geodesic $[x,y]_F$ such that $$d_F(a_{k-1},a_k)=2^i \text{ for all }k\in\{1,\dots,l-2\} \text{ and }d_F(a_{l-1},a_l)\le 2^i.$$ Define $c(x,y)=[a_0,a_1]_{F_i}*[a_1,a_2]_{F_i}*\dots*[a_{l-1},a_l]_{F_i}$ as the concatenation of edges in $F_i$. Given $x,y\in V(F_i)$, we fix once and for all a connected path $c(x,y)\sse F_i$ as defined above.
		
		Now we will show that this collection of paths satisfies the Bowditch's criterion. Let $x,y,z\in V(F_i)$. So (without loss of generality) we have vertices $x=a_0,a_1,\dots,a_l=y$, $y=b_0,b_1,\dots,b_m=z$ and $z=c_0,c_1,\dots,c_n=x$ on geodesics $[x,y]_F$, $[y,z]_F$ and $[z,x]_F$ respectively satisfying the following.
		
		\begin{enumerate}
			\item $d_F(a_{k-1},a_k)=2^i$ for all $k\in\{1,\dots,l-2\}$ and $d_F(a_{l-1},a_l)\le 2^i$.
			
			\item $d_F(b_{k-1},b_k)=2^i$ for all $k\in\{1,\dots,m-2\}$ and $d_F(b_{m-1},b_m)\le 2^i$.
			
			\item $d_F(c_{k-1},c_k)=2^i$ for all $k\in\{1,\dots,n-2\}$ and $d_F(c_{n-1},c_n)\le 2^i$.
		\end{enumerate}
		\noindent We also have $c(x,y)$, $c(y,z)$ and $c(z,x)$ as concatenation of edges in $F_i$ as defined above. We will show that $c(x,y)\sse N_{2+\dl}(c(z,x),c(y,z))$ in $F_i$. Let $q\in c(x,y)$ and $p\in V(c(x,y))$ such that $d_{F_i}(p,q)\le1$. Since $F$ is hyperbolic, without loss of generality, we assume that there is $p'\in V([y,z]_F)$ such that $d_F(p,p')\le\dl$. By the definition of $F_i$, we have $k\in\{0,1,\dots,m\}$ such that $d_{F_i}(p',b_k)=1$. This shows, by triangle inequality, that $d_{F_i}(q,b_k)\le 2+\dl$. In other words, $c(x,y)\sse N_{2+\dl}(c(z,x),c(y,z))$ in $F_i$. This verifies the condition $(1)$ of Proposition \ref{prop-bowditch's version} with $D=2+\dl$. On the other hand, condition $(2)$ is satisfied with $D=1$. Therefore, by Proposition \ref{prop-bowditch's version}, $F_i$ is $\dl'$-hyperbolic for some $\dl'\ge0$ depending only on $\dl$.
	\end{proof}
	
	For the construction of {\em combinatorial horoball} and related results, we refer the reader to \cite[Section $3.1$]{dehn-filling-GM} (cf. \cite[$8.6$]{gromov-hypgps}).\smallskip
	
\noindent	{\bf Combinatorial horoball}: Suppose $F$ is a proper, connected metric graph. Now we construct a metric graph $X$ as follows.
	
	\begin{enumerate}
		\item $V(X)=V(F)\times (\{0\}\cup\N)$
		
		\item Edges of $X$ are defined as follows. Let $u,v\in V(F)$ and $i\in\{0\}\cup\N$.
		
		\begin{enumerate}
			\item If $u,v$ are joined by an edge in $F$, then $(u,0)$ and $(v,0)$ are joined by an edge in $X$.
			
			\item If $i>0$ and $0<d_F(u,v)\le 2^i$ then $(u,i)$ and $(v,i)$ are joined by an edge in $X$. 
			
			\item $(u,i)$ and $(u,i+1)$ are joined by an edge in $X$.
		\end{enumerate}
	\end{enumerate} 
	Since the edges of $X$ are isometric to the closed unit interval, so $X$ is a geodesic metric space. By \cite[Theorem $3.8$]{dehn-filling-GM}, $X$ is $\dl$-hyperbolic for some universal constant $\dl\ge0$. That is, $\dl$ does not depend on $F$. Note that ($X$ is proper and) $\pa X$ consists of single point (see \cite[Lemma $3.11$]{dehn-filling-GM}).
	
	\noindent We have a natural projection map $\pi:X\map[0,\infty)$ as follows. Let $u,v \in V(F)$ and $i\in\{0\}\cup\N$. Edges joining $(u,i)$ and $(v,i)$ are mapped to $i$ and edges joining $(u,i)$ and $(u,i+1)$ are isometrically mapped to the interval $[i,i+1]\sse[0,\infty)$. It is easy to see that $\pi:X\map[0,\infty)$ is a metric graph bundle. $($Condition $(2)$ of Definition \ref{defn-metric graph bundle} follows from $(c)$ above whereas condition $(1)$ follows from \textup{\cite[Lemma $3.10$]{dehn-filling-GM}} and the definition of $X.)$ 
	Note that $\pi^{-1}(i)=F_i$ as constructed in Lemma \ref{lem-uniform hyp}.
	
	\begin{example}\label{exp-combinatorial horoball}
		In the above construction of combinatorial horoball, we assume that $F$ is an unbounded $\dl'$-hyperbolic geodesic metric graph for some $\dl'\ge0$. 
		
		In this case, the fibers are uniformly hyperbolic by Lemma \ref{lem-uniform hyp}, 
		and the inclusion $F_0\map X$ admits a CT map by \textup{\cite{mitra-trees}}. Hence $\pa i_{F_0,X}:\pa F_0\map\pa X$ is surjective as $\pa X$ is singleton.
		
		If the barycenter map for $F$ were not coarsely surjective, then the metric graph bundle $\pi:X\map[0,\infty)$ would not have uniformly coarsely surjective barycenter maps for its fibers.
		
		Note that the fibers do not have uniformly bounded valence, even if the original graph $F$ does have bounded valence. $($The valence of a vertex in $F_i$ increases as $i$ increases.$)$
	\end{example}
	
	\begin{remark}\label{rmk-fibers are not uniformly QI}
		We mention that in Example \ref{exp-combinatorial horoball}, the fibers $F_i$ are uniformly hyperbolic (Lemma \ref{lem-uniform hyp}); however, they are not uniformly quasiisometric to $F$, since a geodesic segment in $F$ of length $2^i$ is mapped to a single edge in $F_i$.
	\end{remark}
	
	\begin{remark}\label{rmk-Hn can not be fiber}
	We refer the reader to Definition \ref{defn-metric bundle} for metric bundle. In Example \ref{exp-combinatorial horoball}, fiber $F$ can be any hyperbolic metric graph (space), and the base of the metric graph bundle is $[0,\infty)$. By Example \ref{exp-combinatorial horoball}, one should not be confused with the result proved in \cite[p. $93$]{bowditch-stacks}. The result says that $\mathbb H^n$, $n\ge3$ can not appear as fibers of a metric bundle over $\R$ such that the total space is hyperbolic. Here author means that the fibers are isometric to $\mathbb H^n$. (We mention that the same result holds if one replaces the isometry by uniform quasiisometry.) Whereas, in Example \ref{exp-combinatorial horoball}, the fibers are not uniformly quasiisometric to a fixed hyperbolic space by Remark \ref{rmk-fibers are not uniformly QI}.
	\end{remark}
	
	\subsection{Application}\label{subsec-CT lamination} Suppose $W$ is a hyperbolic metric space. Let $\pa^2W:=\{(p,q)\in\pa_s W\times\pa_s W:p\ne q\}$. Following \cite{mitra-endlam}, one can define Cannon--Thurston lamination as follows.
	
	\begin{defn}[Cannon--Thurston lamination]
		Suppose $X'\sse X$ are hyperbolic metric spaces such that the inclusion $i_{X',X}:X'\map X$ admits a CT map $\pa i_{X',X}:\pa_s X'\map\pa_s X$. The Cannon--Thurston lamination is then defined as $$\L_{CT}(X',X):=\{(p,q)\in\pa^2X':\pa i_{X',X}(p)=\pa i_{X',X}(q)\}.$$
	\end{defn}
	
	\noindent\underline{{\em Our context}}: Suppose $\pi:X\map B$ is an $f$-metric graph bundle with controlled hyperbolic fibers. Further, suppose that $X$ is hyperbolic. Note that $B$ is also hyperbolic (see Theorem \ref{thm-metric bundle comb} $(1)$). Suppose $B$ is $\dl$-hyperbolic for some $\dl\ge0$. Let $b\in V(B)$ and $F=F_b$, and $\eta\in\pa_sB$. Let $\al:[0,\infty)\map B$ be a $k$-quasigeodesic joining $b$ and $\eta$ for some $k\ge1$ depending on $\dl$ (see Subsection \ref{subsec-gromov bdry and CT}). We also assume that $\al$ is injective and continuous (see \cite[Lemma $1.11$, $III.H$]{bridson-haefliger}). Then $\al$ has a $k$-qi lift in $X$. Define the set $$\pa^{\eta}X:=\{\al'(\infty):\al'\text{ is a }k\text{-qi lift of }\al\}$$ which is determined by $\eta$ (see \cite[Corollary $6.5$]{ps-krishna}). Now define $$\pa^{(2)}_{\eta,X}(F):=\{(p,q)\in\pa^2F:\pa i_{F,X}(p)=\pa i_{F,X}(q)\in\pa^{\eta}X\}.$$
	
We refer the reader to \cite[Subsection $6.2$]{ps-krishna} (see also \cite{mitra-endlam}) for more on the properties of the Cannon--Thurston lamination in the context of metric graph bundles. The following result provides a criterion for which $\pa^{(2)}_{\eta,X}(F)$ is non-empty. Recall that {\em dendrite} is a compact metric space in which any two distinct points are connected by a unique arc.
	
	\begin{theorem}\label{thm-lamination}
Suppose $\pi:X\map B$ is an $f$-metric graph bundle with controlled hyperbolic fibers. Further, suppose that $X$ is hyperbolic. Moreover, we assume one of the following.

$(A)$ The fibers have uniformly bounded valence.
			
$(B)$ The fibers are one-ended and proper metric spaces.

Finally, suppose that one of the following holds.

\begin{enumerate}
\item Let $b\in V(B)$. Suppose that $\pa F_b$ is not homeomorphic to a dendrite.

\item  Fibers are uniformly quasiisometric to a fixed nonelementary hyperbolic group.		
\end{enumerate}		
		
		Then for all $\eta\in\pa_s B$, we have $\pa^{(2)}_{\eta,X}(F)\ne\emptyset$.
	\end{theorem}
	
Note that the boundary $\partial G$ of an infinite hyperbolic group $G$ cannot be homeomorphic to a dendrite. Indeed, if $\partial G$ were a (nontrivial) dendrite (\cite[Section 3, H]{gromov-hypgps}), then it would in particular be connected. By \cite{swarup-cutpoint}, $\partial G$ has no global cut points. On the other hand, every nontrivial dendrite contains cut points. This contradiction shows that $\partial G$ cannot be a dendrite.

\noindent  Finally, Theorem \ref{thm-lamination} follows by combining the following result (Theorem \ref{thm-lam-KP}) due to Krishna and Sardar with the above fact, along with Theorem \ref{thm-application}. We mention that in both Theorems \ref{thm-lamination} and \ref{thm-lam-KP}, the boundary of fibers are non-trivial (see Theorem \ref{thm-emb of T3})

	\begin{theorem}\textup{(\cite[Theorem $6.30$]{ps-krishna})}\label{thm-lam-KP}
		Suppose $\pi:X\map B$ is an $f$-metric graph bundle with controlled hyperbolic fibers. Further, suppose that the fibers are proper metric spaces. Let $b\in V(B)$ and $F=F_b$. Moreover, assume that $\pa F$ is not homeomorphic to a dendrite. Finally, suppose that $X$ is hyperbolic and the CT map $\pa i_{F,X}:\pa F\map\pa X$ is surjective.
		
		Then for all $\eta\in\pa_s B$, we have $\pa^{(2)}_{\eta,X}(F)\ne\emptyset$.
	\end{theorem}
	\section{Appendix}\label{appendix sec}
	This appendix aims to explore our results for metric bundles (see Definition \ref{defn-metric bundle}) via their approximating metric graph bundles, as described in \cite[Section 1]{pranab-mahan}.\smallskip
	
\noindent{\bf Approximating graphs of spaces}: 
	For this discussion, we refer the reader to \cite[Proposition 8.45, I.8]{bridson-haefliger}. Suppose $X$ is a geodesic metric space, and let $X' \subseteq X$ be a maximal (with respect to inclusion) subset such that $d_X(x, y) \ge 1$ for all distinct $x, y \in X'$. (The existence of such a subset follows from Zorn's Lemma.) Note that for every $z \in X$, there exists $x \in X'$ such that $d_X(x, z) \le 1$.
	
	Now we construct a graph $\Gamma$ whose vertex set is $X'$, and where two distinct points $x, y \in X'$ are joined by an edge if $d_X(x, y) \le 3$. Then $\Gamma$ is a connected metric graph. Note that the vertices of $\Gamma$ correspond to points of $X' \subseteq X$. We define a map $\phi: \Gamma \to X$ by sending each vertex of $\Gamma$ to the corresponding point in $X' \subseteq X$, and sending the interior of each edge to the image of one of its endpoints. With this we have the following.
	
	\begin{prop}\textup{(\cite[Proposition 8.45, I.8]{bridson-haefliger})}\label{prop-approx graph}
		There is a universal constant $E\ge1$ such that $\phi:\Gamma\map X$ is an $E$-quasiisometry.
	\end{prop}
	
	
One may compare the following definition with that in \cite{roe-asym}, where the author refers to such spaces as having {\em bounded growth} and shows that they have finite asymptotic dimension. In our setting, however, we require the space to be `uniformly proper', meaning that the properness condition holds uniformly across the entire space, regardless of the base point. This motivates to the following.
	
	\begin{defn}\label{strongly proper}
		A metric space $X$ is said to be strongly proper if we have a function $N:\R_{>0}\times\R_{>0}\map\N$ such that any ball of radius $R$ in $X$ can be covered by $N(R,r)$-many balls of radius $r$. Sometimes, we call the function $N:\R_{>0}\times\R_{>0}\map\N$ as the parameter of strongly proper.
		
		We say that a collection of metric spaces $\{X_{\al}:\al\in\Lambda\}$ is uniformly strongly proper if there is a function $N:\R_{>0}\times\R_{>0}\map\R_{>0}$ such that for all $\al\in\Lambda$, the space $X_{\al}$ is strongly proper with the parameter function $N$.
	\end{defn}
	
Simple examples of strongly proper spaces are $\R^n$, $\mathbb H^n$, and metric graph in which the valence at each vertex is uniformly bounded (i.e., a metric graph with bounded valence, see Definition \ref{defn-valence}). In particular, a Cayley graph of a group with respect to a finite generating set is strongly proper. In contrast, a metric graph where each vertex has finite but not uniformly bounded valence is proper but not strongly proper.
	
	\begin{lemma}\label{lem-st pro imp st pro}
Suppose $X$ is a proper metric space. Let $\Gamma$ be any approximating metric graph of $X$ as in Proposition \ref{prop-approx graph}. Then $\Gamma$ is a proper metric space.

Moreover, if $X$ is a strongly proper metric space with a parameter function $N:\R_{>0}\times\R_{>0}\map\N$, then the valence at each vertex of $\Gamma$ is bounded above by $N(4,1/3)$. In particular, $\Gamma$ is a strongly proper metric space.
	\end{lemma}
	
	\begin{proof}
First, we prove that $\Gamma$ is proper assuming $X$ is a proper metric space. For convenience, we slightly abuse notation and assume that $V(\Gamma) \subseteq X$. Let $B_X(x, r)$ denote the closed ball of radius $r \in \mathbb{R}_{\ge 0}$ centered at $x$ in $X$. By the assumption, $B_X(x,r)$ is compact for all $r\in\R_{\ge0}$ and $x\in X$.
		
Suppose, for contradiction, that $\Gamma$ is not proper. Then there exists a vertex $u \in V(\Gamma)$ with infinite valence. Let $\{u_i : i \in \mathbb{N}\}$ be an infinite collection of distinct vertices in $\Gamma$ adjacent to $u$. Note that $u, u_i \in X$ and $d_X(u, u_i) \le 3$ for all $i \in \mathbb{N}$. Thus, $\{u, u_i : i \in \mathbb{N}\} \subseteq B_X(u, 3)$.
		
\noindent Since $u_i \neq u_j$ for $i \neq j$, we have $d_X(u_i, u_j) \ge 1$, implying that the balls $B_X(u_i, 1/3)$ and $B_X(u_j, 1/3)$ are disjoint for all $i \ne j$. On the other hand, $B_X(u_i, 1/3) \subseteq B_X(u, 4)$ for all $i$. Hence, the sequence $\{u_i\}$ lies entirely in the compact set $B_X(u, 4)$ and has no convergent subsequence, contradicting the compactness of $B_X(u,4)$. This concludes that $\Gamma$ is proper.
		
We now prove that $\Gamma$ is strongly proper, assuming that $X$ is strongly proper with parameter function $N:\R_{>0}\times\R_{>0}\to\N$. It suffices to show that the valence of every vertex of $\Gamma$ is bounded above by $N(4,1/3)$. Suppose not. Then there exists a vertex $u \in V(\Gamma)$ with valence strictly greater than $N(4, 1/3)$. Let $l > N(4, 1/3)$ and consider a set $\{u_i : 1 \le i \le l\}$ of distinct vertices in $\Gamma$ adjacent to $u$. As shown above, the balls $B_X(u_i, 1/3)$ are disjoint and contained in $B_X(u, 4)$. Thus, we cannot cover $B_X(u, 4)$ by $N(4, 1/3)$-many balls of radius $1/3$ centered at $\{u_i:1\le i\le l\}$.

Define $(1/3)$-area of $B_X(u,4)$ as the supremum number $P(1/3)\in\N\cup\{\infty\}$ such that $P(1/3)$-many disjoint balls of radius $1/3$ can be embedded in $B_X(u,4)$. Note that $P(1/3)\ge l>N(4,1/3)$. This concludes that $B_X(u,4)$ cannot be covered by at most $N(4,1/3)$-many balls of radius $1/3$, contradicting the assumption that $X$ is strongly proper with the parameter function $N$. This completes the proof of Lemma \ref{lem-st pro imp st pro}.
	\end{proof}
	
	\begin{defn}\textup{(\cite[Definition 1.2]{pranab-mahan})}\label{defn-metric bundle}
		Let $f:\R_{\ge0}\map\R_{\ge0}$ be a map such that $f(n)\map\infty$ as $n\map\infty$, and let $c\ge1$. Suppose $X$ and $B$ are metric spaces. A surjective $1$-Lipschitz map $\pi:X\map B$ is called $(f,c)$-metric bundle if the following hold.
		\begin{enumerate}
			\item For all $b\in B$, $F_b:=\pi^{-1}(b)$, called fiber, is a geodesic metric space with respect to the path metric $d_b$ from $X$. The inclusion maps $(F_b,d_b)\map X$ are $f$-proper embedding (see Definition \ref{defn-proper embedding}).
			
			\item Let $b_1,b_2\in B$ such that $d_B(b_1,b_2)\le 1$. For any geodesic $\al=[b_1,b_2]_B$, $z\in\al$ and $x\in F_z$, there is a path in $\pi^{-1}(\al)$ of length at most $c$ joining $x$ and a point in $F_{b_i}$, where $i=1,2$.
		\end{enumerate}
		Sometimes, we say $\pi: X \to B$ is an $(f, c)$-metric bundle without explicitly specifying the function $f$ and the constant $c$.
	\end{defn}
	Given a metric bundle, the authors in \cite[Section 1]{pranab-mahan} construct a quasiisometric metric graph bundle that satisfies the condition of `compatibility'. We refer the reader to \cite[Section 1]{pranab-mahan} for the full construction and details. Below, we provide a brief sketch of the idea.\smallskip
	
\noindent{\bf Approximating metric graph bundle}: Suppose $\pi': X' \to B'$ is an $(f, c)$-metric bundle. We first apply Proposition \ref{prop-approx graph} to the base space $B'$ to construct a metric graph $B$ and a quasiisometry $\psi: B \to B'$. For each vertex $u \in B$, we have the fiber $F'_u := \pi'^{-1}(\psi(u))$. Next, we apply Proposition \ref{prop-approx graph} to each $F'_u$ to obtain a metric graph $F_u$ along with a quasiisometry $f_u: F_u \to F'_u$.
	
	Now we construct a connected graph $X''$ with vertex set $V(X'')=\bigcup_{u\in V(B)}V(F_u)$ and we join two distinct vertices $x, y \in V(X'')$ by an edge if $d_X(x, y) \leq 6c + 3$. Define a subgraph $X \subset X''$ with the same vertex set $V(X) = V(X'')$, where an edge $[x, y]$ in $X''$ is remained in $X$ if and only if either $[x,y]\sse F_u$ for some $u\in V(B)$ or $x\in V(F_u)$ and $y\in V(F_v)$ with $d_B(u,v)=1$. Now we define a map $\Psi:X\map X'$ sending vertices of $X$ to the corresponding points of $X'$, and sending the interior of an edge to the image of one of its endpoints.
	
	Finally, we define the projection map $\pi: X \to B$ as follows. For any edge in $X$ that connects two vertices of some $F_u$ (where $u \in V(B)$), we set $\pi$ to map the entire edge to the vertex $u$. For any other edge $[x, y]$ in $X$, where $x \in F_u$ and $y \in F_v$ with $d_B(u, v) = 1$, we define $\pi$ to be an isometry from $[x, y]$ onto the edge $[u, v]$ in $B$.
	
With the construction and notations above, we have the following.
	
	\begin{prop}\textup{(\cite[Lemmas 1.20, 1.21]{pranab-mahan})}\label{prop-approx metric graph bundle}
	Suppose $\pi':X' \to B'$ is an $(f,c)$-metric bundle. Then there exists a constant $K \ge 1$ and a function $g:\N \to \N$ with $g(n) \to \infty$ as $n \to \infty$, both depending only on $f$ and $c$, such that:
		\begin{enumerate}
			\item For all $u \in V(B)$, the maps $\psi:B \to B'$ and $f_u:F_u \to F'_u$ are $E$-quasiisometries, where $E \ge 1$ is the universal constant appearing in Proposition \ref{prop-approx graph}.
			
			\item $\Psi:X\map X'$ is a $K$-quasiisometry.
			
			\item $\pi:X\map B$ is a $g$-metric graph bundle.
		\end{enumerate}
	\end{prop}
	
	We refer to $\pi:X\map B$ in Proposition \ref{prop-approx metric graph bundle} as an {\em approximating metric graph
		bundle} of the metric bundle $\pi':X'\map B'$.
	
	Finally, the discussion also gives the following (coarsely) commutating diagram. 
	\begin{figure}[h]
		\centering
		\begin{subfigure}{0.45\linewidth}  
			\centering
			\begin{tikzcd}
				X \arrow{r}{\Psi} \arrow[swap]{d}{\pi} \arrow[dr, phantom] & X' \arrow{d}{\pi'} \\
				B \arrow{r}{\psi}& B'
			\end{tikzcd}
			\caption{}
		\end{subfigure}
		\quad  
		\begin{subfigure}{0.45\linewidth}
			\centering
			\begin{tikzcd}
				F_u \arrow{r}{f_u~(QI)} \arrow[swap]{d}{} \arrow[dr, phantom] & F'_u\arrow{d}{} \\
				X\arrow{r}{\Psi~(QI)} & X'
			\end{tikzcd}
			\caption{}
		\end{subfigure}
		
		\caption{}
		\label{commutative diagram}  
	\end{figure}
	
	As a consequence of Proposition \ref{prop-approx metric graph bundle} $(1)$ and Lemma \ref{lem-barycenter commuting}, we have the following. We retain the same definition of a metric bundle having {\em controlled hyperbolic fibers} as given in Definition \ref{controlled fibers defn}.
	\begin{cor}\label{cor-approx controlled}
		Suppose $\pi':X'\map B'$ is an $(f,c)$-metric bundle with controlled hyperbolic fibers (that is, the fibers are uniformly hyperbolic and the barycenter maps for the fibers are uniformly coarsely surjective). Then an approximating $g$-metric graph bundle $\pi:X\map B$ (see Proposition \ref{prop-approx metric graph bundle}) also has controlled hyperbolic fibers.
	\end{cor}
	
	

	
	\begin{remark}\label{rmk-formetircbundle}
		We note that $(1)$ Theorem \ref{all direction surj imply surj} holds in the setting of metric bundles as well, and $(2)$ Theorem \ref{thm-metric bundle comb}, along with Remark \ref{crucial remark}, also holds for metric bundles. The cited references there work for metric bundles. 
	\end{remark}
	Finally, we have the following main theorem in the appendix.
	\begin{theorem}\label{thm-all in one appendix}
Suppose $\pi':X'\map B'$ is an $(f,c)$-metric bundle with controlled hyperbolic fibers, and suppose that $X'$ is hyperbolic. $($Note that $B'$ is hyperbolic.$)$ Suppose $A'$ is a qi embedded subspace of $B'$, and let $Y'=\pi'^{-1}(A')$. $($Note also that $Y'$ is hyperbolic$)$. Finally, we assume one of the following.

$(A)$ Suppose $N:\R_{>0}\times\R_{>0}\map\N$ is a map such that the fibers are strongly proper metric spaces with the parameter function $N$ (i.e., the fibers are uniformly strongly proper).

$(B)$ The fibers are one-ended and proper metric spaces.

 Then:
		
		\begin{enumerate}
			\item the inclusion $i_{Y',X'}:Y'\map X'$ admits a CT map $\pa i_{Y',X'}:\pa_sY'\map\pa_sX'$ by \textup{\cite[Theorem $5.2$]{ps-krishna}}, and
			
			\item the CT map $\pa i_{Y',X'}:\pa_s Y'\map\pa_s X'$  is surjective.
		\end{enumerate}
	\end{theorem}
	
	\begin{proof}
		Fix $b\in B'$ and the fiber $F'_b=\pi'^{-1}(b)$. First we will show that the CT map $\pa i_{F'_b,X'}:\pa F'_b\map\pa X'$ is surjective. (Existence follows from Theorem \ref{all direction surj imply surj} $(1)$ and Remark \ref{rmk-formetircbundle} $(1)$.) 
		
		Suppose $\pi:X\map B$ is an approximating $g$-metric graph bundle of the metric bundle $\pi':X'\map B'$ for some function $g:\N\map\N$. For convenience, we slightly abuse notation and assume that $V(B)\sse B'$. Moreover, we assume that $b\in V(B)$, and so we have fiber $F_b=\pi^{-1}(b)$. By Proposition \ref{prop-approx metric graph bundle} $(2)$ and Corollary \ref{cor-approx controlled}, the $g$-metric graph bundle $\pi:X\map B$ has controlled hyperbolic fibers such that $X$ is hyperbolic.\smallskip
		
\noindent{\em Under the condition $(A)$}: Let $u\in V(B)$ and $F_u=\pi^{-1}(u)$. Let $x\in V(F_u)$. Then by Lemma \ref{lem-st pro imp st pro}, the valence at $x$ in $F_u$ is bounded above by $N(4,1/3)$. This is true for any vertices $u\in V(B)$ and $x\in V(F_u)$.
		
\noindent{\em Under the condition $(B)$}: Since the property of being one-ended is invariant under quasiisometry, fibers of the $g$-metric graph bundle $\pi:X\map B$ are one-ended. Moreover, by Lemma \ref{lem-st pro imp st pro}, the fibers are proper metric graphs.\smallskip

By Theorem \ref{thm-application}, the CT map $\pa i_{F_b,X}:\pa F_b\map \pa_sX$ is surjective. Hence, from the (coarse) commutative diagram Figure \ref{commutative diagram} $(B)$ (and Lemma \ref{CT image is limit set}), the CT map $\pa i_{F'_b,X'}:\pa F'_b\map\pa_s X'$ is surjective.

Note that $B'$ is hyperbolic by \textup{\cite[Proposition $2.12$]{pranab-mahan}}, and so $Y'$ is hyperbolic by Remarks \ref{crucial remark} $(2)$ and \ref{rmk-formetircbundle} $(2)$. Finally, by Theorem \ref{all direction surj imply surj} $(2)$ and Remark \ref{rmk-formetircbundle} $(1)$, the CT map $\pa i_{Y',X'}:\pa_s Y'\map \pa_s X'$ is surjective.
	\end{proof}

	\bibliography{Ubib}
	\bibliographystyle{amsalpha}
\end{document}